\documentclass[12pt]{article}
\usepackage[figuresright]{rotating}
\usepackage{amsmath,amssymb,graphicx}
\RequirePackage[numbers]{natbib}
\RequirePackage[colorlinks,citecolor=blue,urlcolor=blue]{hyperref}

\newcommand{\comb}[2]{\left(\begin{array}{c} #1 \\ #2 \end{array}\right)}
\newcommand{\trp}{^{\tt t}}
\newcommand{\plim}{{\rm plim}}
\newcommand{\condist}{\stackrel{d}{\longrightarrow}}
\newcommand{\conprob}{\stackrel{p}{\longrightarrow}}
\newcommand{\conuprob}{\stackrel{up}{\longrightarrow}}

\newtheorem{prop}{Proposition}

\newtheorem{theo}{Theorem}
\newtheorem{lemm}{Lemma}
\newcommand{\email}[1]{{\em E-Mail:} {#1}}
\newenvironment{keywords}{\noindent {\em Key Words and Phrases:}}{}
\newenvironment{proof}{\medskip {\noindent {\bf Proof:}}}{$\parallel$ \medskip}

\begin{document}

\title{Model Selection and Estimation with
Quantal-Response Data in Benchmark Risk Assessment}

\author{Edsel A.\ Pe\~na\footnote{Department of Statistics, University of South Carolina, Columbia, SC 29208 USA.  \email{pena@stat.sc.edu}}
\and
Wensong Wu\footnote{Department of Mathematics and Statistics, Florida International University, Miami, FL 33199 USA. \email{wenswu@fiu.edu}}
\and
Walter Piegorsch\footnote{Program in Statistics and BIO5 Institute,
University of Arizona, Tucson, AZ 85721 USA.  \email{piegorsch@math.arizona.edu}}
\and
Ronald W.\ West\footnote{Department of Statistics, North Carolina State University, Raleigh, NC 27695 USA. \email{websterwest@gmail.com}}
\and
LingLing An\footnote{Program of Statistics, Department of Agricultural and Biosystems Engineering,
University of Arizona, Tucson, AZ 85721 USA.  \email{anling@email.arizona.edu}}
}







\maketitle

\begin{abstract}
This paper describes several approaches for estimating the benchmark dose (BMD) 
in a risk assessment study with quantal dose-response data
and when there are competing model classes
for the dose-response function. Strategies involving a two-step
approach, a model-averaging approach, a focused-inference approach, and a
nonparametric approach based on a PAVA-based estimator of the dose-response function are described and compared.
Attention is raised to the perils involved in data ``double-dipping'' and
the need to adjust for the model-selection stage in the estimation procedure.
Simulation results are presented comparing the performance of five model
selectors and eight BMD estimators.
An illustration using a real quantal-response data set from a carcinogenecity study is provided.
\end{abstract}

%
%


\section{Introduction and Motivation}
\label{sect-Introduction}

The traditional approach to statistical
inference assumes that a random entity $\mathbf{Y}$,
taking values in a sample space  $\mathcal{Y}$, is observable.
Such a $\mathbf{Y}$ represents the outcome of an experiment, a study, or a survey.
The (joint) distribution function
$F$ of $\mathbf{Y}$ is assumed to belong to a specified model class
$\mathcal{F}$ of distribution functions
on $\mathcal{Y}$. The class $\mathcal{F}$ may be parametrically
or nonparametrically specified.
For example, in quantal-response risk assessment studies of exposures to
hazardous agents, the primary focus of this paper, there will
be pre-specified doses $0 = d_1 < d_2 < \ldots < d_J < \infty$, and for each dose
$d_j$, there is an observable random variable $Y_j$,
which is binomially distributed with parameters
$N_j$ and $\pi(d_j)$. Here, $N_j$ represents the number of units placed
on test at dose $d_j$, $\pi(\cdot)$ is the dose-response function, $\pi(d_j)$ is the probability of a
unit at dose $d_j$ exhibiting the
adverse event of interest, and $Y_j$ is the
total out of the $N_j$ units that exhibit
the adverse event \cite[Ch.~4]{PieBai05}.
Commonly, we assume
that the $N_j$s are known, the $Y_j$s are independent, and
$\pi(\cdot)$ belongs to some model class $\mathcal{M}$. An example
of a model class in this setting is
\begin{displaymath}
\mathcal{M}_1 = \{\pi_1(d;\theta_0,\theta_1) = 1 - \exp\{-\theta_0 -\theta_1 d\}: \theta_0 \in \Re_+,
\theta_1 \in \Re_+\}.
\end{displaymath}
This
is the linear complementary log model, also known as the quantal-linear
model or the one-stage
model
in carcinogenesis testing \citep{BucPieWes09}.
Typically, statistical attention will focus on making inferences about
the unknown parameters, e.g., constructing a $100(1-\alpha)\%$ confidence
interval on $\theta_1$.  In risk assessment studies, however, the function
$\pi(d_j)$ is specifically used to model the risk of exhibiting an adverse
response or reaction at dose $d_j$.  Attention is then directed at using
information in $\pi(\cdot)$ to estimate risk at low doses.  By inverting
the estimated dose-response relationship, the analyst can calculate the dose
level at which a predetermined benchmark response (BMR) for the adverse response
is attained.  The corresponding \textit{Benchmark Dose} ($\mbox{\emph{BMD}}$), is an
important quantity in deriving regulatory limits for modern risk
management \cite[\S~4.3]{PieBai05}. $\mbox{\emph{BMD}}$s are employed increasingly in
quantitative risk analyses for setting acceptable levels of human exposure
or to establish modern low-exposure guidelines for hazardous
environmental or chemical agents \citep{SanVicFal08}.

If $\mathbf{Y} \equiv (Y_j,j=1,2,
\ldots,J)$ is the observable vector from such a study, then its joint probability
mass function $p_{\mathbf{Y}}$, which determines $F$, is
\begin{displaymath}
p_{\mathbf{Y}}(y_1,\ldots,y_J;\theta_0,\theta_1) =
\prod_{j=1}^J \comb{N_j}{y_j} \pi_1(d_j;\theta_0,\theta_1)^{y_j}
[1 - \pi_1(d_j;\theta_0,\theta_1)]^{N_j-y_j}
\end{displaymath}
for $y_j \in \{0,1,2,\ldots,N_j\}$.
In this conventional framework with one model class,
methods of inference, e.g., estimation, hypothesis testing,
interval estimation, or prediction, are well-developed, relying on
the maximum likelihood (ML) principle, the Neyman-Pearson
hypothesis testing framework, or the Bayesian paradigm
\citep{DeuGreHab10,Mor92,MR1718844,MR752447}.

Recent years, however, have seen greater appreciation for settings
with more than one model class for $F$ \citep{BurAnd02,ClaHjo08}. Such
situations arise in a variety of scientific settings, including engineering,
reliability, economics, and in particular, in the risk assessment problem emphasized herein \citep{BaiNobWhe05, SanFalVic02}.
An impetus for considering several model classes is the desire for more
inferential robustness without becoming fully nonparametric. For example, in the
quantal-response setting,
the dose-response function $\pi(\cdot)$, aside from being
possibly in the model class $\mathcal{M}_1$, may alternatively belong to the model class
\begin{displaymath}
\mathcal{M}_2 = \{\pi_2(d;\eta_0,\eta_1) = {[1 + \exp\{-(\eta_0 + \eta_1 d)\}]}^{-1}: \eta_0 \in \Re, \eta_1 \in \Re\},
\end{displaymath}

In settings with multiple model classes, a seemingly natural approach is to use
the data to first select the model class, and then use the same data {\em{again}}
to perform inference in the chosen model class. However, caution needs to be exercised
since, when not properly adjusting for such data ``double-dipping,''
detrimental consequences, such as underestimation of standard
errors, loss of control of Type I error probabilities,
or nonfulfillment of coverage probabilities, ensue \cite[\S~7.4]{ClaHjo08}.
It is of importance to examine issues pertaining to these
statistical problems when operating
with several possible model classes and to develop
appropriate statistical procedures that properly adjust for data re-use.
This paper is targeted for this purpose, with particular emphasis on
quantal dose-response modeling and its use in estimating the
benchmark dose for low-dose risk assessment.

\section{Mathematical Underpinnings}
\label{section-Background}

We describe in this section the mathematical framework  and
formally state the  problems of interest.
Consider a  quantal-response study where the collection
$(\mathbf{N},\mathbf{d}) = \{(N_j,d_j), j=1,2,\ldots,J\}$
is given and the random observables
are $\mathbf{Y} = (Y_j, j=1,2,\ldots,J),$ where
\begin{equation}
\label{distribution of Yj} Y_j | (N_j,d_j) \sim B(N_j,\pi(d_j)),
j=1,2,\ldots,J.
\end{equation}
%
Here $0 = d_1 < d_2 < \ldots < d_J$ are the doses,
$\pi(\cdot)$ is the dose-response function, and $B(n,\pi)$ is a
binomial distribution with parameters $(n,\pi)$. The entire data ensemble
will be denoted by $(\mathbf{Y},\mathbf{N},\mathbf{d}) =
\{(Y_j,N_j,d_j), j=1,2,\ldots,J\}$.

In most risk-analytic studies, the differential risk adjusted for
any spontaneous or background effect is typically of interest.  This
leads to consideration of risks in excess of the background.  Quantifying
this, suppose the dose-response function is $\pi: \Re \rightarrow [0,1]$.
Then, the \textit{extra risk function}, which is relative to the background
risk, is
\begin{equation}
\label{extra risk} \pi_E(d) = \frac{\pi(d) - \pi(0)}{1 - \pi(0)}.
\end{equation}
Typically it is assumed that the
mapping $d \mapsto \pi(d)$ is monotone increasing, hence
$d \mapsto \pi_E(d)$ is also monotone increasing.
Given a BMR value $q \in (0,1)$, the BMD at this risk level $q$, denoted $\mbox{\emph{BMD}}(q)$, is the dose $d \ge 0$ satisfying
%
$\pi_E(d) = q.$
%
For brevity, instead of writing $\mbox{\emph{BMD}}(q)$, we
instead use the notation $\xi_q$, so
\begin{equation}
\label{BMD}
\xi_q \equiv BMD(q) = \pi_E^{-1}(q)
\end{equation}
where $\pi_E^{-1}(\cdot)$ is the inverse function of $\pi_E(\cdot)$.
Observe that $\xi_q$ is determined by the dose-response function and
its parameters.
In this paper we will mainly be concerned with obtaining estimators of $\xi_q$
and their properties. 

We describe the mathematical set-up of interest.
Our underlying assumption is that the unknown dose-response function
$\pi(\cdot)$ is a member of the collection
\begin{equation}
\label{big model} \mathcal{M} = \mathcal{M}_0 \cup \left( \cup_{m=1}^M
\cup_{l=1}^{L_m} \mathcal{M}_{ml} \right),
\end{equation}
where $\mathcal{M}_0$ and $\mathcal{M}_{ml}, l=1,2,\ldots,L_m; m =
1,2,\ldots,M$, are model classes of dose-response functions. We
 assume that these model classes satisfy the following
conditions:
\begin{itemize}
\item[] {\bf (C1)} \label{condition 1}
For each $m = 1,2,\ldots,M$, we have
%
$\mathcal{M}_0 \subset \mathcal{M}_{m1}
\subset \mathcal{M}_{m2} \subset \ldots \subset \mathcal{M}_{mL_m}.$
%

\item[] {\bf (C2)} \label{condition 2}
The model classes are, for $m=1,2,\ldots,M$ and $l=1,2,\ldots,L_m$, of forms
\begin{eqnarray*}
& \mathcal{M}_0 =  \{\pi_0(\cdot;\theta_0): \theta_0 \in \Theta_0\}; & \\
& \mathcal{M}_{ml}  =
\{\pi_{ml}(\cdot;\theta_0,\theta_{m1},\ldots,\theta_{ml}): \theta_0
\in \Theta_0, \theta_{mj} \in \Theta_{mj}, j=1,2,\ldots,l\}, &
\end{eqnarray*}
with $\Theta_0$ an open
subset of $\Re^{g_0}$ and $\Theta_{mj}$ an open subset of
$\Re^{g_{mj}}$, and with the dimensions $g_0$ and $g_{mj}$s being
known.

\item[] {\bf (C3)} \label{condition 3}
For each $ m = 1,2,\ldots,M$ and $l=1,2,\ldots,L_m$, there are unique
elements $\{\theta_{ml}^0: l=1,2,\ldots,L_m; m=1,2,\ldots,M\}$ such
that $\theta_{ml}^0 \in \Theta_{ml}$ and
\begin{displaymath}
\pi_{ml-1}(\cdot;\theta_0,\theta_{m1},\theta_{m2},\ldots,\theta_{ml-1})
=
\pi_{ml}(\cdot;\theta_0,\theta_{m1},\theta_{m2},\ldots,\theta_{ml-1},\theta_{ml}^0).
\end{displaymath}

\item[] {\bf (C4)} \label{condition 4}
For each $m \ne m^\prime$, and for $l = 1,2,\ldots,L_m$ and $l^\prime = 1,2,\ldots,L_{m^\prime}$,
\begin{displaymath}
\left( \mathcal{M}_{ml} \setminus \mathcal{M}_0\right) \cap
\left( \mathcal{M}_{m^\prime l^\prime} \setminus \mathcal{M}_0 \right) = \emptyset.
\end{displaymath}
\end{itemize}

Conditions (C1-C3) imply that $\mathcal{M}_0$,
which may be empty, is the smallest model class, and for each model type
$m$, there is a nested structure among the layers
$(\mathcal{M}_{ml},l=1,2,\ldots,L_m)$.
Condition (C4) requires that the $M$ model classes may only
intersect at the smallest model class $\mathcal{M}_0$.
Pictorially, the structure and inter-relationships
among the model classes are shown in Figure \ref{figure-model class structure}.
There are several commonly-used dose-response model classes.
See, for instance, \cite{BaiNobWhe05,WheBai07,WheBai09} for some
examples.

With this mathematical framework in hand, our main objective is to
obtain estimators of $\xi_q$ based on the observable $(\mathbf{Y},\mathbf{N},\mathbf{d})$.
Apart from
estimation of $\xi_q$, it is also of interest to
determine the smallest or most parsimonious model
class containing the dose-response function $\pi(\cdot)$,
the so-called model selection problem.

\begin{figure}[h]
\caption{Depiction of the interrelationships among the model
classes. The columns represent the model class type, while the rows
are the depths within the model class type. Note that the maximal depths, $L_m$, within each
$m$th model class type may differ.} \label{figure-model class structure}
\begin{center}
\begin{tabular}{||c||c|c|c|c||} \hline\hline
Model Class Type Depth & \multicolumn{4}{c||}{Model Class Type} \\
$l$ & \multicolumn{4}{c||}{$m$} \\ \hline\hline
& \multicolumn{4}{c||}{} \\
0 & \multicolumn{4}{c||}{$\mathcal{M}_0$} \\ \cline{1-1}
& $\cap$ & $\cap$ &  & $\cap$ \\
1 & $\mathcal{M}_{11}$ & $\mathcal{M}_{21}$ & $\cdots$ & $\mathcal{M}_{M1}$ \\ \cline{1-1}
& $\cap$ & $\cap$ &  & $\cap$ \\
2 & $\mathcal{M}_{12}$ & $\mathcal{M}_{22}$ & $\cdots$ & $\mathcal{M}_{M2}$ \\ \cline{1-1}
& $\cap$ & $\cap$ & & $\cap$ \\
$\vdots$ & $\vdots$ & $\vdots$ & $\vdots$ & $\vdots$ \\ \cline{1-1}
& $\cap$ & $\cap$ & & $\cap$ \\
$L_m$ & $\mathcal{M}_{1L_1}$ & $\mathcal{M}_{2L_2}$ & $\cdots$ & $\mathcal{M}_{ML_M}$ \\
\hline\hline
\end{tabular}
\end{center}
\end{figure}

\section{Model Class Selection}
\label{sect-Model Selection}

There are several approaches to the model selection problem. Most
apply some form of information-theoretic metric to distinguish among
competing models/model classes.  Two popular measures are Akaike's
information criterion (AIC) \citep{Aka73} and the Bayesian
information criterion (BIC) \citep{Sch78}. Both criteria are
likelihood-based.  Other forms are possible, including a
second-order adjusted AIC \citep{MR1016020}, the Focused IC
\cite[\S~6.2]{ClaHjo08}, Takeuchi's IC \citep{Tak76}, the
Kullback-Leibler IC \citep{MR1707178}. 
The AIC and BIC remain the most popular forms used
in risk-analytic settings
\citep{SanFalVic02,BaiNobWhe05,FaeAerGey07}.

Given the quantal-response data
$(\mathbf{y},\mathbf{N},\mathbf{d})$ and a model $\mathcal{M}$ which
specifies a dose-response function $\pi(\cdot;\theta)$ with $\theta
\in \Theta$, the likelihood function is
\begin{equation}
\label{lik func}
L(\theta|\mathcal{M}) = L(\theta | \mathcal{M}; (\mathbf{y},\mathbf{N},\mathbf{d})) =
\prod_{j=1}^J \comb{N_j}{y_j} \pi(d_j;\theta)^{y_j}
[1 - \pi(d_j;\theta)]^{N_j - y_j},
\end{equation}
so the relevant portion of the log-likelihood function
$\log L(\theta|\mathcal{M})$ is
\begin{equation}
l(\theta|\mathcal{M}) =
\sum_{j=1}^J \left\{ y_j \log \pi(d_j;\theta) +
(N_j - y_j) \log [1 - \pi(d_j;\theta)]\right\}.
\label{log-lik func}
\end{equation}
The maximum likelihood estimate (MLE) of $\theta$ under $\mathcal{M}$ is
\begin{equation}
\label{MLE}
\hat{\theta} = \hat{\theta}(\mathbf{y},\mathbf{N},\mathbf{d}) =
\arg\max_{\theta \in \Theta} L(\theta|\mathcal{M}; (\mathbf{y},\mathbf{N},\mathbf{d})) =
\arg\max_{\theta \in \Theta} l(\theta|\mathcal{M}).
\end{equation}
The AIC for model $\mathcal{M}$ is
\begin{equation}
\label{AIC}
AIC(\mathcal{M}) =-2 l(\hat{\theta}|\mathcal{M}) + 2g
\end{equation}
where $g$ is the dimension of $\Theta$, The BIC is
\begin{equation}
\label{BIC}
BIC(\mathcal{M}) =-2 l(\hat{\theta}|\mathcal{M}) + g\log(n)
\end{equation}
where $n = \sum_{j=1}^K N_j$ is the total number of units.

In the presence of several competing models $\{\mathcal{M}_m: m = 1,2,\ldots,M\}$, the
index of the chosen model class using the AIC or BIC approaches are, respectively,
\begin{eqnarray}
\hat{M}_{AIC}(\mathbf{y},\mathbf{N},\mathbf{d}) & = &
\arg\min_{m=1,2,\ldots,M} AIC(\mathcal{M}_m|\mathbf{y},\mathbf{N},\mathbf{d});
\label{AIC chosen model index} \\
\hat{M}_{BIC}(\mathbf{y},\mathbf{N},\mathbf{d}) & = &
\arg\min_{m=1,2,\ldots,M} BIC(\mathcal{M}_m|\mathbf{y},\mathbf{N},\mathbf{d}).
\label{BIC chosen model index}
\end{eqnarray}

We employ these two model selection approaches to the quantal-response
problem. To simplify our notation, let $\mathcal{M}_{m0}
\equiv \mathcal{M}_0$ and $\Theta_{m0} \equiv \Theta_0$. Recall
that
\begin{displaymath}
g_{ml} = \dim{\Theta}_{ml},\ l=0,1,2,\ldots,L_m; m=1,2,\ldots,M,
\end{displaymath}
the dimension of the sub-parameter space $\Theta_{ml}$ of
$\mathcal{M}_{ml}$. The AIC and BIC become, for
$l=0,1,2,\ldots,L_m; m=1,2,\ldots,M$,
\begin{eqnarray}
AIC(\mathcal{M}_{ml}) & = & -2l(\hat{\theta}_{0},\hat{\theta}_{m1},\ldots,\hat{\theta}_{ml}|
\mathcal{M}_{ml}) + 2\left(g_0 + \sum_{i=1}^l g_{mi}\right);
\label{model class AIC} \\
BIC(\mathcal{M}_{ml}) & = & -2l(\hat{\theta}_{0},\hat{\theta}_{m1},\ldots,\hat{\theta}_{ml}|
\mathcal{M}_{ml}) + \left(g_0 + \sum_{i=1}^l g_{mi}\right) \log(n),
\label{model class BIC}
\end{eqnarray}
where $(\hat{\theta}_0,\hat{\theta}_{m1},\ldots,\hat{\theta}_{ml})$ is the MLE of
$(\theta_0,\theta_{m1},\ldots,\theta_{ml})$ under model $\mathcal{M}_{ml}$.
Observe that $(\hat{\theta}_0,\hat{\theta}_{m1},\ldots,\hat{\theta}_{ml-1}),$
the MLE under model class $\mathcal{M}_{ml-1}$
coincides with the {\em restricted} MLE of $(\theta_0,\theta_{m1},\ldots,\theta_{ml-1})$
under model class $\mathcal{M}_{ml}$ under
the restriction $\theta_{ml} = \theta_{ml}^0$.

Model class $\mathcal{M}_{m^*l^*}$ will then be chosen according to the AIC approach whenever
\begin{displaymath}
AIC(\mathcal{M}_{m^*l^*}) = \min_{m=1,2,\ldots,M}\min_{l=0,1,2\ldots,L_m}
AIC(\mathcal{M}_{ml});
\end{displaymath}
while it will be chosen via the BIC approach whenever
\begin{displaymath}
BIC(\mathcal{M}_{m^*l^*}) = \min_{m=1,2,\ldots,M}\min_{l=0,1,2\ldots,L_m}
BIC(\mathcal{M}_{ml}).
\end{displaymath}

We point out that, though of interest by itself,
the model class selection problem is not the primary aim
in these risk benchmarking studies. Rather, of more importance
is estimation of the BMD $\xi_q$.
Thus, the model class selection aspect,
though possibly crucial in the inferential process,
acquires a somewhat secondary role.
In the next two sections, we describe approaches for
estimating $\xi_q$
which take into account the model class selection step.

\section{Two-Step BMD Estimation Approach}
\label{sect-Two Step Approach}

Let us  suppose that the true underlying model class is $\mathcal{M}_{ml}$
for some $l \in \{0,1,2,\ldots,L_m\}$ and $m \in \{1,2,\ldots,M\}$, with true dose-response
function $\pi_{ml}(\cdot;\theta_0,\theta_{m1},\ldots,\theta_{ml}).$ Denote its
associated extra risk function by
\begin{displaymath}
\pi_{ml,E}(d;\theta_0,\theta_{m1},\ldots,\theta_{ml}) =
\frac{\pi_{ml}(d;\theta_0,\theta_{m1},\ldots,\theta_{ml}) -
\pi_{ml}(0;\theta_0,\theta_{m1},\ldots,\theta_{ml})}
{1 - \pi_{ml}(0;\theta_0,\theta_{m1},\ldots,\theta_{ml})},
\end{displaymath}
and the inverse of this extra risk function by
$\pi_{ml,E}^{-1}(\cdot;\theta_0,\theta_{m1},\ldots,\theta_{ml})$. Then, under model
class $\mathcal{M}_{ml}$, the BMD for a fixed $q = \mbox{BMR}$ is
\begin{equation}
\label{BMD in Mml}
\xi_{q|ml}(\theta_0,\theta_{m1},\ldots,\theta_{ml}) =
\pi_{ml,E}^{-1}(q;\theta_0,\theta_{m1},\ldots,\theta_{ml}).
\end{equation}
It is natural to apply the substitution estimator where
$(\theta_0,\theta_{m1},\ldots,\theta_{ml})$ in (\ref{BMD in Mml})
is replaced by its ML estimate
{\em under} model class $\mathcal{M}_{ml}$. Thus, if $(\hat{\theta}_0,\hat{\theta}_{m1},
\ldots,\hat{\theta}_{ml})$ is the ML estimator under model class $\mathcal{M}_{ml}$
based on $(\mathbf{Y},\mathbf{N},\mathbf{d})$, then the estimator of $\xi_{q|ml}$ is
\begin{equation}
\label{Estimator BMD in Mml}
\hat{\xi}_{q|ml} = 
\xi_{q|ml}(\hat{\theta}_0,\hat{\theta}_{m1},\ldots,\hat{\theta}_{ml}) =
\pi_{ml,E}^{-1}(q;\hat{\theta}_0,\hat{\theta}_{m1},\ldots,\hat{\theta}_{ml}).
\end{equation}

One approach to estimating the BMD
among several competing model classes is
to combine the model selection and estimation steps into
a {\em two-step approach} \citep{HwaYooKim09,FaeAerGey07,SanFalVic02}.
The idea is to use the data to select the model class,
either via AIC or BIC, and having chosen the model class, obtain the
estimate of the BMD in the chosen model class, but with the estimate still based on
the same data utilized in the model class selection step.

Under our framework let us then suppose that we have decided on a model class selection procedure,
either AIC or BIC. Denote by
\begin{equation}
\label{estimated model class indices}
(\hat{m},\hat{l}) \equiv \left(\hat{m}(\mathbf{Y},\mathbf{N},\mathbf{d}),
\hat{l}(\mathbf{Y},\mathbf{N},\mathbf{d})\right)
\end{equation}
the resulting
model type index and the model type depth, respectively, of the selected model
class. The two-step estimator of the BMD is then given by
\begin{equation}
\label{two-step estimator}
\hat{\xi}_q^{TS} \equiv \hat{\xi}_q^{TS}(\mathbf{Y},\mathbf{N},\mathbf{d}))
= {\xi}_{q|\hat{m}(\mathbf{Y},\mathbf{N},\mathbf{d}))\,
\hat{l}(\mathbf{Y},\mathbf{N},\mathbf{d})}(\mathbf{Y},\mathbf{N},\mathbf{d}).
\end{equation}
In (\ref{two-step estimator}) we have explicitly shown where the data enter
the picture. The difficulty in these two-step estimators is the re-use
(``double-dipping'') of the data $(\mathbf{Y},\mathbf{N},\mathbf{d})$ since we use
them to
select the model class indices $(\hat{m},\hat{l})$, and then we again use them
to estimate the BMD within the chosen model class.
In assessing the properties of such two-step estimators, it is imperative that
the impact of this data double-dipping be taken into account.  Unless corrected for
this additional stochastic element of the estimation process, confidence regions and
significance tests will not possess the desired coverage levels or correct error
rates; see, for instance, \cite{ClaHjo03,DukPen05}.

\section{Model-Averaging Approach to BMD Estimation}
\label{sect-Model Averaging Approach}

Another approach to estimating $\xi_q$ is via a {\em model-averaging} procedure;
see, for instance, \cite{BMA99, BurAnd02, HjoCla03, DukPen05}. The idea here is
to combine
estimates from the different model classes via some form of weighted average, with
the weights constructed to quantify each model's relative likelihood in describing the
data. For our framework, we will specify data-dependent weights
\begin{displaymath}
\hat{\mathbf{w}} \equiv \hat{\mathbf{w}}(\mathbf{Y},\mathbf{N},\mathbf{d}) =
\{\hat{w}_0,(\hat{w}_{ml},l=1,2,\ldots,L_m,m=1,2,\ldots,M)\},
\end{displaymath}
so that the associated model-averaged estimator of $\xi_q$ will be
\begin{eqnarray}
\hat{\xi}_q^{MA}(\mathbf{Y},\mathbf{N},\mathbf{d}) =
\hat{w}_0(\mathbf{Y},\mathbf{N},\mathbf{d}) \hat{\xi}_{q|0}(\mathbf{Y},\mathbf{N},\mathbf{d}) + \nonumber \\  \sum_{m=1}^M \sum_{l=1}^{L_m}
\hat{w}_{ml}(\mathbf{Y},\mathbf{N},\mathbf{d}) \hat{\xi}_{q|ml}(\mathbf{Y},\mathbf{N},\mathbf{d}).
\label{model-average estimator}
\end{eqnarray}
There are several ways to specify the weights. Perhaps the simplest avenue is to impose a Bayesian
structure to the problem, so that the weights become related to the posterior probabilities of
each of the model classes \citep{BaiNobWhe05,MorIbrChe06}.
%
Here we describe the more conventional approach
where the weights arise from the AIC or BIC values.
The AIC-based weights are, for $m=1,2,\ldots,M$ and $l=0,1,2,\ldots,L_m$, computed according to
the so-called Akaike weights \cite[\S~2.9]{BurAnd02} given by
\begin{equation}
\label{AIC weights}
\hat{w}_{ml}^{AIC} = \frac{\exp\{-\frac{1}{2} AIC(\mathcal{M}_{ml})\}}
{\exp\{-\frac{1}{2} AIC(\mathcal{M}_{0})\} + \sum_{m=1}^M \sum_{l=1}^{L_m}
\exp\{-\frac{1}{2} AIC(\mathcal{M}_{ml}\}}.
\end{equation}
Observe that these weights are data-dependent since the AIC-values are derived
from (\ref{model class AIC}). Similarly, the data-dependent BIC weights are specified via
\begin{equation}
\label{BIC weights}
\hat{w}_{ml}^{BIC} = \frac{\exp\{-\frac{1}{2} BIC(\mathcal{M}_{ml})\}}
{\exp\{-\frac{1}{2} BIC(\mathcal{M}_{0})\} + \sum_{m=1}^M \sum_{l=1}^{L_m}
\exp\{-\frac{1}{2} BIC(\mathcal{M}_{ml}\}},
\end{equation}
where the BIC values are computed using (\ref{model class BIC}). In the above formulas,
recall our earlier notation where a subscript of ``$_{m0}$'' coincides with
the subscript `$_0$' so that $\hat{w}_{m0}$ is $\hat{w}_{0}$.

As in the two-step estimator of $\xi_q$, investigating the theoretical properties of these
model-averaged estimators is non-trivial owing to the dependence of both the model-averaging
weights and the estimator of $\xi_q$ in each model class; see, for instance, the evaluation of
the properties of such estimators in specific models in \cite{DukPen05}.
For this quantal-response problem, we will investigate the properties
of these model-averaged estimators via computer simulation studies in
a later section.

\section{A Focused-Inference Approach}
\label{sect-Focused Inference Approach}

This section presents an approach which integrates the model
selection and estimation steps. In contrast to the model class selection
procedures in Section \ref{sect-Model Selection} which choose the model class
without direct regard to the parameter of main interest,
the focused-inference approach
takes into consideration in the model class selection stage
the fact that the BMD $\xi_q$ is the parameter of primary interest. This
strategy was developed in \cite{ClaHjo03,HjoCla03,ClaHjo08}. Since the problem
is of a general nature, we will first present the solution for the
larger problem and then apply it  to benchmark dose estimation.

\subsection{Description of the General Setting}
\label{subsect-General Considerations}

We suppose that for a sample size $n$ we are able to observe the realization
of a random observable $X_n$ taking values in a sample space $\mathcal{X}_n$. We
 denote by $F_n$ the distribution function of $X_n$, and assume that
%
$F_n \in \mathcal{F} \equiv \cup_{m=1}^M \mathcal{F}_{mn},$
%
where $\mathcal{F}_{mn} = \{F_{mn}(\cdot;\theta_m):\ \theta_m \in \Theta_m\}$
with $\Theta_m$ an open subset of $\Re^{g_m}$ for a known positive integer $g_m$. We denote
by $f_{mn}(\cdot;\theta_m)$ the density function of $F_{mn}(\cdot;\theta_m)$ with respect
to some dominating measure $\nu_n$, e.g., Lebesgue or counting measure. Model class
$m \in \{1,2,\ldots,M\}$ will be
%
$\mathcal{M}_m = \{\mathcal{F}_{mn}:\ n=1,2,\ldots\}.$
%

We suppose that the parameter of primary interest is a functional $\tau$
on $\mathcal{F}$ which takes the form
\begin{equation}
\label{parameter of interest}
\tau = \sum_{m=1}^M \tau_m(\theta_m) I\{F_n \in \mathcal{F}_{mn}\},
\end{equation}
where $\tau_m: \Theta_m \rightarrow \Re$. We assume that each
$\tau_m$ possesses `smoothness properties' such as differentiability and continuity
with respect to each $\theta_m$ component.
In this general setting, the primary goal is to estimate
$\tau$ based on $X_n$ and to obtain properties of the estimator.
Of secondary interest is to determine a parsimonious model class
containing $F_n$.
We start by examining asymptotic properties of
estimators of $\tau_m$ under the true model class and also under
a misspecified model class.

\subsection{Properties under True Model Class}

First, let us consider the situation where model class $\mathcal{M}_m$ holds, so
$F_n = F_{mn}(\cdot;\theta_m) \equiv F_{mn}(\theta_m)$ for some $\theta_m \in \Theta_m$.
Thus, in the sequel,
probability statements, including expectations, variances, and covariances,
are taken with respect to $F_{mn}(\cdot;\theta_m)$.
Furthermore, we define the operators
%
$\nabla \equiv \nabla_{\theta} = \frac{\partial}{\partial\theta}$
and
$\nabla^{\otimes 2} \equiv \nabla_{\theta\theta\trp} = \frac{\partial^2}{\partial\theta\partial\theta\trp}.$
%
We let
\begin{eqnarray}
\mathfrak{I}_{mn}(\theta_m) & = & E\left\{
\left(
\nabla_{\theta_m} \log f_{mn}(X_n;\theta_m)
\right)^{\otimes 2} | F_{mn}(\theta_m)
\right\} \nonumber  \\
& = &
-E\left\{
\nabla_{\theta_m\theta_m\trp}
 \log f_{mn}(X_n;\theta_m)  | F_{mn}(\theta_m)
\right\}
\label{Fisher Info for Mm}
\end{eqnarray}
be
the Fisher information matrix for model class $\mathcal{M}_m$.
Then, under suitable regularity conditions and results
from
ML estimation theory \cite[\S~6.3]{LehCas98}, the sequence of
MLEs $\{\hat{\theta}_{mn}, n=1,2,\ldots\}$ based on the sequence
of data $\{X_n, n=1,2,\ldots\}$, under model class $\mathcal{M}_m$,
is consistent for $\theta_m$ and has the asymptotic
distributional
property
\begin{equation}
\label{asumptotic normality of ML}
\mathfrak{L}
\left\{
\sqrt{n}(\hat{\theta}_{mn} - \theta_m) | F_{mn}(\theta_m)
\right\}
\condist
N(0,\mathfrak{I}_{m}^{-1}(\theta_m)),
\end{equation}
where $\mathfrak{L}(\cdot|F_{mn}(\theta_m))$
denotes
the distributional law
under $F_n = F_{mn}(\theta_m)$, and where
\begin{displaymath}
\mathfrak{I}_m(\theta_m) \equiv
\plim_{n\rightarrow\infty} \frac{1}{n} \mathfrak{I}_{mn}(\theta_m),
\end{displaymath}
with `plim' meaning convergence in probability.
By the Delta Method \cite[\S~1.8]{LehCas98}, with
%
$\stackrel{\bullet}{\tau}_m(\theta_m) = \nabla_{\theta_m} \tau_m(\theta_m),$
%
the following propostion follows.

\begin{prop}
\label{prop: Asymptotics under True Model}
As $n \rightarrow \infty$,
%
%
\begin{equation*}
\mathfrak{L}\left\{
\sqrt{n}\left(
\tau_m(\hat{\theta}_{mn}) - \tau_m(\theta_m)
\right) | F_{mn}(\theta_m)
\right\}
\condist
N(0,\stackrel{\bullet}{\tau}_m(\theta_m)\trp
\mathfrak{I}_m^{-1}(\theta_m)
\stackrel{\bullet}{\tau}_m(\theta_m)).
\label{Asymptotic of tau-hat in Mm}
\end{equation*}
\end{prop}

\subsection{Properties under a Misspecified Model Class}

Next, we examine the properties of $\hat{\theta}_{mn}$ when the true model
class is $\mathcal{M}_{m^\prime}$. This will enable us to obtain the properties
of $\tau_m(\hat{\theta}_{mn})$ under the model class $\mathcal{M}_{m^\prime}$.
Define the Kullback-Leibler divergence between $f_{mn}(\cdot;\theta_m)$ and
$f_{m^\prime n}(\cdot;\theta_{m^\prime})$, under $F_{m^\prime n}(\theta_{m^\prime})$,
according to
\begin{equation}
\label{KL distance}
K_{m,m^\prime}^{(n)}(\theta_m,\theta_{m^\prime}) =
\int \log
\left\{
\frac{f_{mn}(x_n;\theta_m)}{f_{m^\prime n}(x_n;\theta_{m^\prime})}
\right\}
f_{m^\prime n}(x_n;\theta_{m^\prime}) \nu_n(dx_n).
\end{equation}
We also assume that there is a function $K_{m,m^\prime}:
\Theta_m \times \Theta_{m^\prime} \rightarrow \Re$ such that
\begin{equation}
\label{zero property of score}
K_{m,m^\prime}(\theta_m,\theta_{m^\prime}) =
\lim_{n \rightarrow \infty}
\frac{1}{n} K_{m,m^\prime}^{(n)}(\theta_m,\theta_{m^\prime}).
\end{equation}
Define
\begin{equation}
\label{projection}
\theta_{m,m^\prime}^{*(n)}(\theta_{m^\prime}) =
\arg\max_{\theta_m \in \Theta_m} K_{m,m^\prime}^{(n)}(\theta_m,\theta_{m^\prime}),
\end{equation}
so $f_{mn}(\cdot;\theta_{m,m^\prime}^{*(n)}(\theta_{m^\prime}))$ is the
closest element of $\mathcal{F}_{mn}$ to the true density
$f_{m^\prime n}(\cdot;\theta_{m^\prime})$ according to Kullback-Leibler
divergence, also referred to as the {\em quasi} true model in the assumed model
class $\mathcal{F}_{mn}$. We assume that
\begin{displaymath}
\lim_{n \rightarrow \infty} \theta_{m,m^\prime}^{*(n)}(\theta_{m^\prime}) =
\theta_{m,m^\prime}^*(\theta_{m^\prime})
\end{displaymath}
where
\begin{equation}
\label{limit projection}
\theta_{m,m^\prime}^*(\theta_{m^\prime}) =
\arg\max_{\theta_m \in \Theta_m} K_{m,m^\prime}(\theta_m,\theta_{m^\prime}).
\end{equation}
By Jensen's Inequality \cite[\S~1.7]{LehCas98}
note that we have
$K_{m,m^\prime}^{(n)}(\theta_m,\theta_{m^\prime}) \le 0$
with equality iff $m = m^\prime$. In particular,
%
$K_{m,m^\prime}^{(n)}(\theta_{m,m^\prime}^{*(n)}(\theta_{m^\prime}),\theta_{m^\prime}) \le 0$,
%
with equality iff $m = m^\prime$, in which case $\theta_{m,m}^{*(n)}(\theta_m) = \theta_m$.
Furthermore, under suitable regularity conditions, note that
$\theta_{m,m^\prime}^{*(n)}(\theta_{m^\prime})$ solves the equation
\begin{equation}
\label{zero property}
\int \left\{
\nabla_{\theta_m} \log f_{mn}(x_n;\theta_m)
\right\}
f_{m^\prime n}(x_n;\theta_{m^\prime}) \nu_n(dx_n) = 0.
\end{equation}
Let
%
$$U_{mn}(\theta_m;x_n) = \nabla_{\theta_m} \log f_{mn}(x_n;\theta_m)$$
%
be the score function of $\theta_m$, under $F_{mn}$, given data $x_n$. Then,
from (\ref{zero property}), we have that at
$\theta_m = \theta_{mm^\prime}^{*(n)}(\theta_{m^\prime})$,
\begin{displaymath}
E\left\{U_{mn}(\theta_m;X_n)|F_{m^\prime n}(\theta_{m^\prime})\right\} =
\int U_{mn}(\theta_m;x_n) f_{m^\prime n}(x_n;\theta_{m^\prime}) \nu_n(dx_n) = 0.
\end{displaymath}
Let
%
$$A_{mn}(\theta_m;x_n) =
- \nabla_{\theta_m \theta_m\trp} \log f_{mn}(x_n;\theta_m)$$
%
be the observed information matrix function under model $F_{mn}$ given data $x_n$.
We assume that, under model class $\mathcal{M}_{m^\prime}$, there exists
a vector function $U_{m,m^\prime}: \Theta_m \times \Theta_{m^\prime} \rightarrow \Re^g$ and
a matrix function $A_{m,m^\prime}: \Theta_m \times \Theta_{m^\prime} \rightarrow
(\Re^g)^{\otimes 2}$ such that, under model class $\mathcal{M}_{m^\prime}$, we have
\begin{eqnarray*}
\frac{1}{n} U_{mn}(\theta_m;X_n)  \conuprob  U_{m,m^\prime}(\theta_m,\theta_{m^\prime}) \quad \mbox{and} \quad
\frac{1}{n} A_{mn}(\theta_m;X_n)  \conuprob  A_{m,m^\prime}(\theta_m,\theta_{m^\prime}),
\end{eqnarray*}
where `$\conuprob$' means uniform convergence in probability.
The required uniform convergence in probability
is only needed in a neighborhood of $\theta_{mm^\prime}^{*}(\theta_{m^\prime})$.
Note that
%
$A_{m,m}(\theta_m,\theta_m) = \mathfrak{I}_{m}(\theta_m).$
%

We will now obtain the asymptotic distribution of the sequence of estimators
$\{\hat{\theta}_{mn}(X_n):\ n=1,2,\ldots\}$ when the true model class is
$\mathcal{M}_{m^\prime}$. By the defining property of $\hat{\theta}_{mn}$, we have
%
$\frac{1}{n} U_{mn}(\hat{\theta}_{mn}; X_n) = 0.$
%
Expanding this at $\theta_{m,m^\prime}^{*(n)}(\theta_{m^\prime})$, we
achieve
\begin{eqnarray*}
0 & = & \frac{1}{\sqrt{n}} U_{mn}(\hat{\theta}_{mn};X_n) \\
& = & \frac{1}{\sqrt{n}} U_{mn}\left(\theta_{m,m^\prime}^{*(n)}(\theta_{m^\prime}); X_n\right)
 - \frac{1}{\sqrt{n}} A_{mn}(\tilde{\theta}_n;X_n)
\left(\hat{\theta}_{mn} - \theta_{m,m^\prime}^{*(n)}(\theta_{m^\prime})\right),
\end{eqnarray*}
where $\tilde{\theta}_n \in [\hat{\theta}_{mn}, \theta_{m,m^\prime}^{*(n)}(\theta_{m^\prime})]$.
Consequently,
\begin{eqnarray}
\sqrt{n} \left(\hat{\theta}_{mn} -
\theta_{m,m^\prime}^{*(n)}(\theta_{m^\prime})\right) = 
\left[
\frac{1}{n} A_{mn}(\tilde{\theta}_n;X_n)
\right]^{-1}
\frac{1}{\sqrt{n}}  U_{mn}\left(\theta_{m,m^\prime}^{*(n)}(\theta_{m^\prime}); X_n\right).
\label{thetahat representation}\end{eqnarray}
Under
suitable
regularity conditions it can be shown that, under model class $\mathcal{M}_{m^\prime}$,
$\hat{\theta}_{mn} \conprob \theta_{m,m^\prime}^*(\theta_{m^\prime})$. This implies that
$\tilde{\theta}_n \conprob \theta_{m,m^\prime}^*(\theta_{m^\prime})$. As a consequence,
we have that
\begin{displaymath}
\frac{1}{n} A_{mn}(\tilde{\theta}_n; X_n) \conprob
A_{m,m^\prime}(\theta_{m,m^\prime}^*(\theta_{m^\prime}),\theta_{m^\prime}).
\end{displaymath}
On the other hand, under model class $\mathcal{M}_{m^\prime}$, we assume that
\begin{equation}
\label{asymptotic normality of score}
\frac{1}{\sqrt{n}}
U_{mn}\left(\theta_{m,m^\prime}^{*(n)}(\theta_{m^\prime}),X_n\right)
\condist
N(0,\Sigma_{m,m^\prime}(\theta_{m,m^\prime}^*(\theta_{m^\prime}),\theta_{m^\prime})),
\end{equation}
where, with
%
$\Sigma_{m,m^\prime}^{(n)}(\theta_m,\theta_{m^\prime}) = Cov\left\{
U_{mn}(X_n;\theta_m)
| F_{m^\prime n}(\theta_{m^\prime})
\right\},$
%
we have
\begin{displaymath}
\Sigma_{m,m^\prime}(\theta_{m,m^\prime}^*(\theta_{m^\prime}),\theta_{m^\prime})) =
\lim_{n\rightarrow\infty}
\frac{1}{n} \Sigma_{m,m^\prime}^{(n)}(\theta_{m,m^\prime}^{*(n)}(\theta_{m^\prime}),\theta_{m^\prime}).
\end{displaymath}
Note that, for $m = 1,2,\ldots,M$, we find
%
$\Sigma_{m,m}(\theta_{m,m}^*(\theta_m),\theta_m) =
\mathfrak{I}_m(\theta_m).$
%
As a consequence, we have the following proposition:

\begin{prop}
\label{prop: Asymptotics under Misspec Model}
As $n \rightarrow \infty$,
%
$$\mathfrak{L}
\left\{
\sqrt{n}
\left(
\hat{\theta}_{mn} - \theta_{m,m^\prime}^{*(n)}(\theta_{m^\prime})
\right) | \mathcal{M}_{m^\prime}
\right\}
\condist
N(0,\Xi_{m,m^\prime}(\theta_{m^\prime})),$$
%
where
\begin{eqnarray*}
\lefteqn{ 
\Xi_{m,m^\prime}(\theta_{m^\prime}) =
A_{m,m^\prime}^{-1}\left(\theta_{m,m^\prime}^*(\theta_{m^\prime}),\theta_{m^\prime}\right)
} \\ &&
\Sigma_{m,m^\prime}(\theta_{m,m^\prime}^*(\theta_{m^\prime}),\theta_{m^\prime}))
A_{m,m^\prime}^{-1}\left(\theta_{m,m^\prime}^*(\theta_{m^\prime}),\theta_{m^\prime}\right).
\end{eqnarray*}
\end{prop}

Finally, by applying the Delta Method, we obtain the following result concerning
the asymptotic properties of ${\tau}(\hat{\theta}_{mn})$ under a misspecified model.

\begin{theo}
\label{theo: Asymptotics under Misspec Model}
As $n \rightarrow \infty$,
\begin{equation*}
\label{asymptotic of tauhat}
\mathfrak{L}\left\{
\sqrt{n}
\left(
\tau_m(\hat{\theta}_{mn}) - \tau_m(\theta_{m,m^\prime}^{*(n)}(\theta_{m^\prime}))
\right) | \mathcal{M}_{m^\prime}
\right\}
\condist
N\left(0,\Gamma_{m,m^\prime}(\theta_{m^\prime})\right),
\end{equation*}
where
%
$\Gamma_{m,m^\prime}(\theta_{m^\prime}) =\
\stackrel{\bullet}{\tau}_m(\theta_{m,m^\prime}^*(\theta_{m^\prime}))\trp
\Xi_{m,m^\prime}(\theta_{m^\prime})
\stackrel{\bullet}{\tau}_m(\theta_{m,m^\prime}^*(\theta_{m^\prime})).$
%
\end{theo}

We remark that the result in Proposition \ref{prop: Asymptotics under True Model}
can be recovered from Theorem \ref{theo: Asymptotics under Misspec Model} by noting
that for each $m = 1,\ldots,M$, we have $\theta_{m,m}^*(\theta_m) = \theta_m$ and
%
$\Gamma_{m,m}(\theta_m) = \, \stackrel{\bullet}{\tau}_m(\theta_m)\trp \,
\mathfrak{I}_m(\theta_m)^{-1} \, \stackrel{\bullet}{\tau}_m(\theta_m).$
%
In addition, we also point out that the true model need not actually be a parametric
model. The derivations above, with a slight change in notation, also hold if
the true model is simply represented by $\mathcal{M}^{\prime(n)}$ with the governing
distribution of $F^{\prime(n)}(\cdot)$. In such a case, for the Kullback-Leibler
divergence, we
use the mapping
\begin{displaymath}
\theta_m \mapsto \int \log\{f_{mn}(x_n;\theta_m)\} F^{\prime(n)}(dx_n).
\end{displaymath}

\subsection{Rationale for the Focused-Inference Approach}

Consider now a sequence of estimators
for $\tau$, say $\{\hat{\tau}_n: \ n=1,2,3,\ldots\}$.
How should we evaluate this sequence of estimators?
Clearly, the evaluation will depend on the true value
of $\tau$, which in turn depends on the model class that holds. A reasonable measure
of the quality of this sequence of estimators would be
\begin{equation}
\label{general risk of tau estimators}
R_{m^\prime}(\hat{\tau}_n,\theta_{m^\prime}) =
n E\left\{\left(\hat{\tau}_n - \tau_{m^\prime}(\theta_{m^\prime})\right)^2 |
\mathcal{M}_{m^\prime}\right\}, n=1,2,3,\ldots.
\end{equation}
This represents the (re-scaled) risk function, associated with squared-error loss, under model class
$\mathcal{M}_{m^\prime}$.
[Note that by `risk function' here we mean expected loss, as in the usual
decision-theoretic paradigm.  This should not be confused with the extra risk
function in \eqref{extra risk} used to define the BMD.]
To simplify our notation, let
%
$\hat{\tau}_{mn} \equiv \tau_m(\hat{\theta}_{mn}),\ n=1,2,3,\ldots.$
%
Then, let
%
$R_{m,m^\prime}^{(n)}(\theta_{m^\prime}) \equiv
R_{m^\prime}(\hat{\tau}_{mn},\theta_{m^\prime}).$
%
By using the identity
\begin{eqnarray*}
\lefteqn{ 
n\left(
\tau_m(\hat{\theta}_{mn}) - \tau_{m^\prime}(\theta_{m^\prime})
\right)^2   = } \\ &&
n\left[
\left(\tau_m(\hat{\theta}_{mn}) - \tau_m(\theta_{m,m^\prime}^*(\theta_{m^\prime}))\right) +
\left(\tau_m(\theta_{m,m^\prime}^*(\theta_{m^\prime})) - \tau_{m^\prime}(\theta_{m^\prime})\right)
\right]^2,
\end{eqnarray*}
we obtain from the earlier asymptotic results that, for large $n$,
\begin{equation}
\label{approximate risks}
R_{m,m^\prime}^{(n)}(\theta_{m^\prime}) \approx
\Gamma_{m,m^\prime}(\theta_{m^\prime}) +
n \left[
\tau_m(\theta_{m,m^\prime}^*(\theta_{m^\prime})) - \tau_{m^\prime}(\theta_{m^\prime})
\right]^2,
\end{equation}
a variance-bias decomposition. Note that to obtain (\ref{approximate risks}), we also used
the result that
\begin{displaymath}
E\{\sqrt{n}[\hat{\tau}_n - \tau_m(\theta_{m,m^\prime}^{*}(\theta_{m^\prime})]|\mathcal{M}_{m^\prime}\} = o(1).
\end{displaymath}
Furthermore, observe that, for large $n$,
\begin{displaymath}
R_{m,m}^{(n)}(\theta_m) \approx \Gamma_{m,m}(\theta_m) =
\, \stackrel{\bullet}{\tau}_m(\theta_m)\trp \,
\mathfrak{I}_m(\theta_m)^{-1} \, \stackrel{\bullet}{\tau}_m(\theta_m).
\end{displaymath}
We now describe possible approaches to utilizing the above risks for model selection and
BMD estimation.

\subsection{An Empirical-Based Approach}
\label{subsec: empirical-based approach}
Presumably there is a sequence of true models $\{\mathcal{M}^\prime_n\}$ governing the
data sequence $\{X_n\}$. We do not know this sequence of true models, and it need not
coincide with the possible model classes under consideration, so we will not know
the quantities $\theta_{\mathcal{M}^\prime}$ and $\theta^*_{m,\mathcal{M}^\prime}$. As such
we will not know the risks $R_{m,{\mathcal{M}^\prime}}^{(n)}$.

A possible approach is to use a nonparametric estimate of $F_n^\prime$, say
$\hat{F}_n^\prime$, and to use this estimate to obtain estimates of both
$\tau_{\mathcal{M}^\prime}$ and $\theta_{m,\mathcal{M}^\prime}^{*}$. Let such
estimates be $\hat{\tau}_{\mathcal{M}^\prime}^{(n)}$ and $\hat{\theta}_{m,\mathcal{M}^\prime}^{*(n)}$,
the latter being the KL projection of $\hat{F}_n^\prime$ on $\Theta_m$. For each
$m$ we may then estimate $R_{m,\mathcal{M}^\prime}$ by
\begin{displaymath}
\hat{R}_{m,\mathcal{M}^\prime}^{EMP,(n)} =
\Gamma_{m,\mathcal{M}^\prime}(\hat{\theta}_{m,\mathcal{M}^\prime}^{*(n)}) +
n\left[\tau_m(\hat{\theta}_{m,\mathcal{M}^\prime}^{*(n)}) -
\hat{\tau}_{\mathcal{M}^\prime}^{(n)}\right]^2.
\end{displaymath}

On the basis of these empirically estimated risks, a possible model selector is
\begin{equation}
\label{empirical-based model selector}
\hat{m}^{EMP} = \arg\min_{m} \hat{R}_{m,\mathcal{M}^\prime}^{EMP},
\end{equation}
and the associated BMD estimator is
\begin{equation}
\label{BMD empirical-based estimator}
\hat{\tau}_n^{EMP} = \tau_{\hat{m}^{EMP}}(\hat{\theta}_{\hat{m}^{EMP},n}).
\end{equation}

\subsection{A Model-Based Approach}
\label{subsec: model-based approach}
Another approach to estimating the risks $R_{m,m^\prime}^{(n)}(\theta_{m^\prime})$ is by
substituting for $\theta_{m^\prime}$ the estimator $\hat{\theta}_{m^\prime n}$.
However, because we could not really be certain that model $\mathcal{M}^\prime$ is the
true model, in estimating the bias term we replace $\tau_{\mathcal{M}^\prime}$ by
an empirical estimator such as the one utilized in the preceding subsection. Observe that
if we also estimate $\tau_{m^\prime}(\theta_{m^\prime})$ by $\tau_{m^\prime}(\hat{\theta}_{m^\prime})$,
then the estimated bias term will always become zero whenever $m = m^\prime$. This would be fine if
the actual underlying model truly belongs to the models under consideration, but this could be
misleading if the true model class is not among the considered models. Hence, the rationale for
the use of a nonparametric estimate of $\tau_{m^\prime}(\theta_{m^\prime})$ in estimating
the bias term.

Now, denote the resulting estimator of $R_{m,m^\prime}(\theta_{m^\prime})$ by
$\hat{R}_{m,m^\prime}^{(n)}$ for $n = 1,2,3,\ldots$ and
$m,m^\prime \in \{1,2,\ldots,M\}$. We may picture these risk estimates as in Table
\ref{table-risk estimates}.

\begin{table}[h]
\caption{Decision-theoretic risk estimates of different model class-based estimators under the possible
model classes.}
\label{table-risk estimates}
\begin{center}
\begin{tabular}{||c||c|c|c|c||} \hline
Estimator Based On & \multicolumn{4}{c||}{True Underlying Model Class} \\
Model Class & \multicolumn{4}{c||}{$m^\prime$} \\ \cline{2-5}
$m$ & $1$ & $2$ & $\cdots$ & $M$ \\ \hline
1 & $\hat{R}_{1,1}^{(n)}$ & $\hat{R}_{1,2}^{(n)}$ & $\cdots$ & $\hat{R}_{1,M}^{(n)}$ \\
2 & $\hat{R}_{2,1}^{(n)}$ & $\hat{R}_{2,2}^{(n)}$ & $\cdots$ & $\hat{R}_{2,M}^{(n)}$ \\
$\vdots$ & $\vdots$ & $\vdots$ & $\vdots$ & $\vdots$ \\
$M$ & $\hat{R}_{M,1}^{(n)}$ & $\hat{R}_{M,2}^{(n)}$ & $\cdots$ & $\hat{R}_{M,M}^{(n)}$ \\
\hline
\end{tabular}
\end{center}
\end{table}

Suppose for the moment that our goal
is to select the model class that holds
as informed by the parametric function $\tau$. For each possible model class
$\mathcal{M}^{\prime}$, we may determine the estimator yielding the smallest risk.
Having done so, we may then determine the model class that yields the smallest
among these lowest risks. As such, a possible model class index selector,
focused towards the estimation of $\tau$, is
\begin{equation}
\label{focused model class selector}
\hat{m}^{FM} =
{\arg\min}_{m^\prime = 1,2,\ldots,M}
\left\{
\min_{m = 1,2,\ldots,M} \hat{R}_{m,m^\prime}^{(n)}
\right\}.
\end{equation}
This could be referred to as a {\em $\tau$-focused model class selector}. If the primary
goal is to select the model class, then it will be $\mathcal{M}_{\hat{m}^{FM}}$.
An associated BMD estimator will then be
\begin{equation}
\label{BMD estimator tau focused}
\hat{\tau}^{FM} = \tau_{\hat{m}^{FM}}(\hat{\theta}_{\hat{m}^{FM}}).
\end{equation}

As noted earlier, however, the model class selection problem is not of
primary interest. Rather, we seek to estimate the parametric
functional $\tau$.
The viewpoint utilized
in the development of the model class selector (\ref{focused model class selector})
may not therefore be the most appropriate in terms of choosing the estimator
of $\tau$.

Instead, we argue as follows. Given the $\tau$-estimator based on
model class $\mathcal{M}_m$, we may ask which model class $\mathcal{M}_{m^\prime}$ leads
to the smallest risk. Having done so, we then ask which among the $M$ model class-based estimators
yields the smallest among these lowest risks. This motivates the model class index selector
\begin{equation}
\label{general focused estimator selector}
\hat{m}^{FE} =
{\arg\min}_{m = 1,2,\ldots,M}
\left\{
\min_{m^\prime = 1,2,\ldots,M} \hat{R}_{m,m^\prime}^{(n)}
\right\}.
\end{equation}
The resulting {\em focused estimator of $\tau$} is
\begin{equation}
\label{focused tau estimator}
\hat{\tau}^{FE} = \tau_{\hat{m}^{FE}}(\hat{\theta}_{\hat{m}^{FE}}).
\end{equation}
Note that $\hat{m}^{EMP}$, $\hat{m}^{FM}$, $\hat{m}^{FE}$, and $\hat{\theta}^{m}$
are all functions of the data $X_n$, hence the estimators of $\tau$ given by
$\hat{\tau}^{EMP}$, $\hat{\tau}^{FM}$, and $\hat{\tau}^{FE}$ all possess a two-step flavor
instead of a model-averaged flavor.

\section{Application to Quantal-Response Problem}
\label{sect - Quantal response problem}

We now apply the theory presented in Section \ref{sect-Focused Inference Approach}
to the quantal-response problem, where the random observable $X_n$ is
$(\mathbf{Y},\mathbf{N},\mathbf{d}) \equiv
(\mathbf{Y}^{(n)},\mathbf{N}^{(n)},\mathbf{d}^{(n)})$
and the model classes are indexed by $(m,l)$, where $l \in \{0,1,2,\ldots,L_m\}$ and $m \in \{1,2,
\ldots,M\}$. The parameter of interest that coincides with
$\tau$ in the preceding
section is the BMD $\xi_q$, defined by
\begin{displaymath}
\xi_q = \tau_{ml}(\theta_0,\theta_{m1},\theta_{m2},\ldots,\theta_{ml};q) =
\pi_{ml;E}^{-1}(q;\theta_0,\theta_{m1},\theta_{m2},\ldots,\theta_{ml})
\end{displaymath}
on model class $\mathcal{M}_{ml}$.
 Note that $\pi_{ml;E}(\cdot;\theta_{ml})$ is the extra risk
function associated with the dose-response function
$\pi_{ml}(\cdot;\theta_{ml})$. For this $\tau$-function we have
an explicit form of its gradient as provided in the following proposition.

\begin{prop}
\label{prop: gradient of tau}
The gradient of the $\tau$ function for a BMR at $q$ and a dose-response
function $\pi(\cdot;\theta)$ is
\begin{displaymath}
\stackrel{\bullet}{\tau}(\theta) \equiv \stackrel{\bullet}{\tau}(\theta;q) =
\frac{(1-q)\pi_{01}(0,\theta) - \pi_{01}(\tau(\theta);\theta)}
{\pi_{10}(\tau(\theta);\theta)},
\end{displaymath}
where $\pi_{10}(d;\theta) = {\frac{\partial}{\partial d}}\pi(d;\theta)$ and
$\pi_{01}(d;\theta) = \nabla_{\theta} \pi(d;\theta)$.
\end{prop}

\begin{proof}
The result follows from a straightforward application of the chain-rule of
differentiation and the total derivative rule from the defining equation of $\tau(\theta)$
given by
\begin{displaymath}
\pi_E(\tau(\theta;q);\theta) \equiv \frac{\pi(\tau(\theta;q);\theta)
- \pi(0;\theta)}{1 - \pi(0;\theta)} = q.
\end{displaymath}
\end{proof}

We first seek general expressions for the relevant entities needed to
implement the theory in the preceding section specialized to the
quantal-response problem.
Let us suppose that the true model is specified by a
probability measure $\tilde{P}$ with dose-response function
$\tilde{\pi}(\cdot)$, so that given $(N_j,d_j)$, $Y_j$ has a binomial
distribution with parameters $N_j$ and $\tilde{\pi}(d_j) \equiv
\tilde{\pi}_j$. Consider a model class $\mathcal{M}$ which specifies a
dose-response function $\pi_{\mathcal{M}}(\cdot;\theta_{\mathcal{M}})$ where
$\theta_{\mathcal{M}} \in \Theta_{\mathcal{M}}$. To simplify our notation, we
will let
%
$\pi_j(\theta) \equiv \pi_{\mathcal{M}}(d_j;\theta_{\mathcal{M}}),
j=1,2,\ldots,J,$
%
where $J$ is the number of distinct dose levels.

\begin{lemm}
\label{lemm: KL in dose response}
The relevant portion of the Kullback-Leibler divergence for assumed model
class $\mathcal{M}$ and true model $\tilde{P}$ is given by
\begin{displaymath}
K(\theta) = \sum_{j=1}^J N_j \left\{\tilde{\pi}_j \log \pi_j(\theta) + (1 -
\tilde{\pi}_j) \log(1 - \pi_j(\theta)) \right\}.
\end{displaymath}
\end{lemm}

\begin{proof}
Denote by $P_\theta$ the probability measure determined by
$\pi(\cdot;\theta)$. Denoting by $\nu$ the dominating counting measure for
both $P_\theta$ and $\tilde{P}$, then the KL divergence is
\begin{displaymath}
K(\theta) \equiv K(P_\theta,\tilde{P}) = \int
\log\left[\frac{dP_\theta/d\nu}{d\tilde{P}/d\nu}(\mathbf{y})\right] \tilde{P}(d\mathbf{y}).
\end{displaymath}
Since $d\tilde{P}/d\nu$ does not involve $\theta$, then the portion of this function
involving $\theta$ is given by
\begin{displaymath}
\tilde{E}\left\{
\sum_{j=1}^J \left[
Y_j \log\pi_j(\theta) + (N_j - Y_j)\log(1-\pi_j(\theta))
\right]\right\}
\end{displaymath}
where $\tilde{E}\{\cdot\}$ is the expectation operator with respect to the
probability measure $\tilde{P}$. The result then follows since $\tilde{E}(Y_j)
= N_j \tilde{\pi}_j$ for $j=1,2,\ldots,J$.
\end{proof}

The closest $P_\theta$ determined by the model class $\mathcal{M}$ to the true
$\tilde{P}$ with respect to KL divergence is $P_{\theta^*}$ where
\begin{equation}
\label{closest theta*}
\theta^* \equiv \theta^*_{\mathcal{M}}(\tilde{P}) = \arg\max_{\theta \in
  \Theta_{\mathcal{M}}} K(\theta).
\end{equation}
Such a $\theta^*$ could be obtained via numerical methods, such as using
optimization functions in {\tt R} \citep{R11}, e.g., {\tt optim, optimConstr}, or through
Newton-Raphson or gradient techniques.
%
%
%

\begin{lemm}
\label{lemm: gradient and information of KL}
For the KL divergence function in Lemma \ref{lemm: KL in dose response},
\begin{equation}
\label{score for KL}
U^{(n)}(P_\theta,\tilde{P}) \equiv \nabla K(\theta) = \sum_{j=1}^J
N_j \left[
\frac{\nabla \pi_j(\theta)}{\pi_j(\theta)[1-\pi_j(\theta)]}
\right]
[\tilde{\pi}_j - \pi_j];
\end{equation}
and
\begin{eqnarray}
A^{(n)}(P_\theta,\tilde{P}) & \equiv & - \nabla^{\otimes 2} K(\theta) \nonumber \\
& = & \sum_{j=1}^J N_j \left[\frac{1}{\pi_j(\theta)[1-\pi_j(\theta)]}\right]
\left\{(-\nabla^{\otimes 2}\pi_j(\theta))[\tilde{\pi}_j - \pi_j(\theta)] +
\nonumber \right. \\
&& \left.
\frac{(\nabla \pi_j(\theta))^{\otimes 2}}{\pi_j(\theta)[1-\pi_j(\theta)]}
\left[
\tilde{\pi}_j(1-\tilde{\pi}_j) + (\tilde{\pi}_j - \pi_j(\theta))^2
\right]
\right\}.
\label{info for KL}
\end{eqnarray}
\end{lemm}

\begin{proof}
Proofs of these two results are straightforward hence omitted.
\end{proof}

Expressions (\ref{score for KL}) and (\ref{info for KL}) could be utilized to
obtain $\theta^*$ via Newton-Raphson iteration given by the updating
\begin{displaymath}
\theta^* \leftarrow \theta^* + [A(P_{\theta^*},\tilde{P})]^{-1}
U(P_{\theta^*},\tilde{P}).
\end{displaymath}
From (\ref{score for KL}), we also deduce the following intuitive
result. Suppose that $\tilde{P}$ is determined by a model class $\tilde{M}$
which is contained in the model class $\mathcal{M}$, so that
\begin{displaymath}
\tilde{\pi}_{\tilde{\mathcal{M}}}(\cdot;\tilde{\theta}) =
\pi_{\mathcal{M}}(\cdot;(\tilde{\theta},\eta^0))
\end{displaymath}
for every $\tilde{\theta}=\tilde{\theta}_{\tilde{M}} \in \Theta_{\tilde{\mathcal{M}}}$ and for some
vector $\eta^0$. Then, it follows that, for this situation, we have
\begin{displaymath}
\{\tilde{\mathcal{M}} \subset \mathcal{M}\}
\Rightarrow \{\theta^* = \theta^*_{\mathcal{M}}(\tilde{\mathcal{M}}) = (\tilde{\theta},\eta^0)\}.
\end{displaymath}

The next quantity that we need is the covariance of the score function of
model $p_{\mathcal{M}}(\cdot;\theta)$ under the true model $\tilde{P}$ defined via:
\begin{displaymath}
\Sigma^{(n)}(P_\theta,\tilde{P}) = Cov\{\nabla
\log p_{\mathcal{M}}(\mathbf{Y};\theta) | \tilde{P}\}.
\end{displaymath}

\begin{lemm}
\label{lemm: covariance KL}
\begin{displaymath}
\Sigma^{(n)}(P_\theta,\tilde{P}) = \sum_{j=1}^J N_j
\frac{\tilde{\pi}_j (1 - \tilde{\pi}_j)}{[\pi_j(\theta) (1 -
    \pi_j(\theta))]^2}
(\nabla \pi_j(\theta))^{\otimes 2}.
\end{displaymath}
\end{lemm}

\begin{proof}
Again, the proof of this result is straightforward, hence omitted.
\end{proof}

With these quantities at hand, we are then able to estimate the limit
matrices $\Sigma_{\mathcal{M}}(\tilde{P})$ and $A_{\mathcal{M}}(\tilde{P})$
via
\begin{equation}
\hat{\Sigma}_{\mathcal{M}}(\tilde{P}) = \frac{1}{n}
\Sigma^{(n)}(P_{\theta^*},\tilde{P});
\label{estimator of Sigma}
\end{equation}
and
\begin{equation}
\label{estimator of A-matrix}
\hat{A}_{\mathcal{M}}(\tilde{P}) = \frac{1}{n}
A^{(n)}(P_{\theta^*},\tilde{P}),
\end{equation}
where $n = \sum_{j=1}^J N_j$. In turn, we are able to estimate the
$\Xi$-matrix from Proposition \ref{prop: Asymptotics under Misspec Model} via
\begin{equation}
\label{estimator of Xi}
\hat{\Xi}_{\mathcal{M}}(\tilde{P}) = [\hat{A}_{\mathcal{M}}(\tilde{P})]^{-1}
\hat{\Sigma}_{\mathcal{M}}(\tilde{P}) [\hat{A}_{\mathcal{M}}(\tilde{P})]^{-1},
\end{equation}
and the $\Gamma$-matrix from Theorem \ref{theo: Asymptotics under Misspec Model} via
\begin{equation}
\label{estimator of Gamma}
\hat{\Gamma}_{\mathcal{M}}(\tilde{P}) =
    [\stackrel{\bullet}{\tau}_{\mathcal{M}}(\theta^*)]\trp
[\hat{\Xi}_{\mathcal{M}}(\tilde{P})]
    [\stackrel{\bullet}{\tau}_{\mathcal{M}}(\theta^*)],
\end{equation}
where $\tau_{\mathcal{M}}(\cdot;\theta)$ is the BMD function under model
class $\mathcal{M}$.

Let us denote by $\tau(\tilde{P}) = \tau(\tilde{P};q)$ the
BMD function at BMR value $q$ under the probability measure $\tilde{P}$.
The BMD point estimator is
\begin{equation}
\label{BMD estimator under M}
\hat{\tau}_{\mathcal{M}} = \tau_{\mathcal{M}}(\hat{\theta}_{\mathcal{M}};q),
\end{equation}
where $\hat{\theta}_{\mathcal{M}}$ is the ML estimator of
$\theta_{\mathcal{M}}$ under model class $\mathcal{M}$.
When the true probability measure is
$\tilde{P}$, an estimate of the (decision-theoretic) risk of
$\hat{\tau}_{\mathcal{M}}$ is given by
\begin{equation}
\label{estimate of risk}
\hat{R}(\hat{\tau}_{\mathcal{M}},\tilde{P}) =
\hat{\Gamma}_{\mathcal{M}}(\tilde{P}) +
n[\tau_{\mathcal{M}}(\theta^*_{\mathcal{M}}(\tilde{P});q) -
  \tau(\tilde{P};q)]^2.
\end{equation}

Next, we obtain a nonparametric estimator of the true dose-response
function $\tilde{\pi}(\cdot)$. Given the observable $\{(Y_j,N_j,d_j), j=1,2,\ldots,J\}$, a
simple estimator of $(\tilde{\pi}(d_j), j=1,2,\ldots,J)$ is given by
\begin{equation}
\label{empirical estimate of probs}
\hat{\mathbf{\pi}} \equiv (\hat{\pi}_1,\hat{\pi}_2,\ldots,\hat{\pi}_J) = \left(
\frac{Y_1}{N_1}, \frac{Y_2}{N_2}, \ldots, \frac{Y_J}{N_J}\right).
\end{equation}
However, this estimator need not satisfy the monotonicity
constraint. As such, to obtain a nonparametric estimator which upholds the
monotonicity constraint, we apply the Pooled-Adjacent-Violators-Algorithm
(PAVA) (cf. \cite{RobWriDyk88}) to the estimator $\hat{\mathbf{\pi}}$ to obtain the
estimator
\begin{equation}
\label{PAVA empirical}
\check{\mathbf{\pi}} = \mbox{PAVA}(\hat{\mathbf{\pi}}).
\end{equation}
Over the region $[d_1, d_J]$, we then form the estimator $\check{\pi}(\cdot)$ of
$\tilde{\pi}(\cdot)$
as the piecewise linear function whose value at $d
= d_j$ is $\check{\pi}_j$ for $j=1,2,\ldots,J$.
This mimics a piecewise-linear, isotonic construct employed by \cite{PieXioBha12} for estimating
monotone dose-response functions in benchmark analysis.
We shall denote by $\check{P}$
the probability measure on $\mathbf{Y}$ induced by $\check{\pi}(\cdot)$.

On the other hand, for model class $\mathcal{M}^\prime$ with dose-response
function $\pi_{\mathcal{M}^\prime}(\cdot;\theta_{\mathcal{M}^\prime})$, by
replacing $\theta_{\mathcal{M}^\prime}$ by its ML estimator
$\hat{\theta}_{\mathcal{M}^\prime}$, we are also able to obtain an estimator
of the dose-response function given by
$\pi_{\mathcal{M}^\prime}(\cdot;\hat{\theta}_{\mathcal{M}^\prime})$, which in
turn induces the model-based estimated probability measure
$\hat{P}_{\mathcal{M}^\prime}$.

We are now in proper position to describe focused model selectors and
estimators. Let us denote by $\mathfrak{M}$ the collection of all model
classes under consideration, with generic element denoted by $\mathcal{M}$.
We then have the collection of risk estimates
\begin{equation}
\label{collection of empirical risk estimates}
\{\hat{R}(\hat{\tau}_{\mathcal{M}},\check{P}): \mathcal{M} \in \mathfrak{M}\}.
\end{equation}
Our empirical-based model selector becomes
\begin{equation}
\label{empirical model selector}
\hat{\mathcal{M}}^{EMP} = \arg\min_{\mathcal{M} \in \mathfrak{M}}
\hat{R}(\hat{\tau}_{\mathcal{M}},\check{P}),
\end{equation}
with corresponding empirical-based BMD estimator at BMR value of $q$ given by
\begin{equation}
\label{BMD empirical estimator}
\hat{\tau}^{EMP}(q) =
\tau_{\hat{\mathcal{M}}^{EMP}}(\hat{\theta}_{\hat{\mathcal{M}}^{EMP}};q).
\end{equation}

Following our theoretical prescription in Section \ref{sect-Focused Inference
  Approach} we also obtain the collection of risk estimates
\begin{equation}
\label{collection of model based risk estimates}
\left\{\hat{R}(\hat{\tau}_{\mathcal{M}}, \hat{P}_{\mathcal{M}^\prime}): \mathcal{M}
\in \mathfrak{M}, \mathcal{M}^\prime \in \mathfrak{M}\right\}
\end{equation}
where
\begin{displaymath}
\hat{R}(\hat{\tau}_{\mathcal{M}}, \hat{P}_{\mathcal{M}^\prime}) =
\hat{\Gamma}_{\mathcal{M}}(\hat{P}_{\mathcal{M}^\prime}) +
n[\tau_{\mathcal{M}}(\theta^*_{\mathcal{M}}(\hat{P}_{\mathcal{M}^\prime});q) -
\tau(\check{P};q)]^2.
\end{displaymath}
Note that in computing the bias, we use the empirical estimate $\tau(\check{P};q)$ of the true BMD
instead of the model-based estimate
$\tau_{\mathcal{M}^\prime}(\hat{\theta}_{\mathcal{M}^\prime};q)$ of the true
BMD.

The next model selector under consideration is defined via
\begin{equation}
\label{focused model selector}
\hat{\mathcal{M}}^{FM} = \arg\min_{\mathcal{M}^\prime \in \mathfrak{M}}
\left\{\min_{\mathcal{M} \in \mathfrak{M}} \hat{R}(\hat{\tau}_{\mathcal{M}},
\hat{P}_{\mathcal{M}^\prime})\right\}
\end{equation}
with a corresponding BMD estimator of
\begin{equation}
\label{focused model selector estimator}
\hat{\tau}^{FM}(q) =
\tau_{\hat{\mathcal{M}}^{FM}}(\hat{\theta}_{\hat{\mathcal{M}}^{FM}};q).
\end{equation}
The last model selector for consideration is defined via
\begin{equation}
\label{focused estimator selector}
\hat{\mathcal{M}}^{FE} = \arg\min_{\mathcal{M} \in \mathfrak{M}}
\left\{\min_{\mathcal{M}^\prime \in \mathfrak{M}} \hat{R}(\hat{\tau}_{\mathcal{M}},
\hat{P}_{\mathcal{M}^\prime})\right\}
\end{equation}
with a corresponding BMD estimator of
\begin{equation}
\label{focused estimator}
\hat{\tau}^{FE}(q) =
\tau_{\hat{\mathcal{M}}^{FE}}(\hat{\theta}_{\hat{\mathcal{M}}^{FE}};q).
\end{equation}

We now consider two special model classes commonly employed in benchmark
analysis \citep{ShaSma11}:
the logistic model class and the multi-stage model class. We will utilize
these model classes in our illustration and in the computer simulations.

\subsection{Logistic Model Class}

The logistic model class of order $p$ has the dose-response function given by
\begin{equation}
\label{logistic model class}
\pi(d;\mathbf{\beta}) = \frac{\exp\{\mathbf{d}\trp\mathbf{\beta}\}}
{1 + \exp\{\mathbf{d}\trp\mathbf{\beta}\}}
\end{equation}
where
\begin{equation}
\label{dvec}
\mathbf{d} = (1, d, \ldots, d^p)\trp
\quad \mbox{and} \quad
\mathbf{\beta} = (\beta_0, \beta_1, \ldots, \beta_p)\trp.
\end{equation}
The parameter $\mathbf{\beta}$ takes values in $\Theta = \Re^{p+1}$.
Let us also define the matrix
\begin{equation}
\label{Dmat}
\mathbf{D} = \mathbf{d}^{\otimes 2} = \mathbf{d}\mathbf{d}\trp.
\end{equation}
For this logistic dose-response function, we routinely find that
\begin{displaymath}
\nabla \pi(d;\mathbf{\beta}) = \pi(d;\beta) [1 - \pi(d;\beta)] \mathbf{d}
\end{displaymath}
(note that the parameter here is labeled $\beta$ instead of $\theta$
so the operators
$\nabla$ and $\nabla^{\otimes 2}$ are with respect to $\beta$)
and
\begin{displaymath}
-\nabla^{\otimes 2} \pi(d;\mathbf{\beta}) =
\pi(d;\mathbf{\beta})[1-\pi(d;\mathbf{\beta}))][2\pi(d;\mathbf{\beta})-1]
  \mathbf{D}.
\end{displaymath}
Using these
quantities, we achieve the following simplified expressions,
where $\mathbf{d}_j$ and $\mathbf{D}_j$ are
given in (\ref{dvec}) and
(\ref{Dmat}) with $d$ replaced by $d_j$:

\begin{prop}
\label{prop: simplifications for logistic}
Under the logistic model class $\mathcal{M}$ of order $p$,
\begin{eqnarray*}
A^{(n)}_{\mathcal{M}}(\beta;\tilde{P}) & = & \sum_{j=1}^J N_j
\pi(d_j;\mathbf{\beta})[1-\pi(d_j;\mathbf{\beta})] \mathbf{D}_j; \\
\Sigma^{(n)}_{\mathcal{M}}(\mathbf{\beta};\tilde{P}) & = &
\sum_{j=1}^J N_j \tilde{\pi}(d_j) [1 - \tilde{\pi}(d_j)] \mathbf{D}_j.
\end{eqnarray*}
\end{prop}

\begin{proof}
These expressions follow easily after simplifications.
\end{proof}

In addition, for this logistic model class, we also obtain an explicit form of
the gradient vector function of the BMD function. This is given below, where
for any $a \in [0,1]$, we write $\bar{a} = 1- a$.

\begin{prop}
\label{prop: gradient of BMD for logistic}
Under the logistic model class $\mathcal{M}$ of order $p$,
\begin{displaymath}
\stackrel{\bullet}{\tau}(\mathbf{\beta};q) = \frac{
\bar{q} \left[\begin{array}{c}1\\ 0\\ \vdots\\ 0\end{array}\right]\pi(0;\mathbf{\beta})
  \bar{\pi}(0;\mathbf{\beta})
-
\left[\begin{array}{c}1\\ \tau(\mathbf{\beta};q)\\ \vdots\\ \tau(\mathbf{\beta};q)^p\end{array}\right]
\pi(\tau(\mathbf{\beta};q);\mathbf{\beta}) \bar{\pi}(\tau(\mathbf{\beta};q);\mathbf{\beta})}
{\left(\sum_{j=1}^p \beta_j j \tau(\mathbf{\beta};q)^{j-1}\right)
\pi(\tau(\mathbf{\beta};q);\mathbf{\beta}) \bar{\pi}(\tau(\mathbf{\beta};q);\mathbf{\beta})}
\end{displaymath}
where
\begin{displaymath}
\pi(0;\mathbf{\beta}) = \frac{\exp(\beta_0)}{1 + \exp(\beta_0)}
\quad \mbox{and} \quad
\pi(\tau(\mathbf{\beta};q);\mathbf{\beta}) = q + \bar{q} \pi(0;\mathbf{\beta}).
\end{displaymath}
\end{prop}

\begin{proof}
This follows from Proposition \ref{prop: gradient of tau} since under the
logistic model we have
\begin{displaymath}
\pi_{01}(d;\beta) = \nabla \pi(d;\beta) = \mathbf{d} \pi(d;\beta)
\bar{\pi}(d;\beta)
\end{displaymath}
and
\begin{displaymath}
\pi_{10}(d;\beta) = \frac{\partial}{\partial d} \pi(d;\beta) =
\left[\sum_{j=1}^p \beta_j j d^{j-1}\right] \pi(d;\beta)
\bar{\pi}(d;\beta).
\end{displaymath}
\end{proof}

\subsection{Multi-Stage Model Class}

The multistage model class of order $p$ is characterized by the dose-response
function given by
\begin{equation}
\label{multi-stage dose response}
\pi(d;\mathbf{\beta}) = 1 - \exp\{-\mathbf{d}\trp\mathbf{\beta}\}
\end{equation}
where the parameter space for $\mathbf{\beta}$ is $\Theta = \{\mathbf{\beta}
\in \Re^{p+1}: \mathbf{d}_j\trp\mathbf{\beta} \ge 0, j=1,2,\ldots,J\}$. Then,
it follows easily that
\begin{displaymath}
\nabla \pi(d;\mathbf{\beta}) = \bar{\pi}(d;\mathbf{\beta}) \mathbf{d}
\quad \mbox{and} \quad
-\nabla^{\otimes 2} \pi(d;\mathbf{\beta}) =
\bar{\pi}(d;\mathbf{\beta}) \mathbf{D}.
\end{displaymath}

\begin{prop}
\label{prop: simplifications for multistage}
Under the multistage model class $\mathcal{M}$ of order $p$,
\begin{eqnarray*}
A^{(n)}_{\mathcal{M}}(\beta;\tilde{P}) & = & \sum_{j=1}^J N_j
\frac{\tilde{\pi}(d_j) \bar{\pi}(d_j;\mathbf{\beta})}{\pi(d_j;\mathbf{\beta})^2} \mathbf{D}_j; \\
\Sigma^{(n)}_{\mathcal{M}}(\mathbf{\beta};\tilde{P}) & = &
\sum_{j=1}^J N_j \frac{\tilde{\pi}(d_j) [1- \tilde{\pi}(d_j)]}{\pi(d_j;\mathbf{\beta})^2} \mathbf{D}_j.
\end{eqnarray*}
\end{prop}

\begin{proof}
These expressions follow easily after simplification.
\end{proof}

And, finally, we also have for this multistage model class:

\begin{prop}
\label{prop: gradient of BMD for multistage}
Under the multistage model class $\mathcal{M}$ of order $p$,
\begin{displaymath}
\stackrel{\bullet}{\tau}(\mathbf{\beta};q) = \frac{
\bar{q} \left[\begin{array}{c}1\\ 0\\ \vdots\\ 0\end{array}\right] \bar{\pi}(0;\mathbf{\beta})
-
\left[\begin{array}{c}1\\ \tau(\mathbf{\beta};q)\\ \vdots\\ \tau(\mathbf{\beta};q)^p\end{array}\right]
\pi(\tau(\mathbf{\beta};q);\mathbf{\beta})}
{\left(\sum_{j=1}^p \beta_j j \tau(\mathbf{\beta};q)^{j-1}\right)
\bar{\pi}(\tau(\mathbf{\beta};q);\mathbf{\beta})}
\end{displaymath}
where
\begin{displaymath}
\bar{\pi}(0;\mathbf{\beta}) = \exp(-\beta_0)
\quad \mbox{and} \quad
{\pi}(\tau(\mathbf{\beta};q);\mathbf{\beta}) = 1 - \bar{q} \exp(-\beta_0).
\end{displaymath}
\end{prop}

\begin{proof}
Again, this is straightforward following easily from Proposition
\ref{prop: gradient of tau} and the fact that under the multistage model we
have
\begin{displaymath}
\pi_{01}(d;\beta) = \mathbf{d}\bar{\pi}(d;\beta)
\quad \mbox{and} \quad
\pi_{10}(d;\beta) = \bar{\pi}(d;\beta) \left[\sum_{j=1}^p \beta_j j
  d^{j-1}\right].
\end{displaymath}
\end{proof}

\section{Example: Nasal Carcinogenicity in Laboratory Rodents}
\label{sec: Application}

We illustrate the methods discussed in the preceding sections via a toxicological
dose-response data set studied by \cite{NitPieWes07} and originally described in \cite{KusLasDrew75}. 
The data represent occurrences of respiratory tract tumors in male rats after inhalation exposure to
the industrial compound bis(chloromethyl)ether (BCME), a chloroalkyl ether known to have toxic respiratory effects in mammals.
The quantal-response data are reproduced in Table \ref{table: real data}.
In applying the procedures, instead of using the actual concentrations (in ppm) we standardize
the values to range from 0 to 1, obtained by dividing the
original concentrations by 100. This is the second row in
the table which will serve as our $d_j$'s. There are seven concentration levels in this
data set (including the zero-concentration control).

\begin{table}[h]
\caption{Quantal response data from a carcinogenicity experiment of bis(chloromethyl)ether (BCME).
`Orig Conc' indicates original exposure concentrations (ppm), while `Std Conc' indicates standardized concentrations.}
\label{table: real data}
\begin{center}
\begin{tabular}{|l||c |c |c |c |c |c |c||} \hline
Orig Conc & 0 & 10 & 20 & 40 & 60 & 80 & 100\\\hline
Std Conc ($d_j$) &  0 & .1 & .2 & .4 & .6 & .8 & 1.0\\ \hline
Subjects ($N_j$) & 240 & 41 & 46 & 18 & 18 & 34 & 20 \\\hline
Events ($Y_j$) & 0 & 1 & 3 & 4 & 4 & 15 & 12 \\ \hline
\end{tabular}
\end{center}
\end{table}

For purposes of illustration, we consider the problem of estimating the BMD
with these data at the standard BMR values of $q = 0.01, 0.05, 0.10$ \citep{USEPA12}.
We place under consideration
the logistic model classes of orders $p = 1, 2$, referred to
respectively as Models LG1 and LG2, and the multistage model
classes of order $p =1, 2$, referred to as Models MS1 and MS2.
The model selectors considered are described
in Table \ref{table: model selectors}, while the BMD estimators considered are
described in Table \ref{table: BMD Estimators}. All procedures were
implemented using {\tt R}, with the likelihood maximization for the logistic model
classes performed using the object function {\tt optim}, while the likelihood
maximization under the multistage model classes were performed using the
object function {\tt optimConstr}. Finding the zeroes of a function was
performed by the object function {\tt uniroot}. The PAVA was implemented using the object
function {\tt pava}, which is contained in the {\tt Iso} package in {\tt R}.

\begin{table}
\caption{Description of Model Selectors Considered in Illustration and in
  Simulations.}
\label{table: model selectors}
\begin{center}
\begin{tabular}{|c|l|} \hline
Model Selector Label & Description \\ \hline
FIC1 & Model selector given by $\hat{\mathcal{M}}^{FE}$ in (\ref{focused estimator selector}).
 \\
FIC2 & Model selector given by $\hat{\mathcal{M}}^{FM}$ in (\ref{focused model selector}). \\
FIC3 & Model selector given by $\hat{\mathcal{M}}^{EMP}$ in (\ref{empirical model selector}). \\
AIC & Akaike information criteria based model selector. \\
BIC & Bayesian information criteria based model selector. \\ \hline
\end{tabular}
\end{center}
\end{table}

\begin{table}
\caption{Description of BMD Model Estimators Considered in Illustration and in
  Simulations.}
\label{table: BMD Estimators}
\begin{center}
\begin{tabular}{|c|l|} \hline
BMD Estimator Label & Description \\ \hline
FIC1 & BMD estimator given by $\hat{\tau}^{FE}(q)$ in (\ref{focused estimator}).
 \\
FIC2 & BMD estimator given by $\hat{\tau}^{FM}(q)$ in (\ref{focused model selector estimator}). \\
FIC3 & BMD estimator given by $\hat{\tau}^{EMP}(q)$ in (\ref{BMD empirical estimator}). \\
AIC & MLE of BMD of the AIC-selected model class. \\
BIC & MLE of BMD of the BIC-selected model class. \\
AICModAve & AIC-based model averaged estimator of BMD. \\
BICModAve & BIC-based model averaged estimator of BMD. \\
NONPAR & BMD Estimator from the PAVA Estimator of $\pi(\cdot)$. \\ \hline
\end{tabular}
\end{center}
\end{table}

%

The model selected by each of the model selectors FIC1, FIC2, and FIC3 will
depend on the chosen BMR value, while the models
selected by AIC and BIC remain
independent of the
BMR. For the BCME carcinogenicity data in Table \ref{table: real data},
the models selected at the three different BMR values are given in
Table \ref{table: example model selected}.

\begin{table}
\caption{Model class selected by each of the five model selectors for
  different values of the BMR for the carcinogencity data set. {\em Legend:} LG1 = Logistic order 1; LG2 = Logistic
  order 2; MS1 = Multistage order 1;
MS2 = Multistage order 2.}
\label{table: example model selected}
\begin{center}
\begin{tabular}{|c|c|c|c|} \hline
 Model Selector &    $\mbox{BMR} =0.01$ & \mbox{BMR} = 0.05  & $\mbox{BMR} =0.1$  \\ \hline
FIC1  &    MS2 & MS2 & MS1  \\
FIC2   &   MS2 & MS2 & LG1 \\
FIC3 &     LG2 & MS1 & MS1 \\
AIC  &    MS2 & MS2 & MS2 \\
BIC   &   MS1 & MS1 & MS1 \\ \hline
\end{tabular}
\end{center}
\end{table}

The interesting aspect about the model selectors FIC1, FIC2, and FIC3 is that
they adapt to the BMR value.
Figure \ref{figure: example estimated probabilities} presents the estimated
dose-response functions evaluated at $d_j, j=1,2,\ldots, 7$,
obtained by each of the model selectors for the two BMR values. We have also
included in these plots the empirical probabilities
and the PAVA estimates, though the empirical probability estimates are masked 
by the PAVA estimates since for this data set these two sets of estimates are identical.
\begin{figure}
\caption{Connected lines of the estimated dose function evaluated at the standardized concentrations 
for the carcinogenecity data set provided by the different model
  selectors for BMR values of $0.01$ (top) and $0.10$ (bottom). Also included are the empirical
  probabilities and the PAVA estimates.}
\label{figure: example estimated probabilities}
\begin{center}
\begin{tabular}{c}
\includegraphics[width=3in,height=\textwidth,angle=-90,trim=15mm 0mm 0mm 0mm]{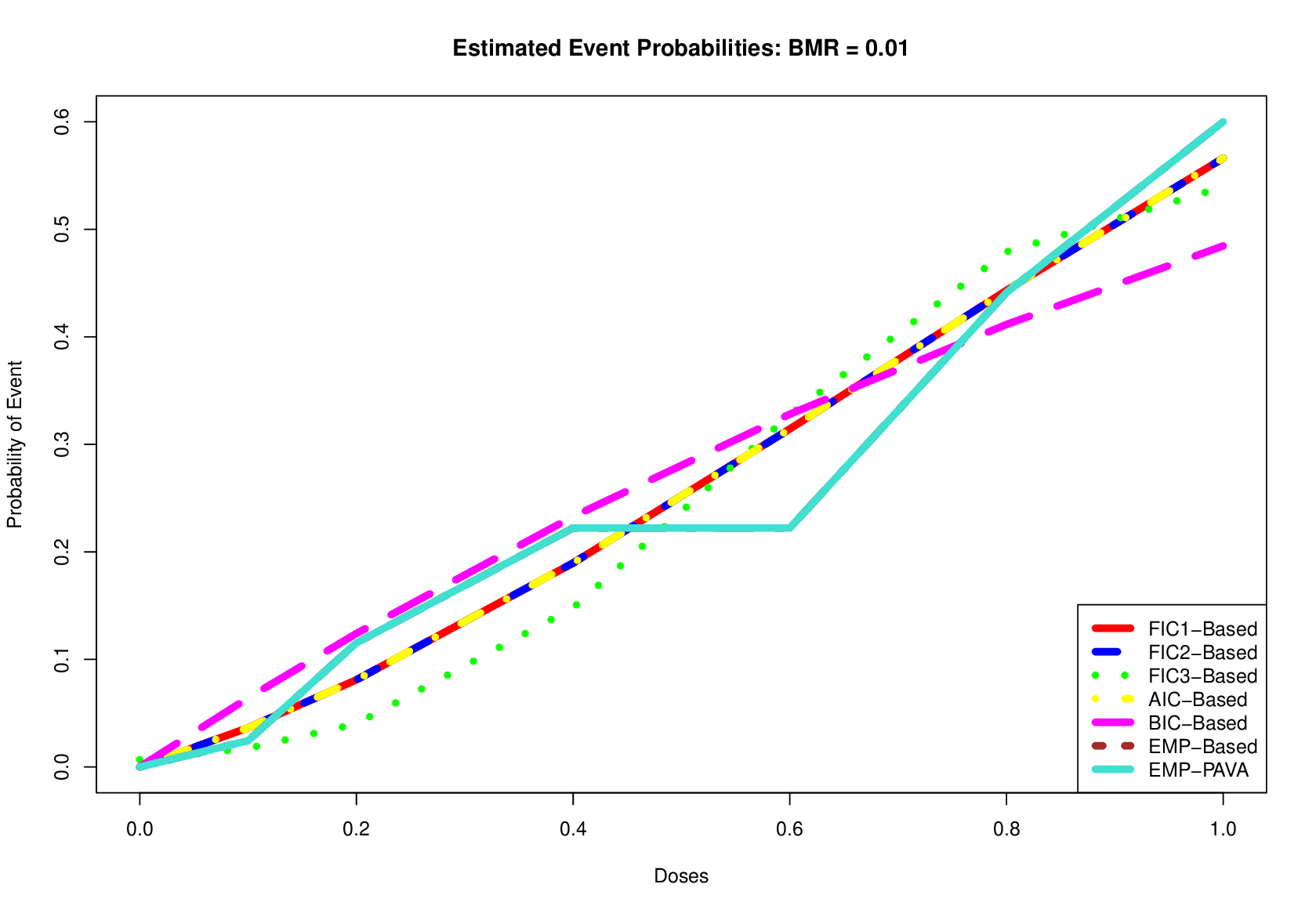} \\
\includegraphics[width=3in,height=\textwidth,angle=-90,trim=15mm 0mm 0mm 0mm]{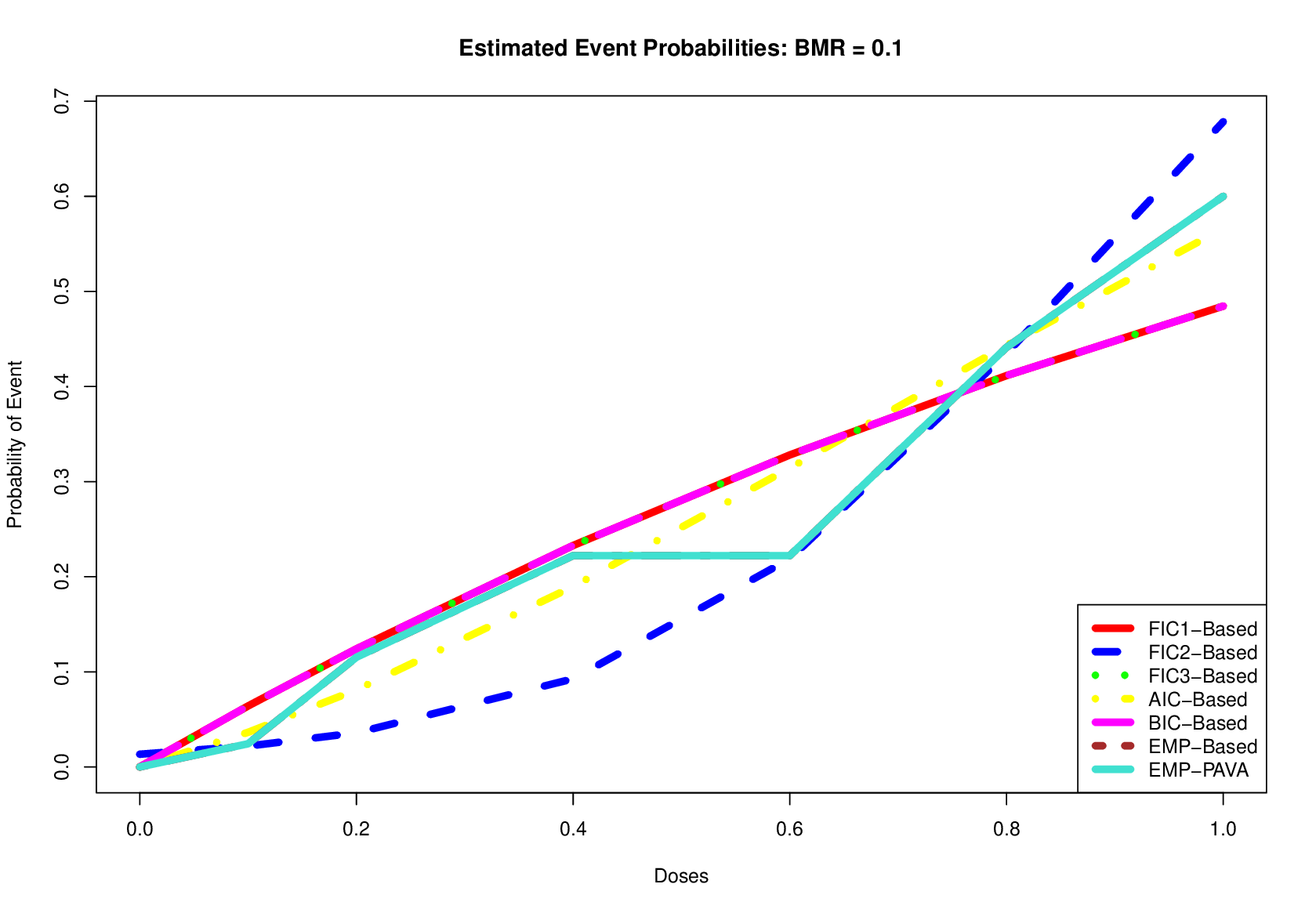}
\end{tabular}
\end{center}
\end{figure}

Table \ref{table: example BMD estimates} provides the BMD estimates,
under our standardized concentration scale,
provided by the eight estimators described in Table \ref{table: BMD Estimators}.
As expected, the estimation procedures for the BMD provide lower estimates
at smaller BMRs. Note that the BMD estimates at BMR = .01 are close to each other
except for that provided by the FIC3 method; while at BMR = .10 the estimate
provided by FIC2 is drastically different from those of the other estimates. This is
tied-in to the fact that at BMR = .10, the chosen model by FIC2, which is LG1, is
highly different from the models chosen by the other methods; see the second
panel in Figure \ref{figure: example estimated probabilities}.
In the next section, we compare the performance of each of these BMD estimators with
respect to their biases, standard errors, and root mean-square errors, under
different scenarios via a modest simulation study.

\begin{table}
\caption{BMD estimates (in standardized concentrations) obtained by the eight estimators for the carcinogenecity data
  set for three different BMR values.}
\label{table: example BMD estimates}
\begin{center}
\begin{tabular}{|l|c|c|c|} \hline
  Estimator &        BMR=0.01 & BMR=0.05 & BMR=0.1 \\ \hline
FIC1    &     0.030 &   0.132 &  0.159 \\
FIC2    &     0.030 &   0.132 &  0.442 \\
FIC3    &     0.095 &   0.077 &  0.159 \\
AIC     &     0.030 &   0.132 &  0.238 \\
BIC     &     0.015 &   0.077  & 0.159 \\
AICModAve  &  0.025 &   0.109 &  0.203 \\
BICModAve &   0.018 &   0.087 &  0.172 \\
NONPAR   &    0.041 &   0.128 &  0.183 \\ \hline
\end{tabular}
\end{center}
\end{table}

\section{Simulation Studies}
\label{sec: Simulations}

In order to compare the performance of the model selectors and BMD estimators
illustrated in the Example, we performed a short series of computer simulation
experiments. Our basic design for each simulation experiment
was to generate dose-response data for a given set of $(N_j,d_j)$ values, $j=1,2,\ldots,J$,
from a specific, true dose-response function $\pi(\cdot)$.  For this
data set, and over a range of BMR values, we obtained the selected model by each
model selector and
apply each BMD estimator just as in
the Example above. This process was replicated ${\tt MREPS = 2000}$ times.
For these {\tt MREPS} replications, we summarize the
performance of the model selector by tabulating the number of times that it
had chosen a given model class, and for the BMD estimators we obtain the mean, bias,
standard error, and root mean-squared error. Note that we are able to compute
the bias of each estimate since we know the exact BMD values
under the true model
for each BMR. The competing model classes
utilized in the simulation coincide with those in the data illustration of
Section \ref{sec: Application}.

\subsection{Simulation Experiment \#1}

For the first simulation experiment, we used for our true data-generating
model a multistage dose-response function
of order $p=2$ (Model MS2 in Sec.~\ref{sec: Application}) with $\beta$-coefficients
given by
\begin{displaymath}
\mathbf{\beta} = (\beta_0, \beta_1, \beta_2) = (0, 0.32, 0.52).
\end{displaymath}
This choice is motivated by the estimated $\beta$-coefficients of the
fitted model MS2 from the
BCME carcinogenicity data example above.
We also utilized the
vector of doses from these data, given by
$$(d_1,d_2,\ldots,d_7) = (0, 0.1, 0.2, 0.4, 0.6, 0.8, 1.0),$$
and vector of number of subjects, given by
$$(N_1,N_2,\ldots,N_7) = (240,  41,  46,  18,  18,  34, 20).$$
The true simulation
dose-response function is plotted in Figure \ref{fig: true response function expt 1}.

\begin{figure}
\caption{The true multistage dose-response function under Model MS2 for simulation experiment 1.}
\label{fig: true response function expt 1}
\includegraphics[width=3in,height=\textwidth,angle=-90]{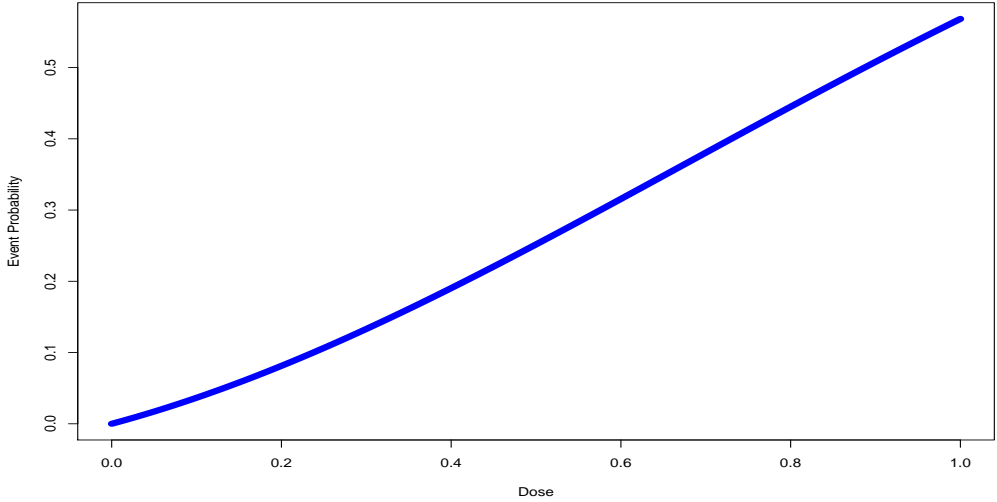}
\end{figure}

Figure \ref{fig: expt 1 boxplots} presents comparative boxplots of the BMD
estimates obtained by the eight BMD estimators for BMRs of $0.01$ and $0.10$.
Figures \ref{fig: mean, bias expt 1}  and \ref{fig: se, rmse expt 1} present the
corresponding
bias and root mean-squared error
plots
for BMRs of $0.01$ and $0.10$.
For additional resolution in the graphs, we also include results at
BMR = $0.05$.  Each of the curves are
plotted as a function of BMR.

From the results of this particular simulation study, we observe that
the BIC estimator underestimates the true BMD value, while the FIC3 estimator overestimates
the true BMD value. The AIC Model-Averaged estimator, as well as the
BIC Model-Averaged estimator, performed better than the others; with the
AIC, FIC1, and PAVA-based Nonparametric estimators having comparatively
mid-level performance.
The BIC, FIC2, and FIC3 did not fare well relative to the other estimators.

With respect to the model selectors, from  Table \ref{table: expt 1 model choices of five selectors},
we observe that the BIC model selector
hardly chose the correct model, though both the AIC and BIC model selectors
tended to choose the lower-order multistage model (MS1). The FIC1 model selector
did quite well, as well as the FIC2 model selector. The FIC3 model selector did
not also choose the correct model, and appeared undecided between the LG2 and MS1
models especially at BMR-values of 0.05 and 0.10. Note that the AIC and BIC selectors'
model choices do not vary with the BMR, whereas for the other three
selectors the model choices do depend on the BMR value under
consideration.

\begin{figure}
\caption{Comparative boxplots of the BMD estimates obtained by the eight
  different estimation schemes in simulation experiment \#1 for BMR values $0.01$ (top) and $0.1$ (bottom). The gray
  horizontal line is the true BMD for the given BMR value. (Estimation scheme labels are given in Table \ref{table: model selectors}.)}
\label{fig: expt 1 boxplots}
\begin{center}
\begin{tabular}{c}
\includegraphics[width=3in,height=\textwidth,angle=-90]{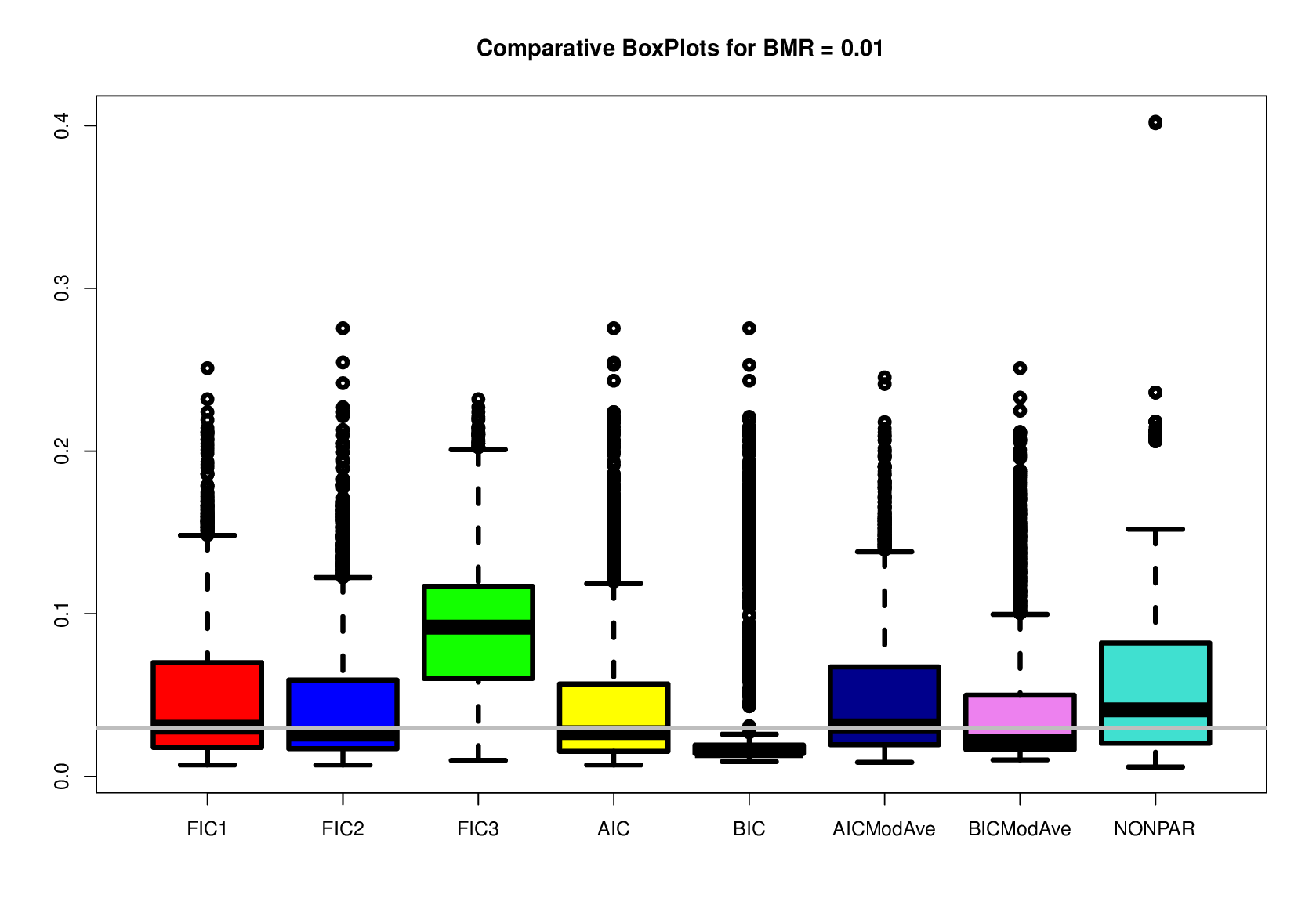} \\
\includegraphics[width=3in,height=\textwidth,angle=-90]{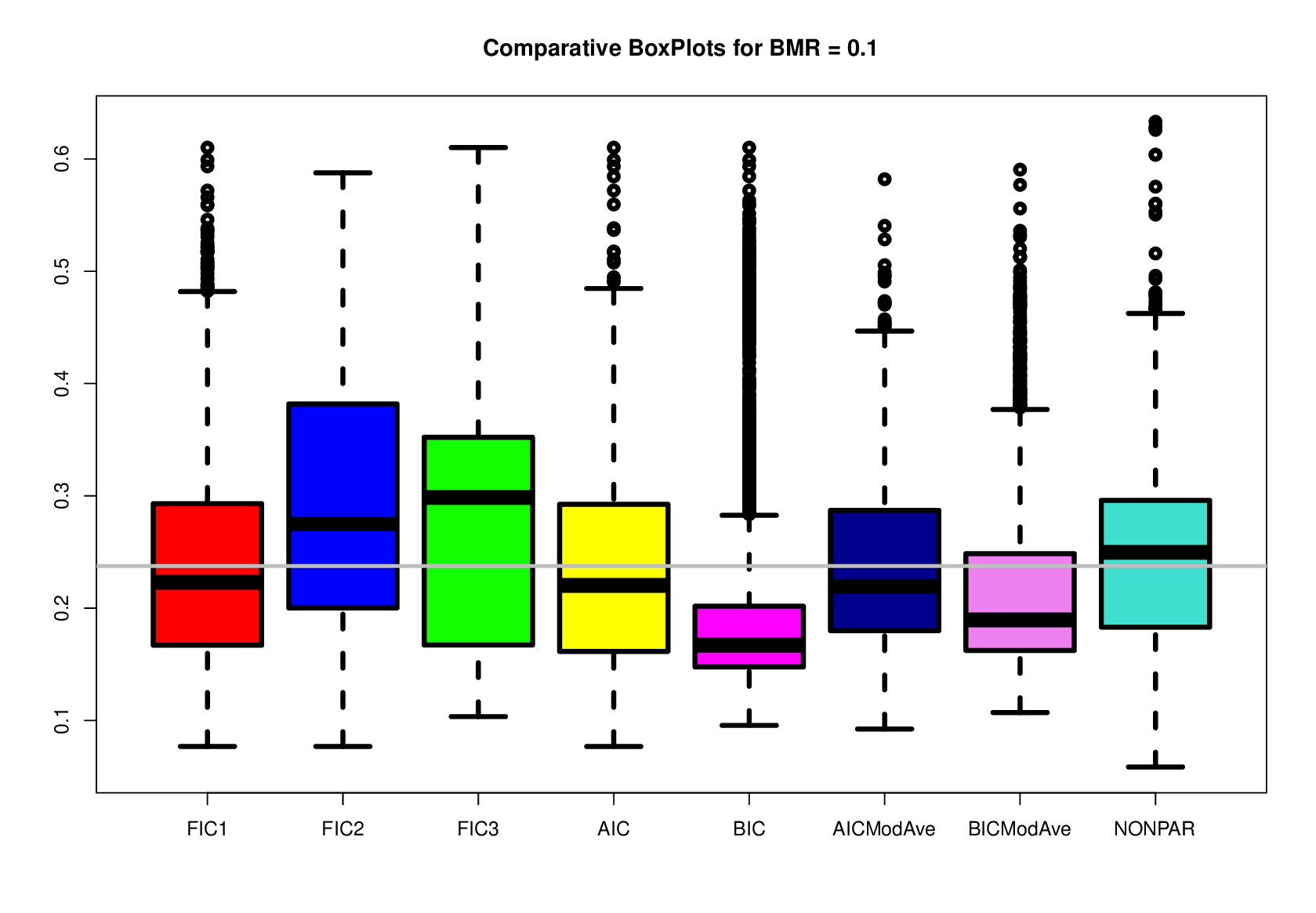}
\end{tabular}
\end{center}
\end{figure}

\begin{figure}
\caption{Bias plots of the eight BMD
  estimators for simulation experiment \#1 for BMR values $0.01, 0.05, 0.1$. For each BMR and each estimator, 
the bias is the average of the differences between the BMD estimates and the true BMD value over the $2000$ simulations.}
\label{fig: mean, bias expt 1}
\begin{center}
\begin{tabular}{c}
\includegraphics[width=3in,height=\textwidth,angle=-90]{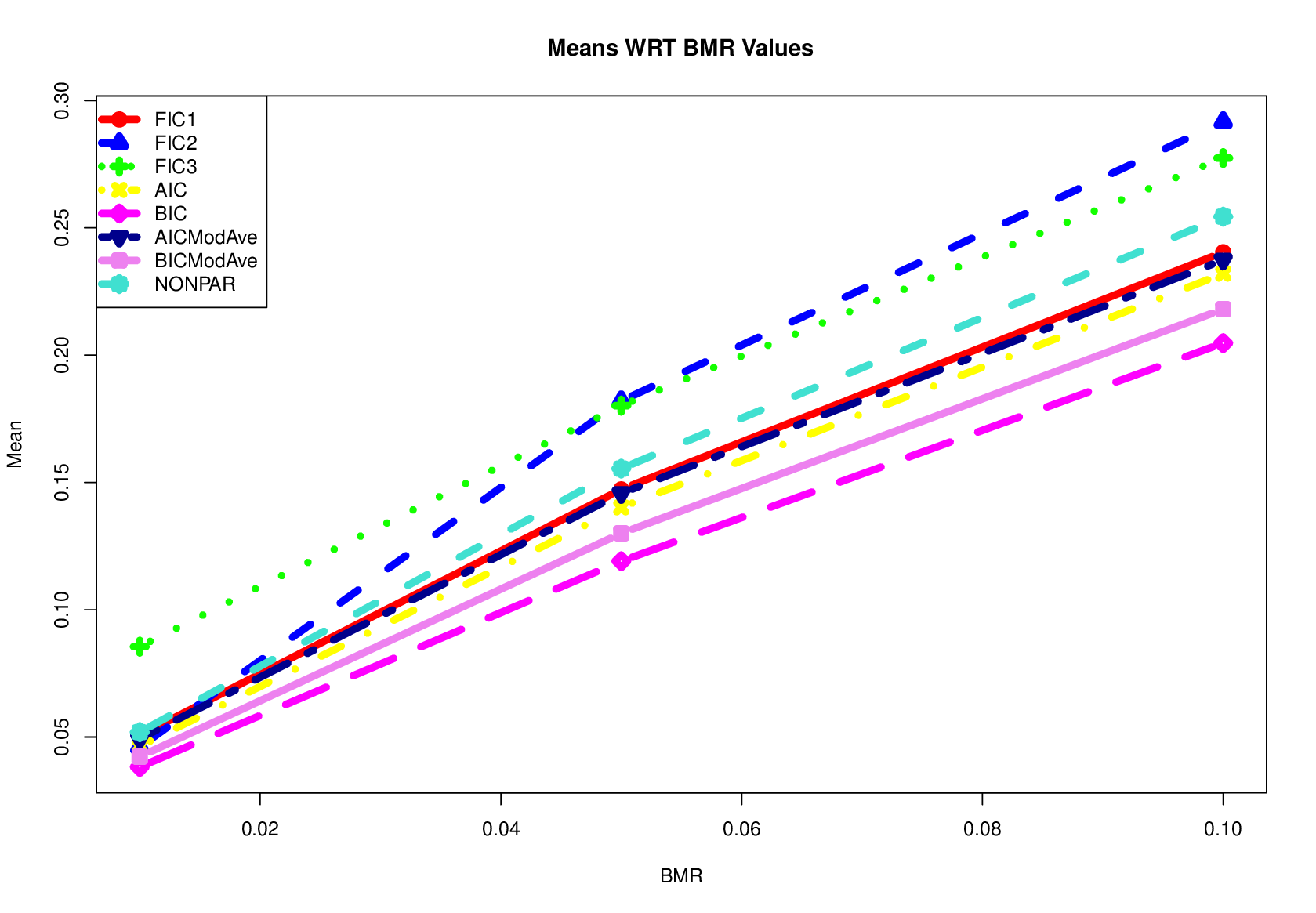} \\
\includegraphics[width=3in,height=\textwidth,angle=-90]{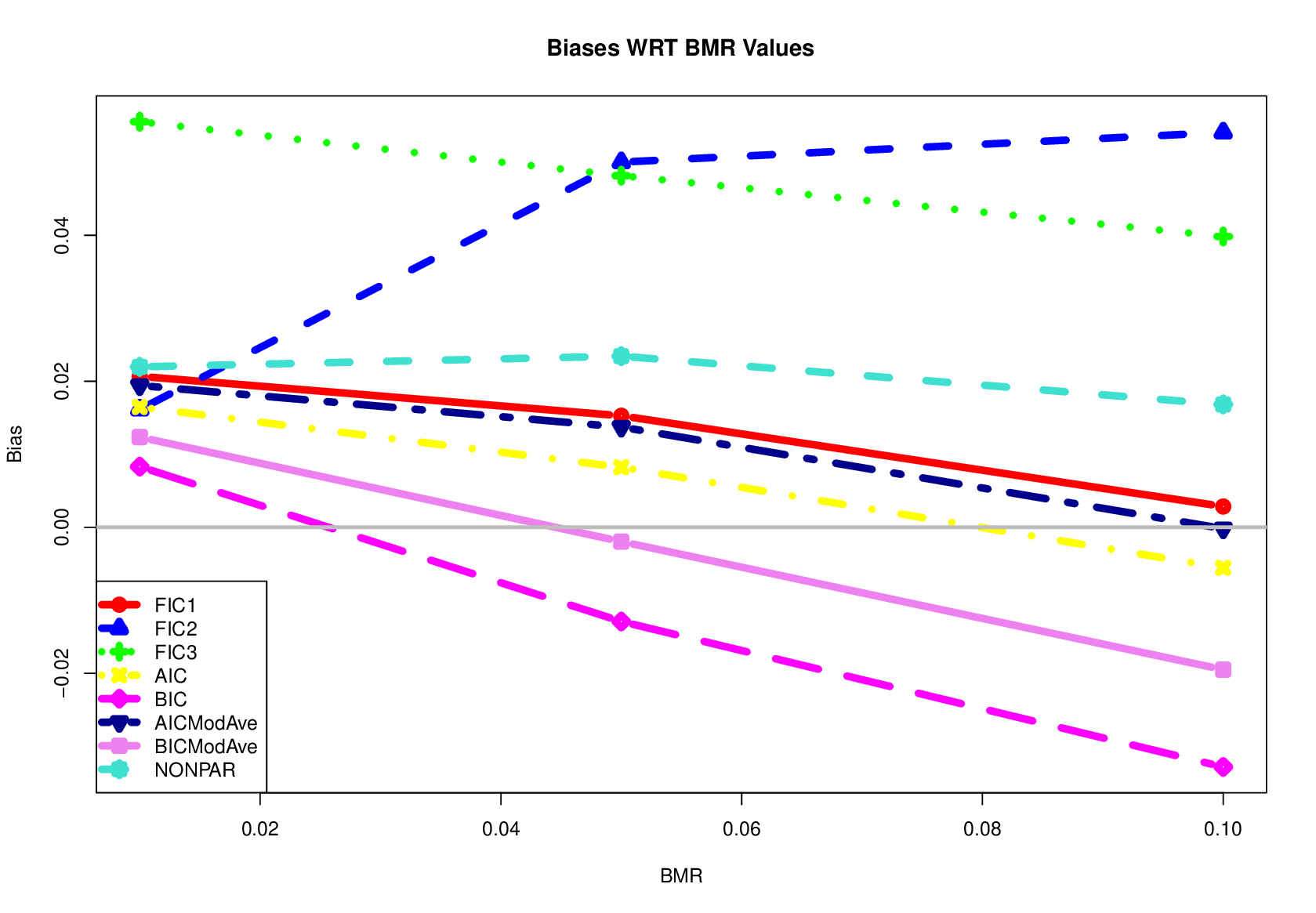}
\end{tabular}
\end{center}
\end{figure}
\begin{figure}
\caption{Root Mean-Square Error plots of the eight BMD
  estimators for simulation experiment \#1 for BMR values $0.01, 0.05, 0.1$. For each BMR and each estimator, 
the Root Mean-Square Error is the average of the squared differences between the BMD estimates and the true BMD value
over $2000$ simulations.}
\label{fig: se, rmse expt 1}
\begin{center}
\begin{tabular}{c}
\includegraphics[width=3in,height=\textwidth,angle=-90]{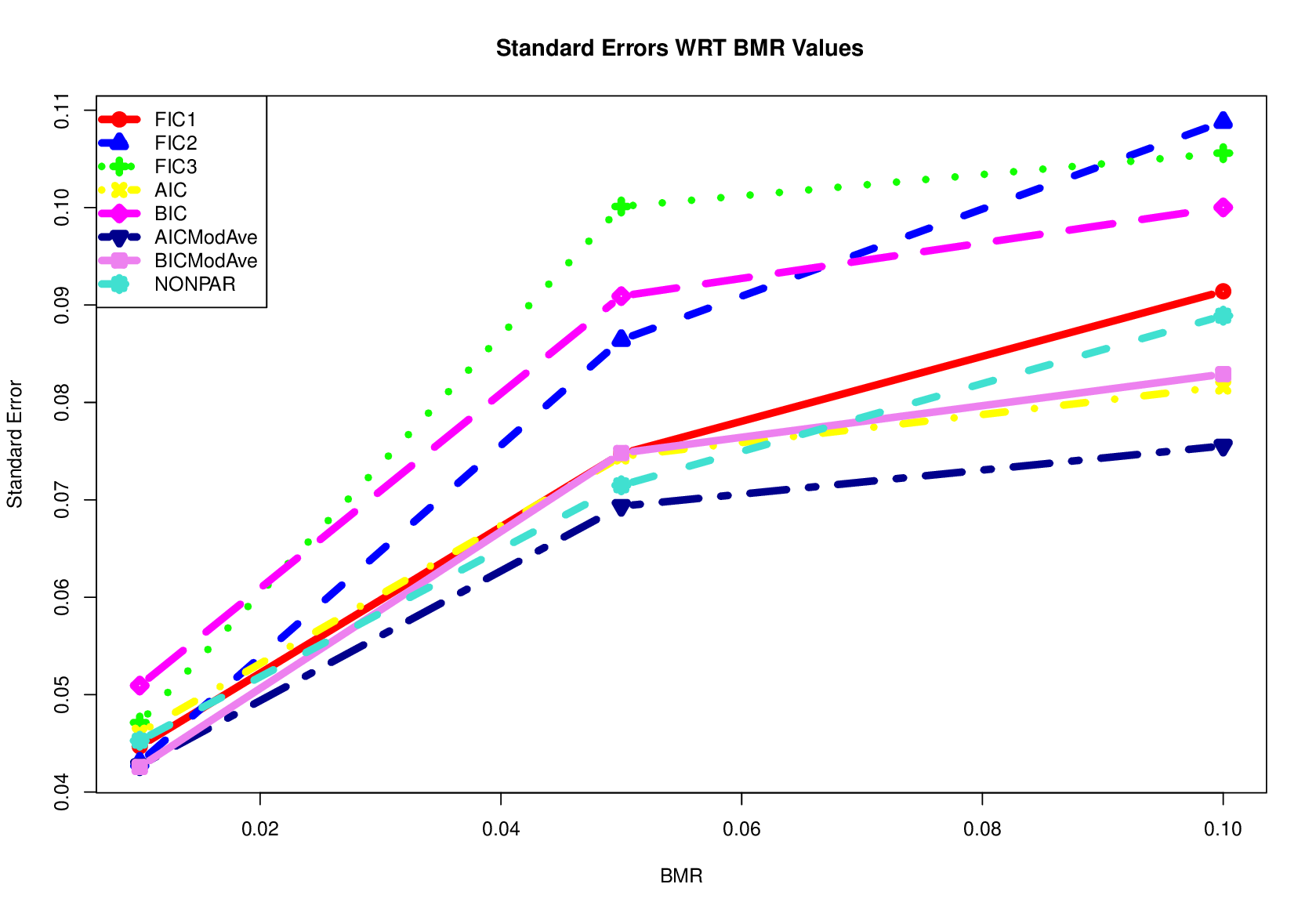} \\
\includegraphics[width=3in,height=\textwidth,angle=-90]{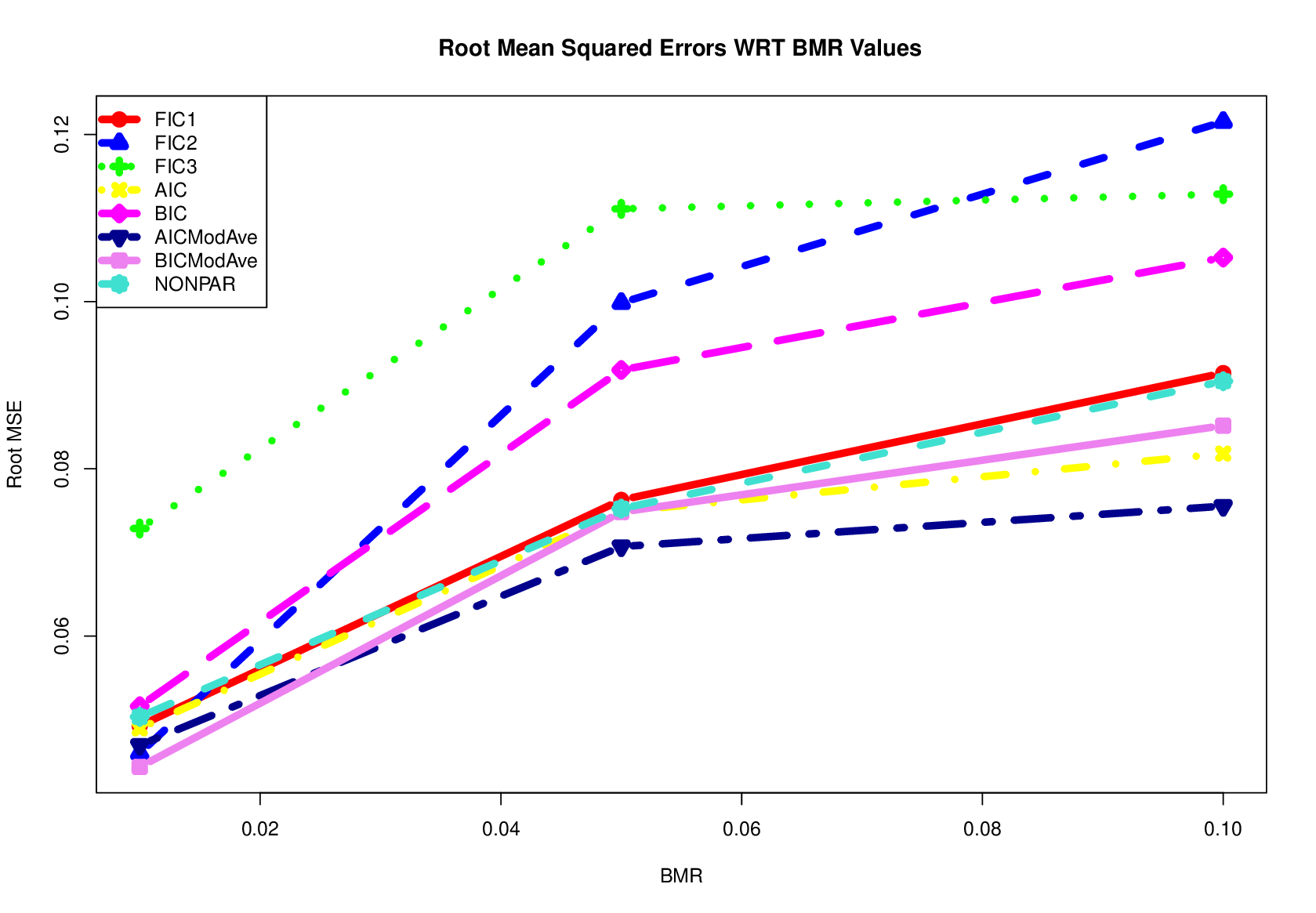}
\end{tabular}
\end{center}
\end{figure}

\begin{table}
\caption{Percentages of model selection by the five model selectors across the 2000 replications
  of Experiment \#1.
  A multistage model of order 2 ({\bf MS2}) is the true generating model, as described in Sec.~\ref{sec: Simulations}.}
\label{table: expt 1 model choices of five selectors}
\begin{center}
\begin{tabular}{|c||c||c|c|c|c|}
 \hline
Selector & BMR-Value & LG1 & {LG2} & MS1 & {\bf MS2}
\\ \hline\hline
     AIC & ALL   & 1.30  & 7.15   &  47.20   & {\bf 44.35}  \\
     BIC & ALL   & 7.55   & 3.60   &  81.00   & {\bf 7.85}  \\ \hline\hline
    FIC1 & 0.01   & 11.15   & 2.10   &  18.35   & {\bf 68.40}  \\
    FIC2 & 0.01   & 1.70   & 14.50   &  25.00   & {\bf 58.80}  \\
    FIC3 & 0.01   & 24.35   & 52.20   &  23.45  & {\bf 0}  \\ \hline\hline
    FIC1 & 0.05   & 2.80   & 5.55   &  22.15   & {\bf 69.50}  \\
    FIC2 & 0.05   & 9.55   & 27.90   &  10.10   & {\bf 52.45}  \\
    FIC3 & 0.05   & 2.80   & 51.55   &  45.00   & {\bf 0.65}  \\ \hline\hline
    FIC1 & 0.10   & 5.95   & 6.40   &  31.35   & {\bf 56.30}  \\
   FIC2 & 0.10   &  21.55   & 20.30   &  13.50   & {\bf 44.65}  \\
   FIC3 & 0.10   & 6.60   & 48.55   &  37.00   & {\bf 7.85}  \\ \hline\hline
\end{tabular}
\end{center}
\end{table}

\subsection{Additional Experiments}

We performed eight additional simulation experiments
using different dose-response functions
with varied shapes. The characteristics of each
true generating model are provided in Table \ref{table: true models for experiments 2 to 9}, where the true model parameters are determined by specifying the values of $\pi(\cdot)$ at two or three dose values. Notice that the true generating models in Experiments \#2 to \#5 have the same constraints on the smallest and largest doses, as do the models in Experiments \#6 to \#9. For each generating model we set $J=4$ doses with $(d_1,d_2,d_3,d_4)=(0,0.25,0.5,1)$ and $J=8$ doses with $(d_1,d_2,\ldots,d_8) = (0, 0.00625, 0.03125, 0.0625, 0.125, 0.25, 0.50, 1.00)$.
The four-dose setting corresponds to a popular design in cancer risk experimentation \citep{Port94}, while the eight-dose setting expands upon this geometric spacing to focus on doses closer to the origin.
Across the doses $d_j$, we took $N_j$ to be constant, i.e., $N_j = N$, and considered three different per-dose samples sizes:
$N = 50, 100, 1000$.
The number
of simulation replications remained ${\tt MREPS} = 2000$. For Experiments \#2 to \#9, we present only the RMSE plots vs. BMR
for $J=4$ doses and $N_j=100, j=1,2,\ldots,J$. These plots are given collectively in Figure \ref{fig: RMSE plots for experiment 2 to 9}. The percentages of selected models by all five model selectors for BMR values of $0.01$ and $0.1$ are presented in comparative bar plots given in Figure \ref{fig: model selection for expt 2 to 9}. Notice that the model selectors based on AIC and BIC
remain independent of BMR, while the model selectors based on FIC1, FIC2, FIC3 differ.

\begin{table}
\caption{Characteristics of true generating models for additional simulation experiments.}
\label{table: true models for experiments 2 to 9}
\begin{center}
\begin{tabular}{|c|c|c|c|} \hline
Experiment & Model & Order & Constraint on the  \\
Number & Class & $p$ & True Dose-Response Function $\pi(\cdot)$\\ \hline\hline
2 & Logistic & 1 & $\pi(0)=0.05$, $\pi(1)=0.50$ \\ \hline
3 & Logistic & 2 & $\pi(0)=0.05$, $\pi(0.5)=0.30$, $\pi(1)=0.50$ \\ \hline
4 & Multistage & 1 & $\pi(0)=0.05$, $\pi(1)=0.50$ \\ \hline
5 & Multistage & 2 & $\pi(0)=0.05$, $\pi(0.5)=0.30$, $\pi(1)=0.50$ \\ \hline
6 & Logistic & 1 & $\pi(0)=0.30$, $\pi(1)=0.75$ \\ \hline
7 & Logistic & 2 & $\pi(0)=0.30$, $\pi(0.5)=0.52$, $\pi(1)=0.75$ \\ \hline
8 & Multistage & 1 & $\pi(0)=0.30$, $\pi(1)=0.75$ \\ \hline
9 & Multistage & 2 & $\pi(0)=0.30$, $\pi(0.5)=0.52$, $\pi(1)=0.75$ \\ \hline
\end{tabular}
\end{center}
\end{table}

\newcommand{\mywidth}{1.5in}
\newcommand{\myheight}{2.6in}
\begin{figure}
\caption{Simulated RMSE plots of the eight BMD estimators at BMR values of $0.01, 0.05$ and $0.10$.
Per-dose sample size is fixed at $N = 100$ at each of $J = 4$ doses.
Plots from top left to bottom left and then top right to bottom right are for Experiments \#2 to \#9.}
\label{fig: RMSE plots for experiment 2 to 9}
\begin{center}
\begin{tabular}{cc}
\includegraphics[width=\mywidth,height=\myheight,angle=-90]{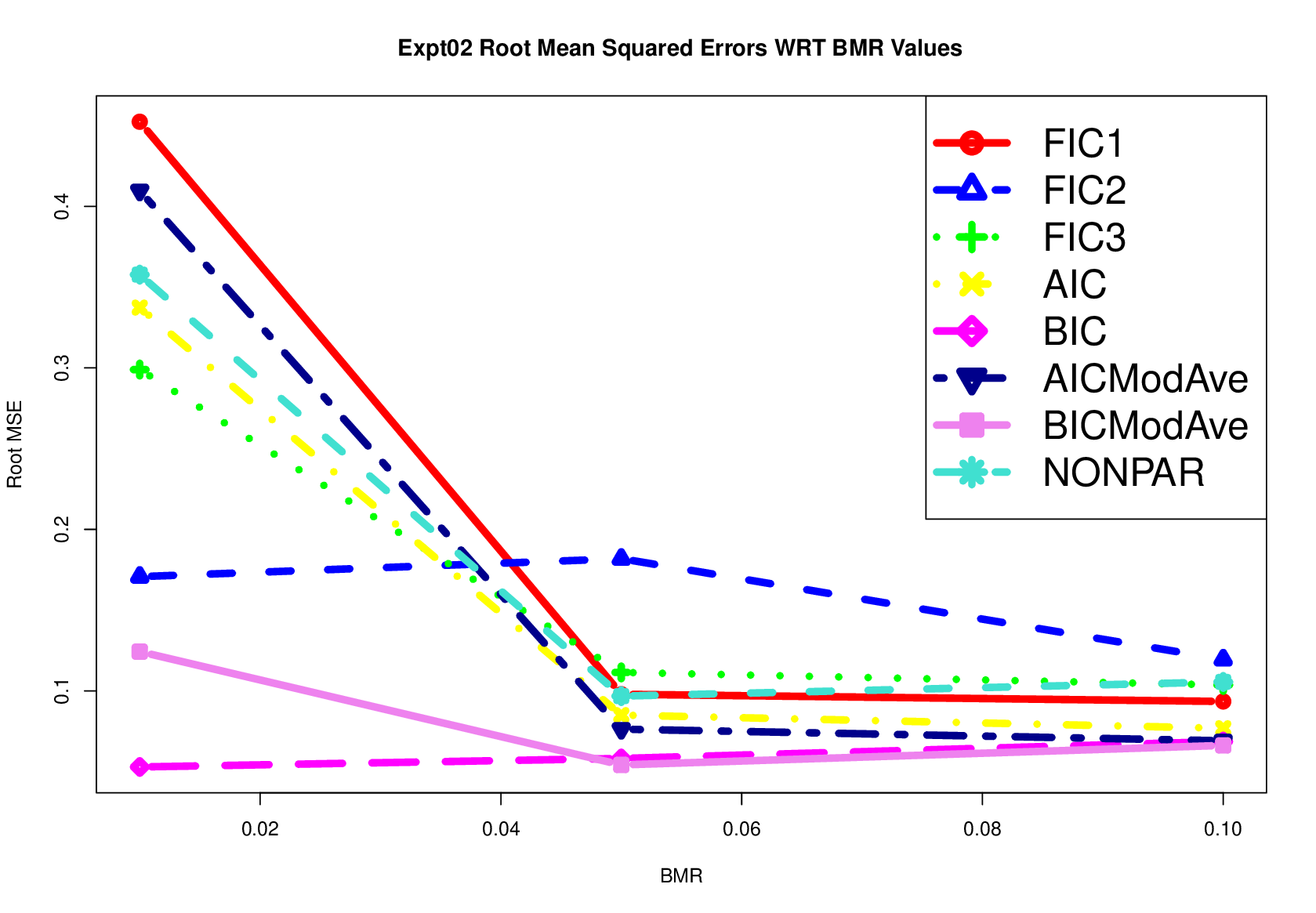} &
\includegraphics[width=\mywidth,height=\myheight,angle=-90]{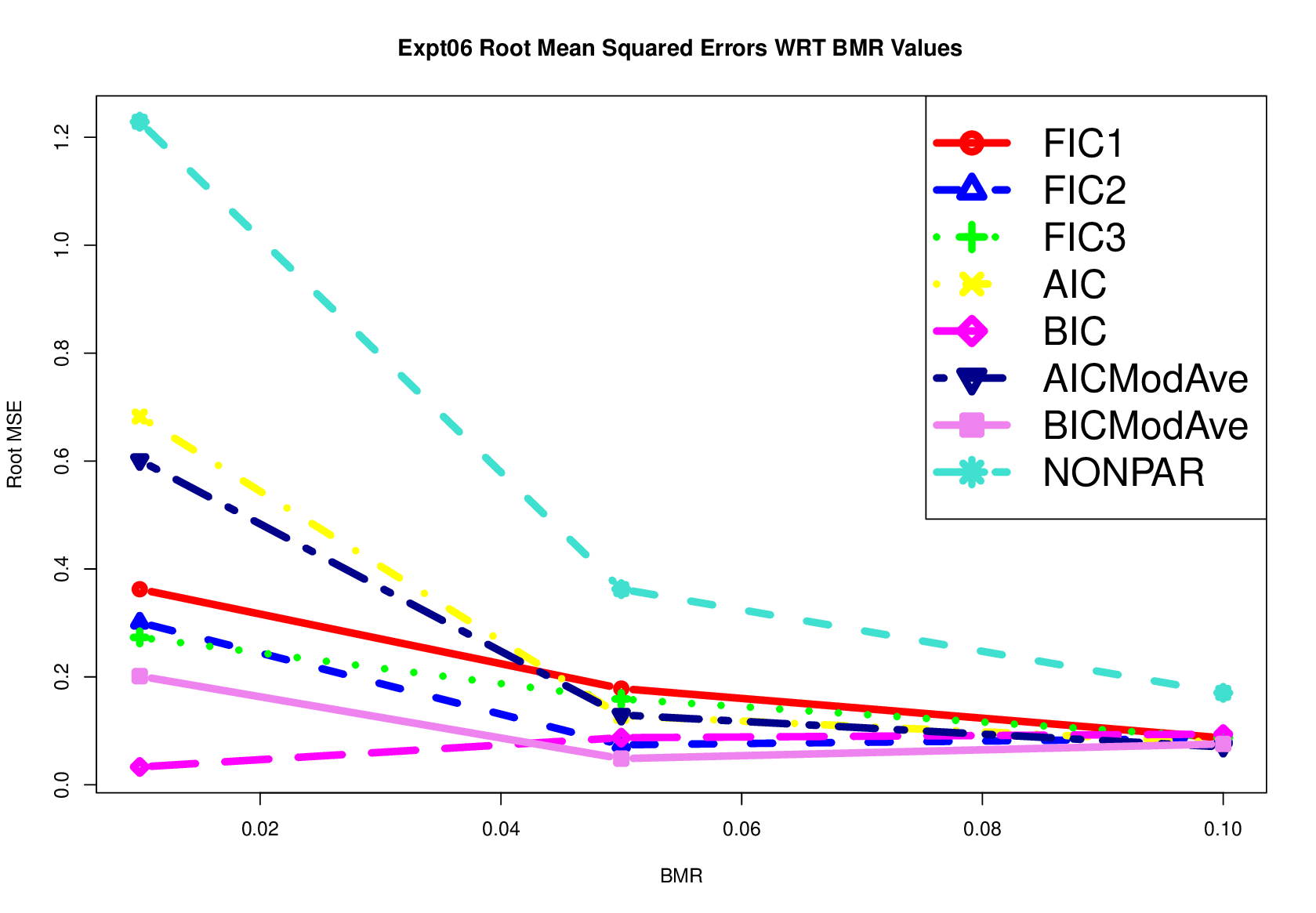} \\
\includegraphics[width=\mywidth,height=\myheight,angle=-90]{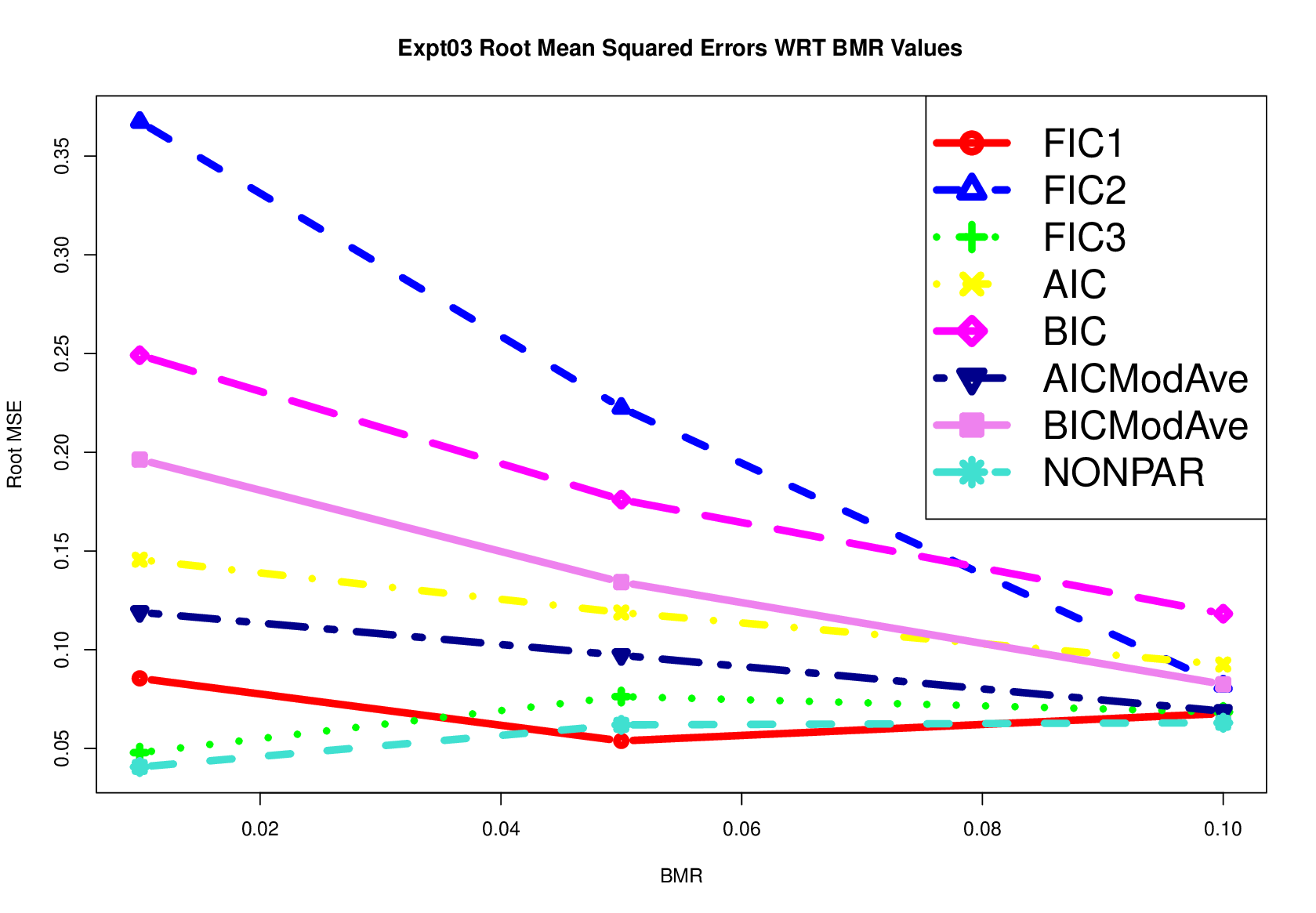} &
\includegraphics[width=\mywidth,height=\myheight,angle=-90]{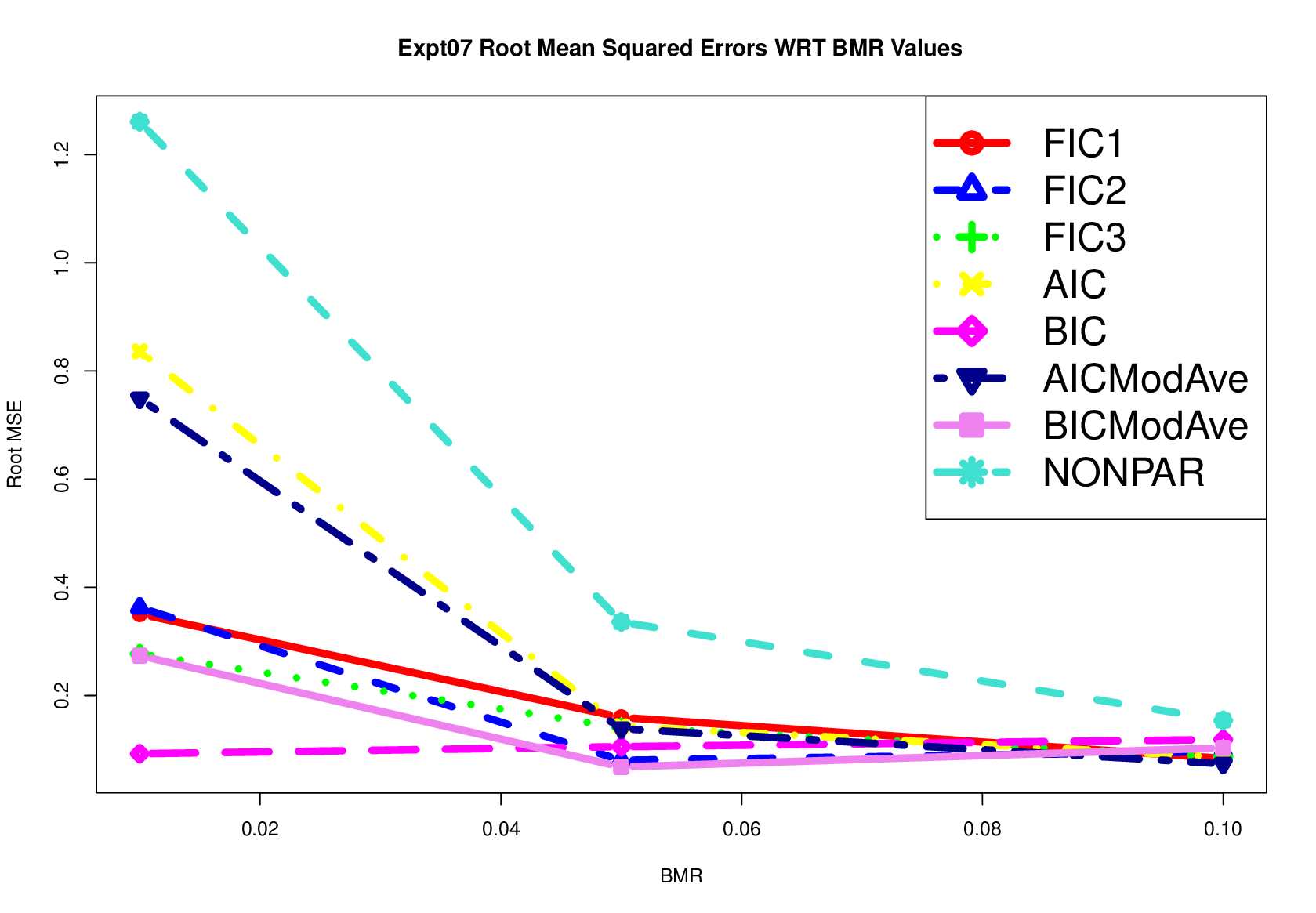} \\
\includegraphics[width=\mywidth,height=\myheight,angle=-90]{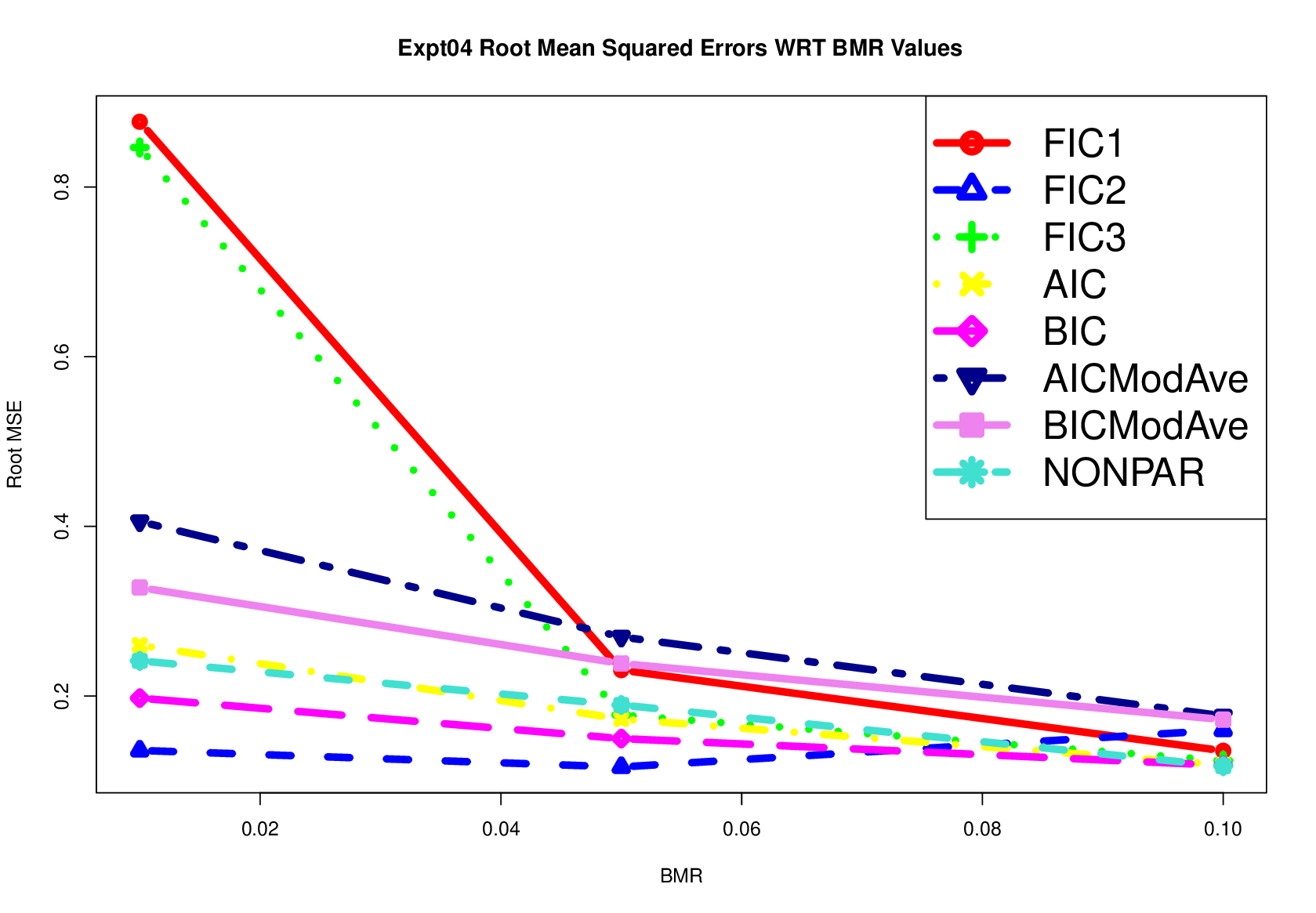} &
\includegraphics[width=\mywidth,height=\myheight,angle=-90]{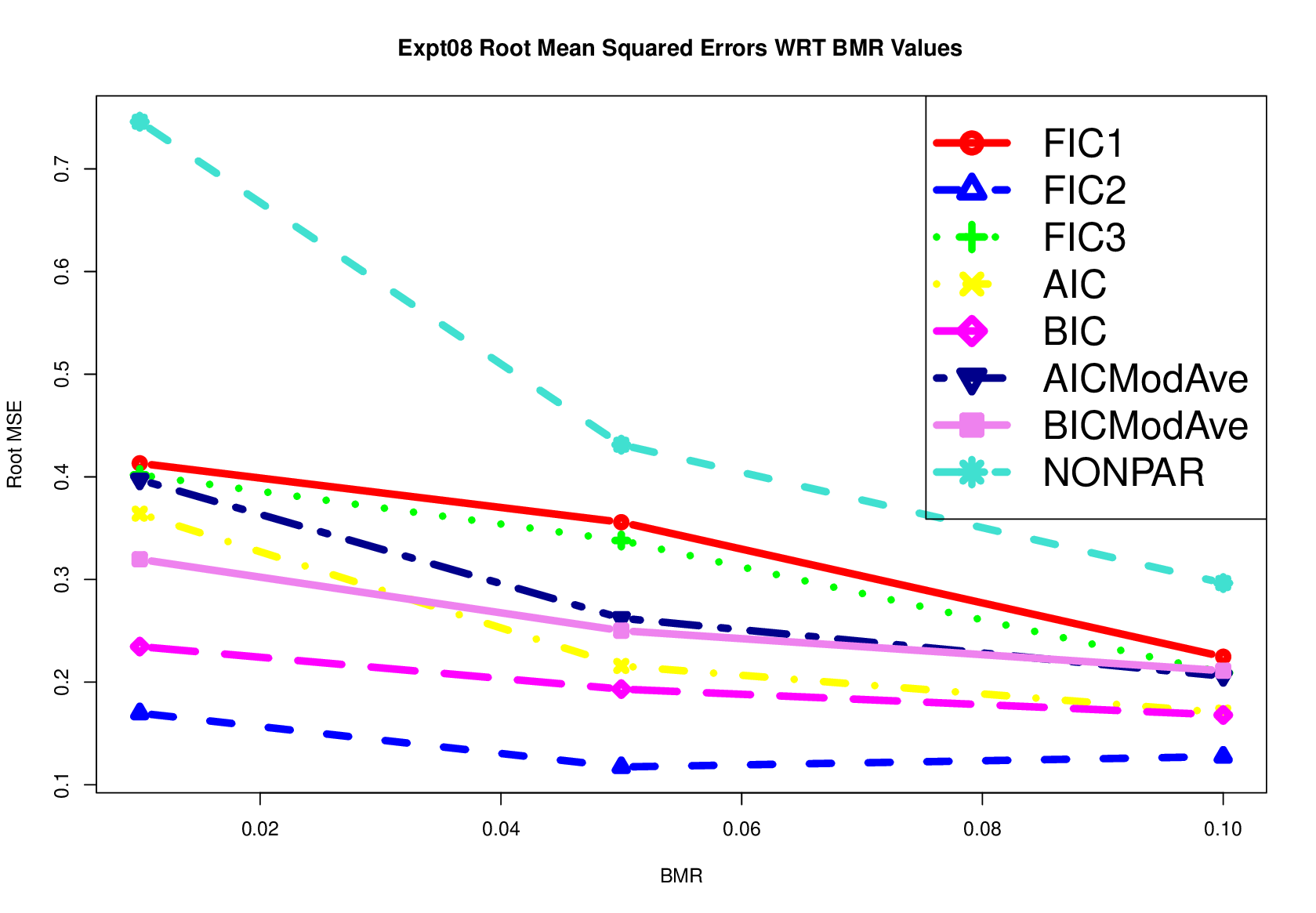} \\
\includegraphics[width=\mywidth,height=\myheight,angle=-90]{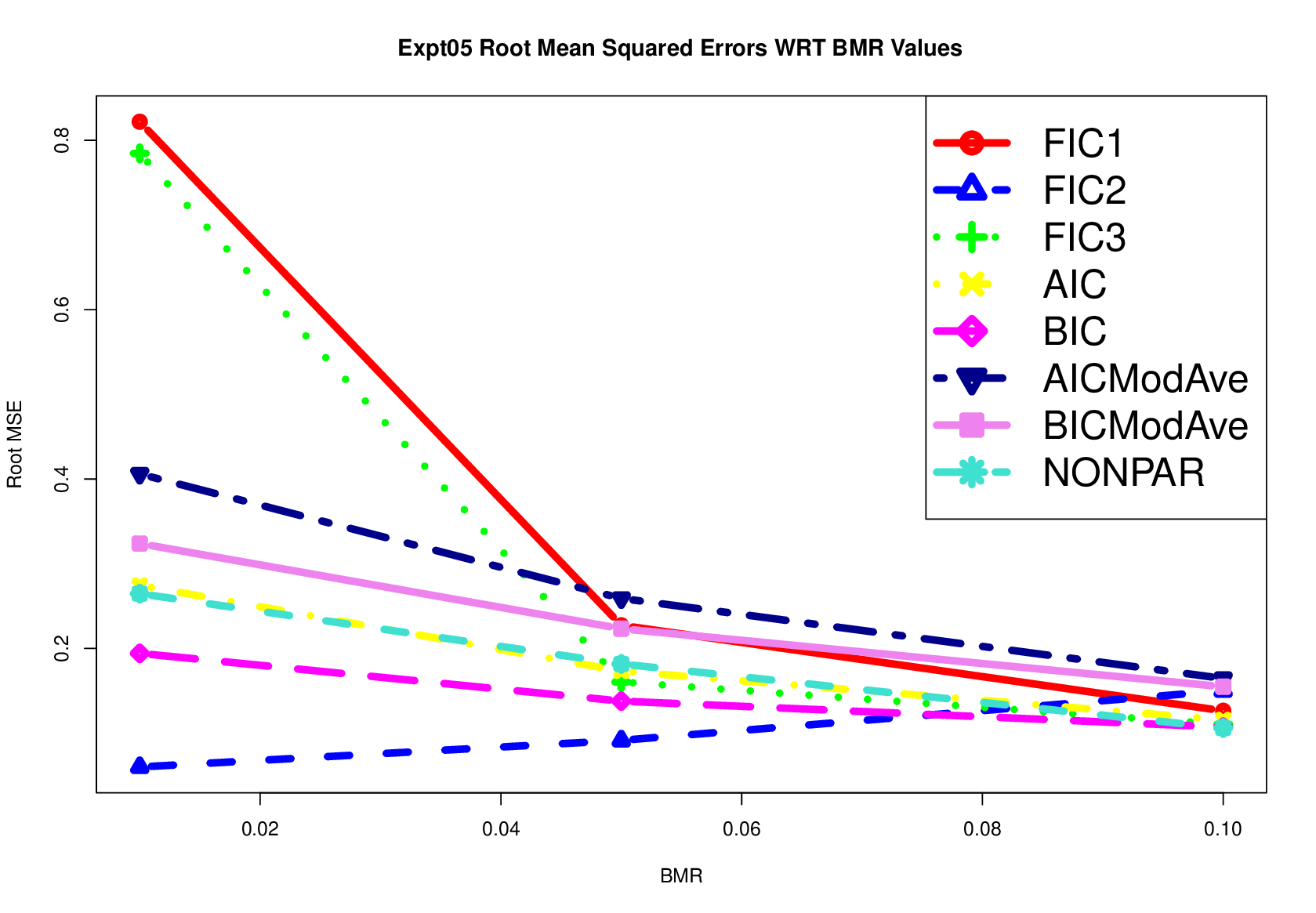} &
\includegraphics[width=\mywidth,height=\myheight,angle=-90]{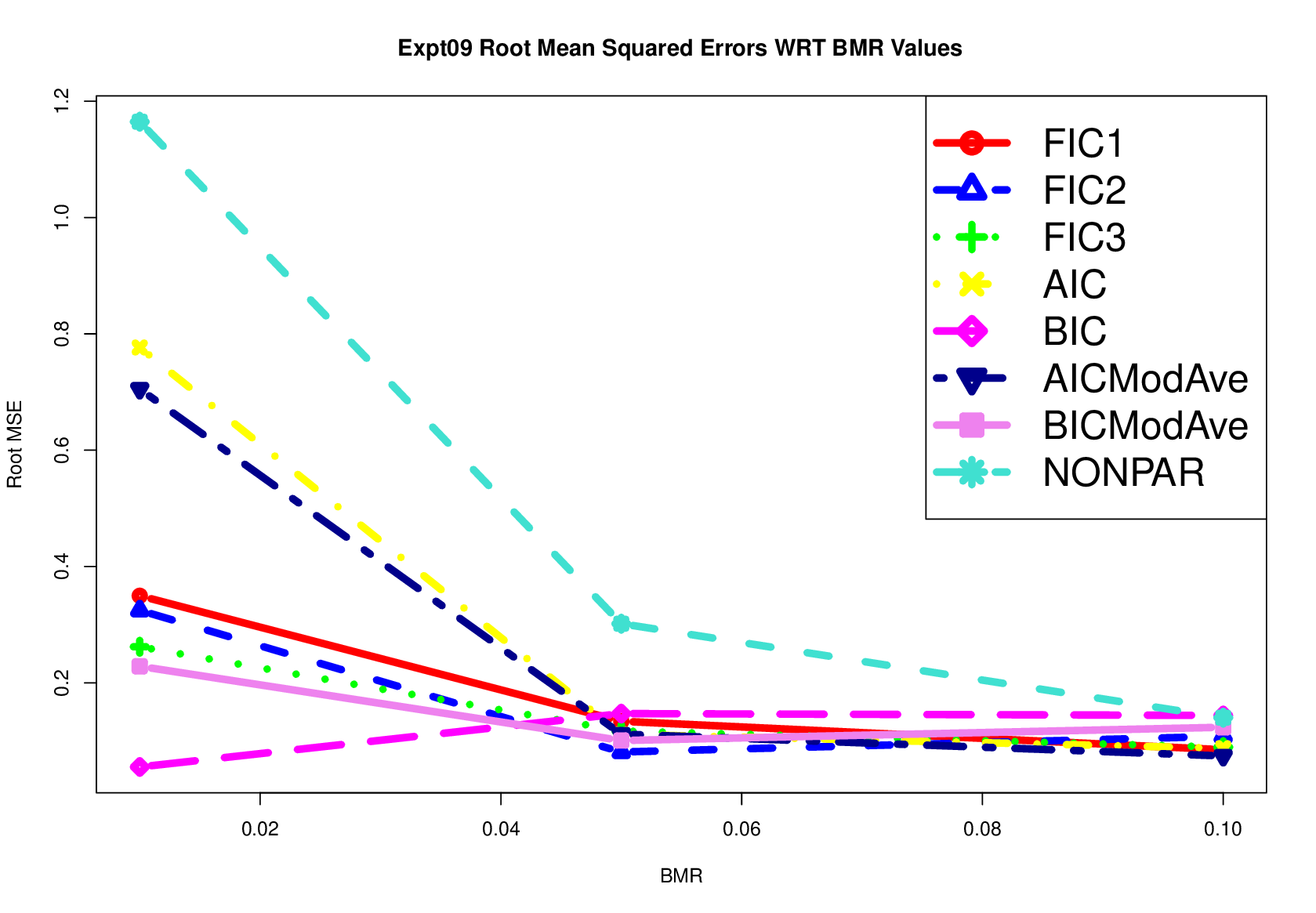} \\
\end{tabular}
\end{center}
\end{figure}

\newcommand{\mysmallwidth}{1.3in}
\newcommand{\mysmallheight}{1.15in}
\begin{figure}
\caption{Comparative bar charts of the five model selectors from table \ref{table: BMD Estimators}. Successive portions from bottom to top in each vertical position in a plot frame depict percentages that model classes LG1, LG2, MS1, and MS2 were selected, with the five vertical positions from left to right associated with selectors FIC1, FIC2, FIC3, AIC, and BIC. Blue bars correspond to the correct models. The two plot frames in each plot panel are for BMR $= 0.01$ (left) and $0.10$ (right). The eight plot panels going from top left to bottom left then top right to bottom right are for Experiments \#2 to \#9.}
\label{fig: model selection for expt 2 to 9}
\begin{center}
\begin{tabular}{|cc|cc|} \hline
\includegraphics[width=\mysmallwidth,height=\mysmallheight,angle=-90]{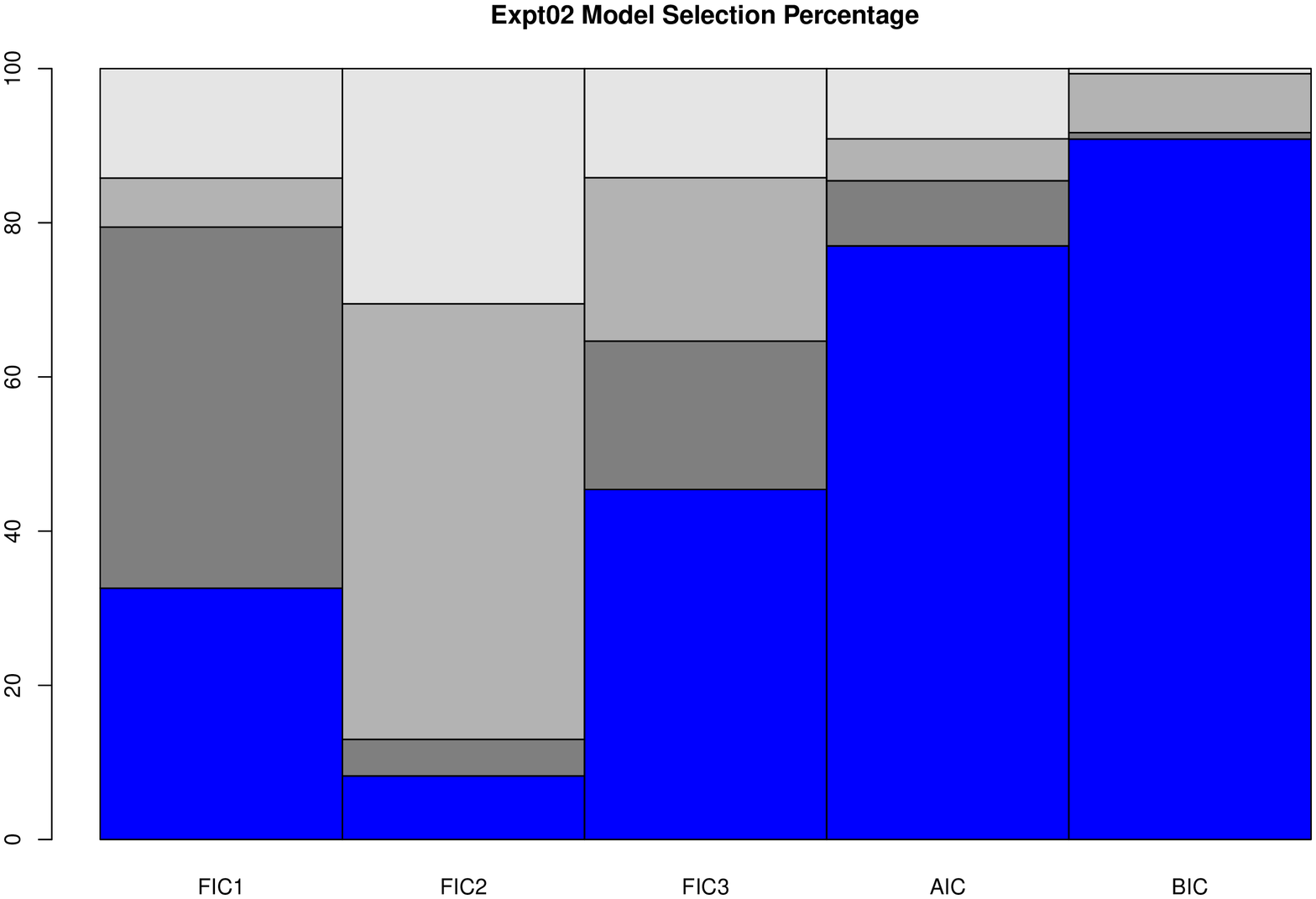} &
\includegraphics[width=\mysmallwidth,height=\mysmallheight,angle=-90]{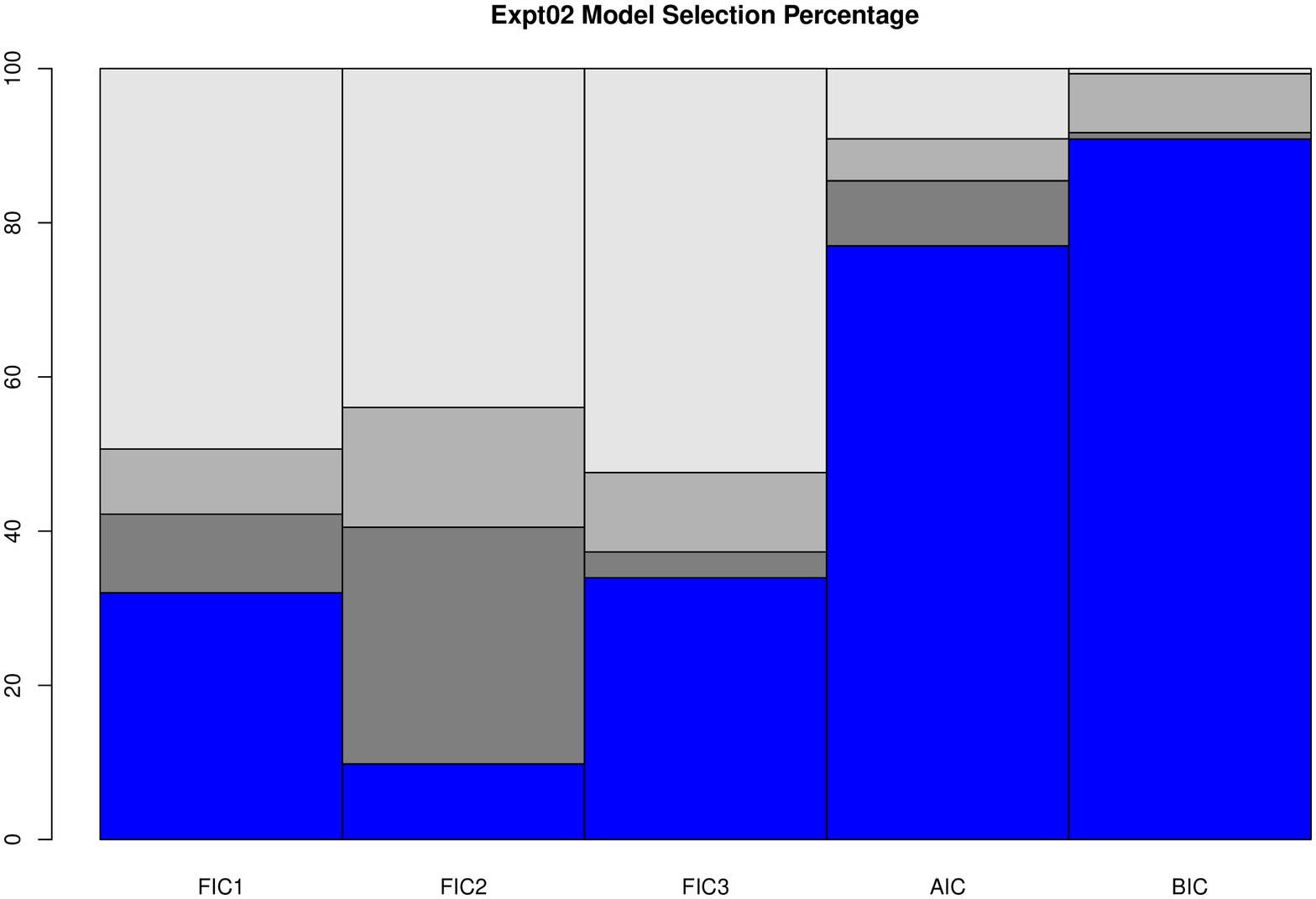} &
\includegraphics[width=\mysmallwidth,height=\mysmallheight,angle=-90]{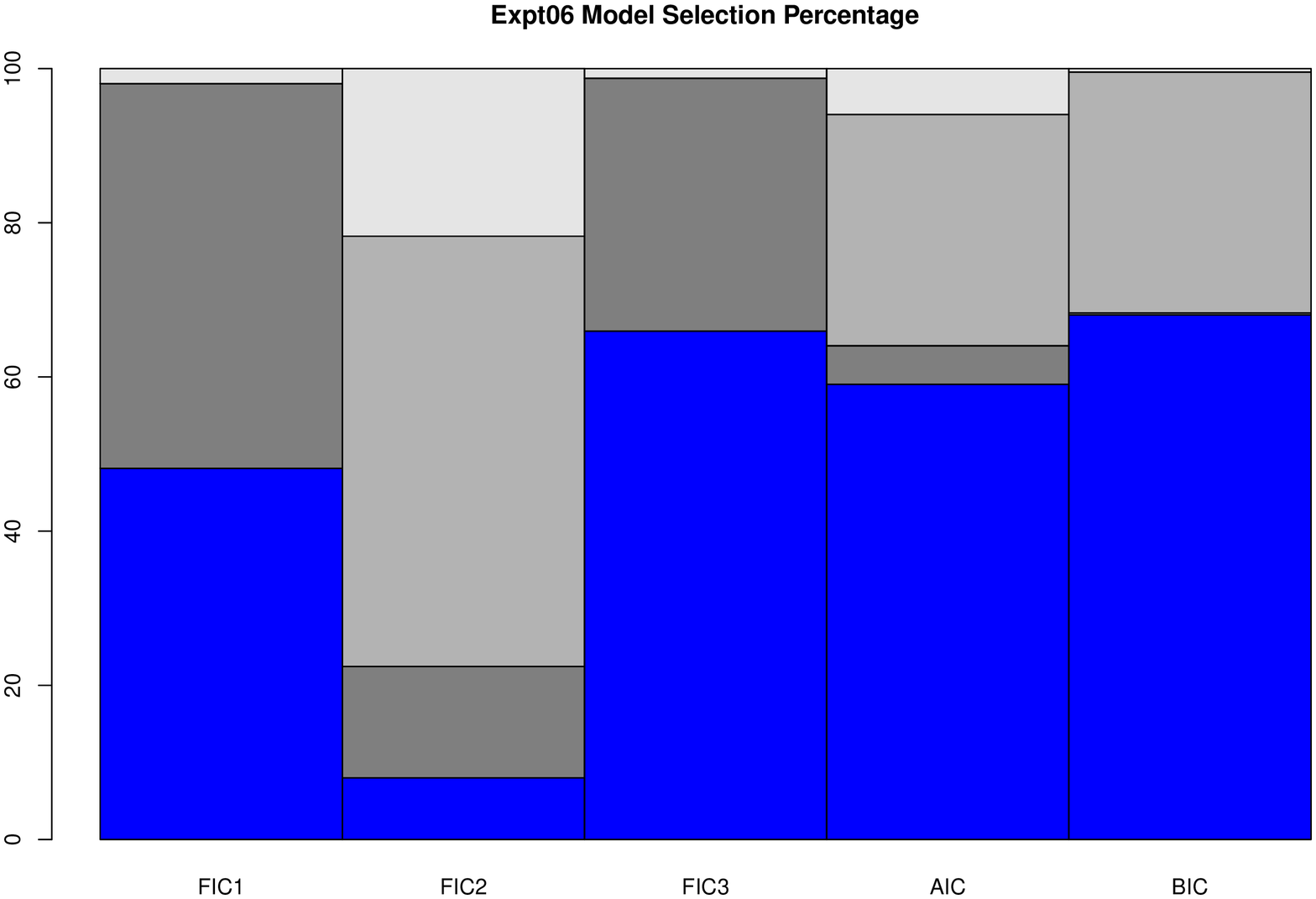} &
\includegraphics[width=\mysmallwidth,height=\mysmallheight,angle=-90]{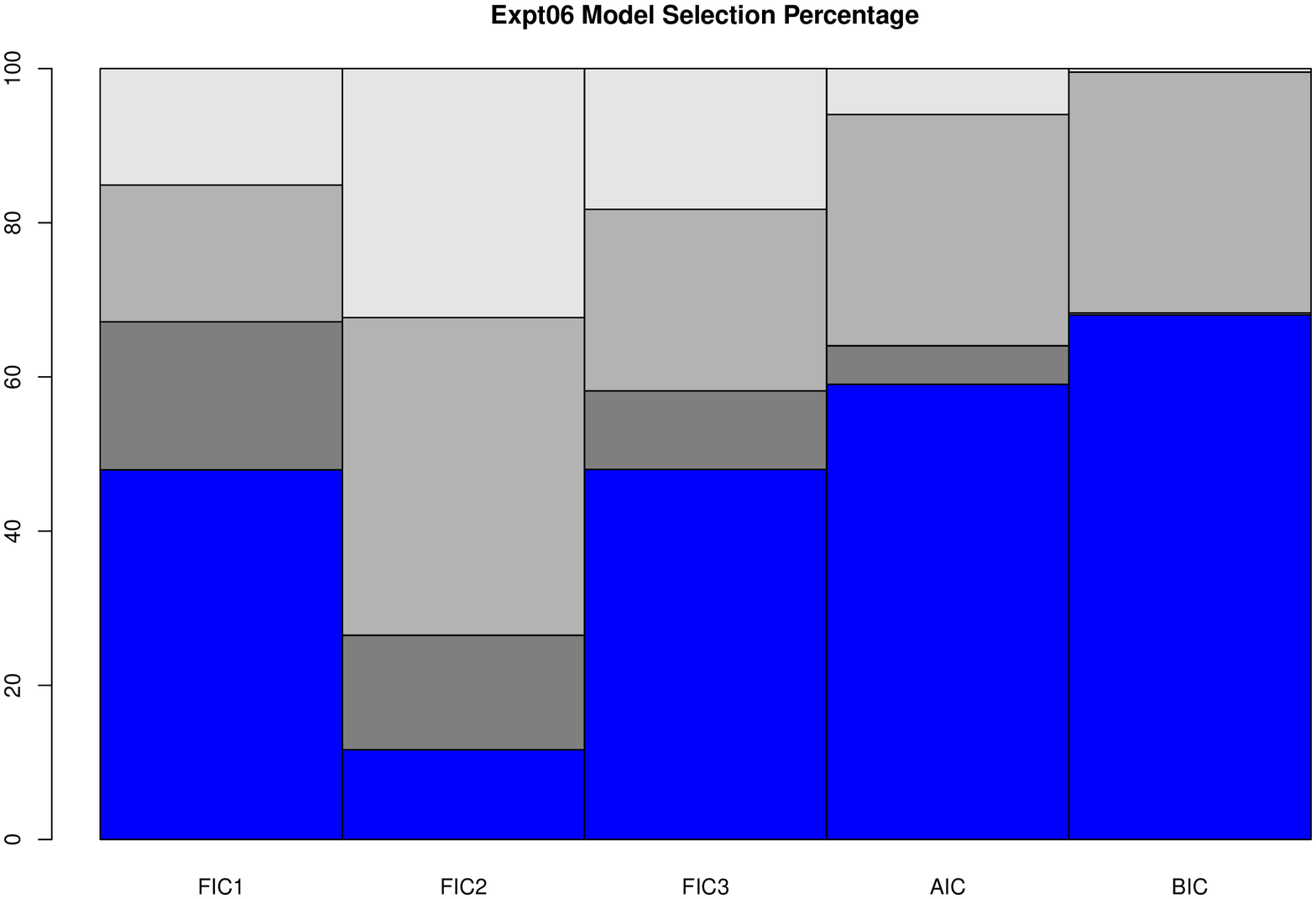} \\ \hline
\includegraphics[width=\mysmallwidth,height=\mysmallheight,angle=-90]{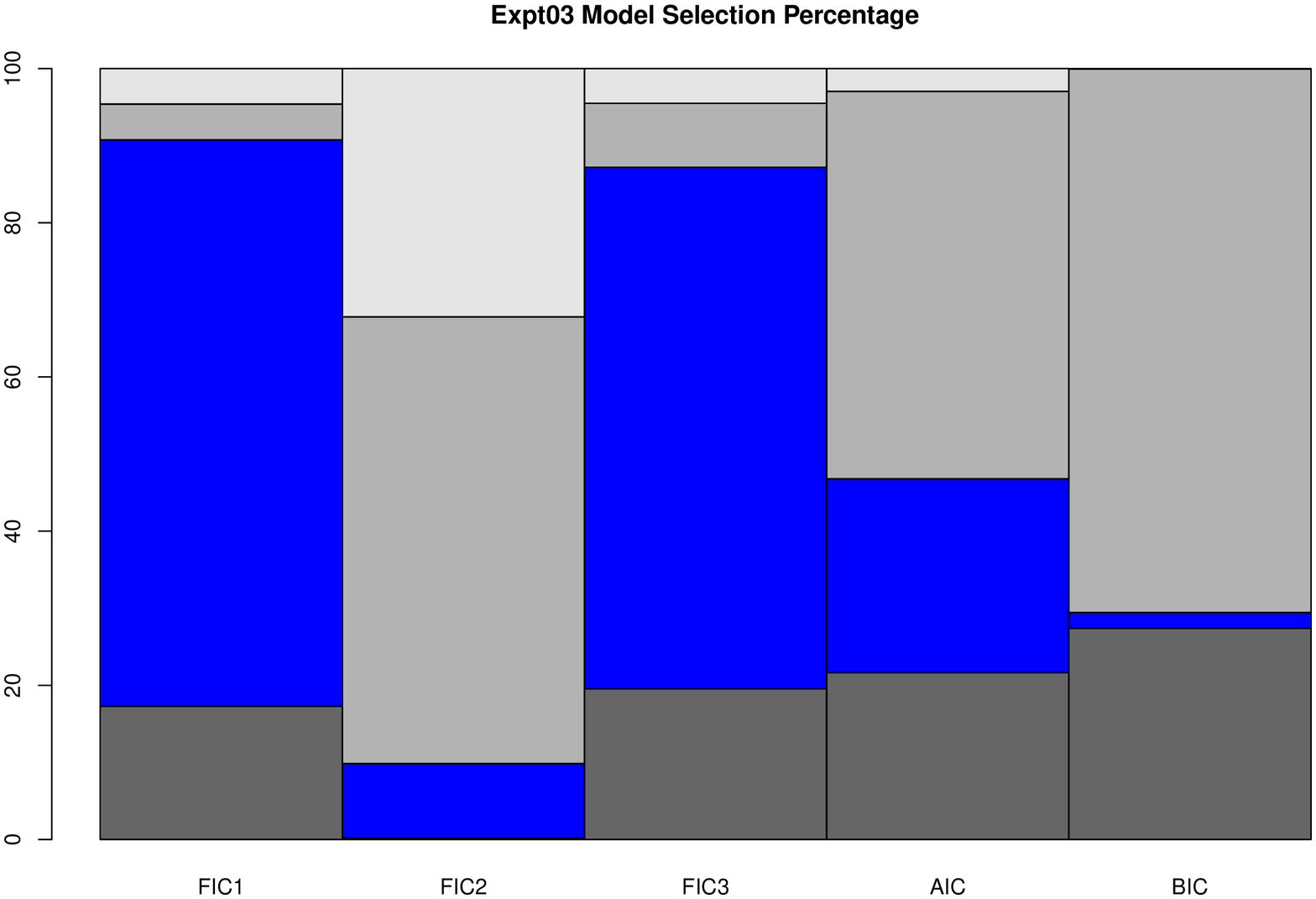} &
\includegraphics[width=\mysmallwidth,height=\mysmallheight,angle=-90]{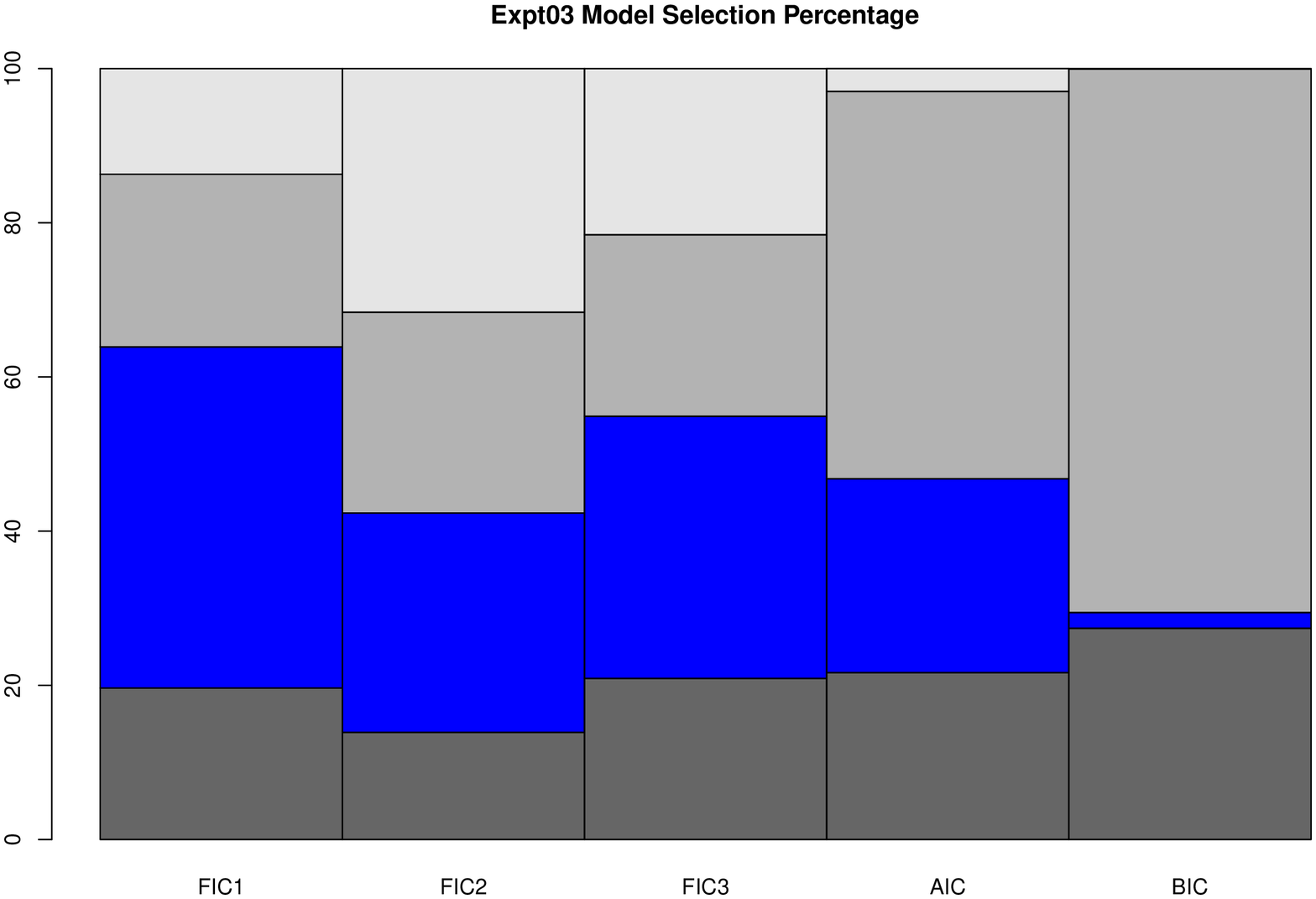} &
\includegraphics[width=\mysmallwidth,height=\mysmallheight,angle=-90]{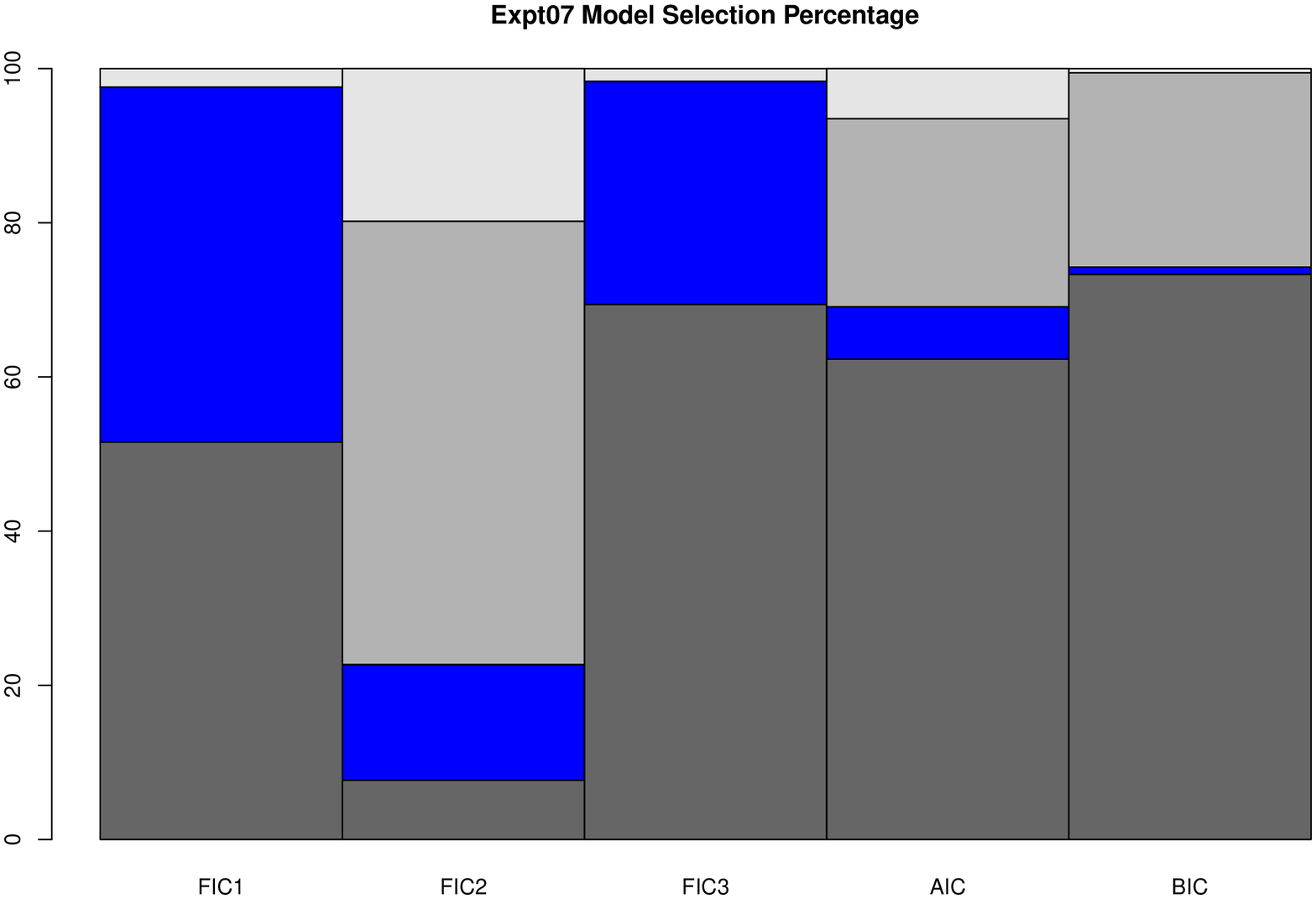} &
\includegraphics[width=\mysmallwidth,height=\mysmallheight,angle=-90]{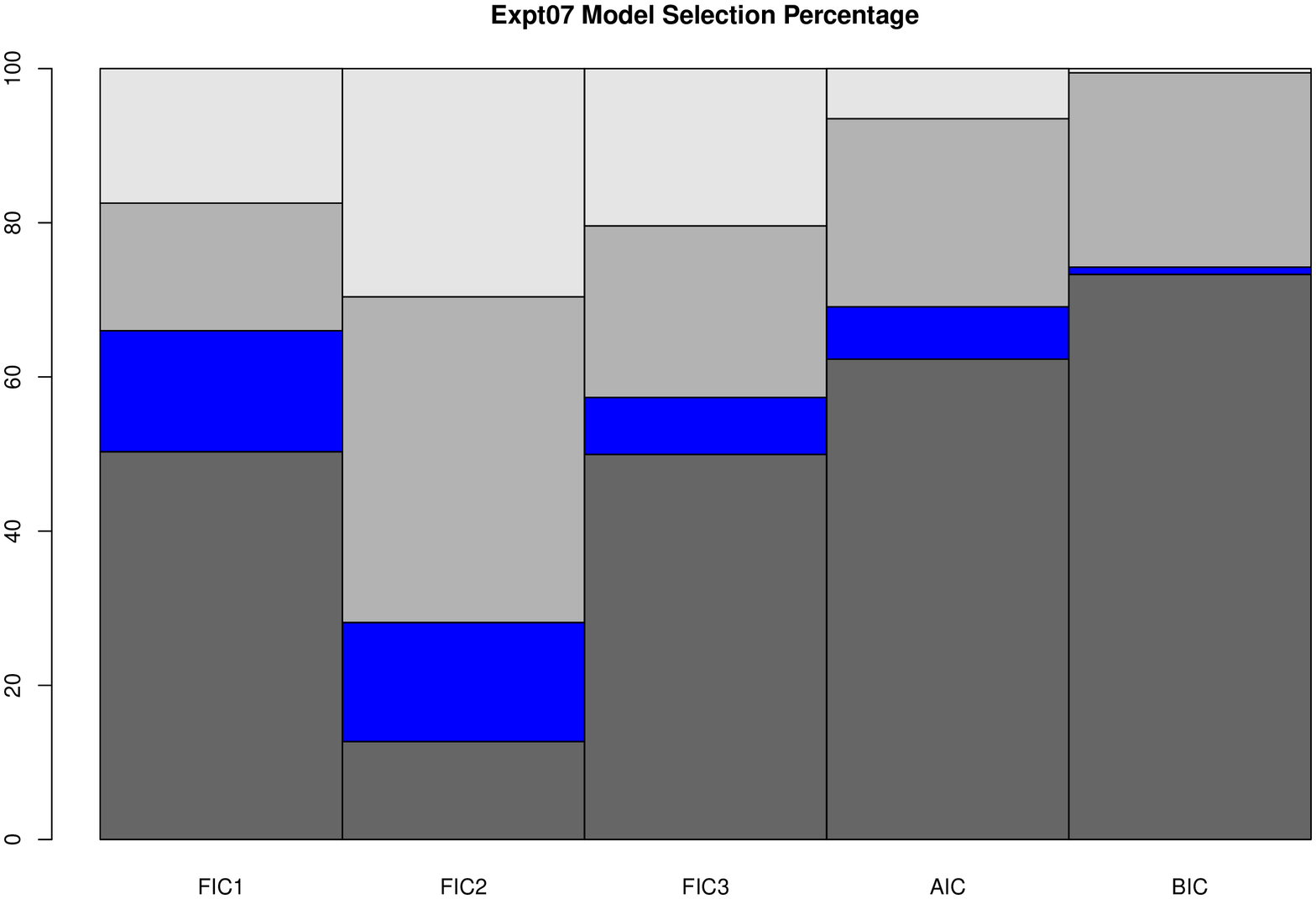} \\ \hline
\includegraphics[width=\mysmallwidth,height=\mysmallheight,angle=-90]{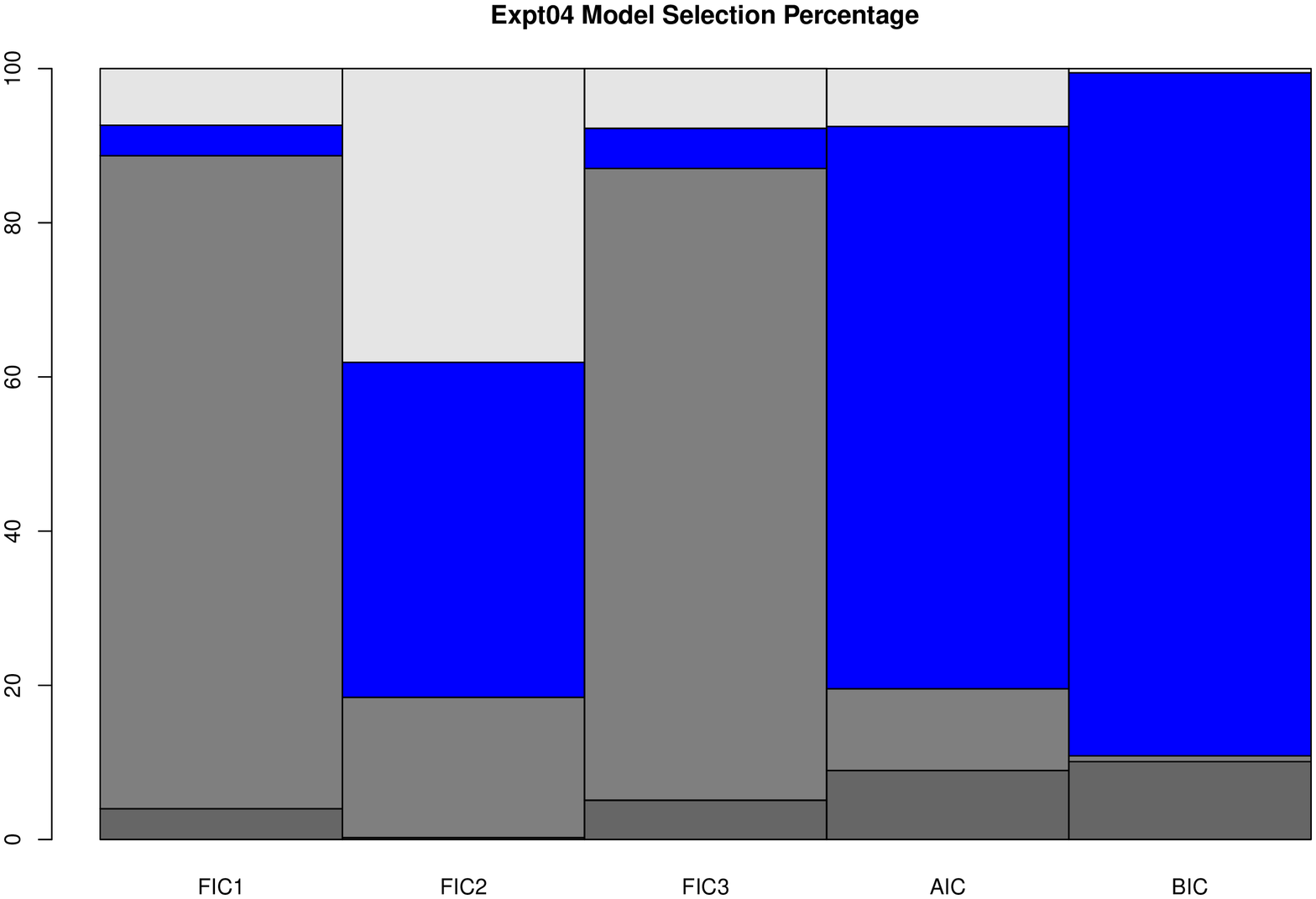} &
\includegraphics[width=\mysmallwidth,height=\mysmallheight,angle=-90]{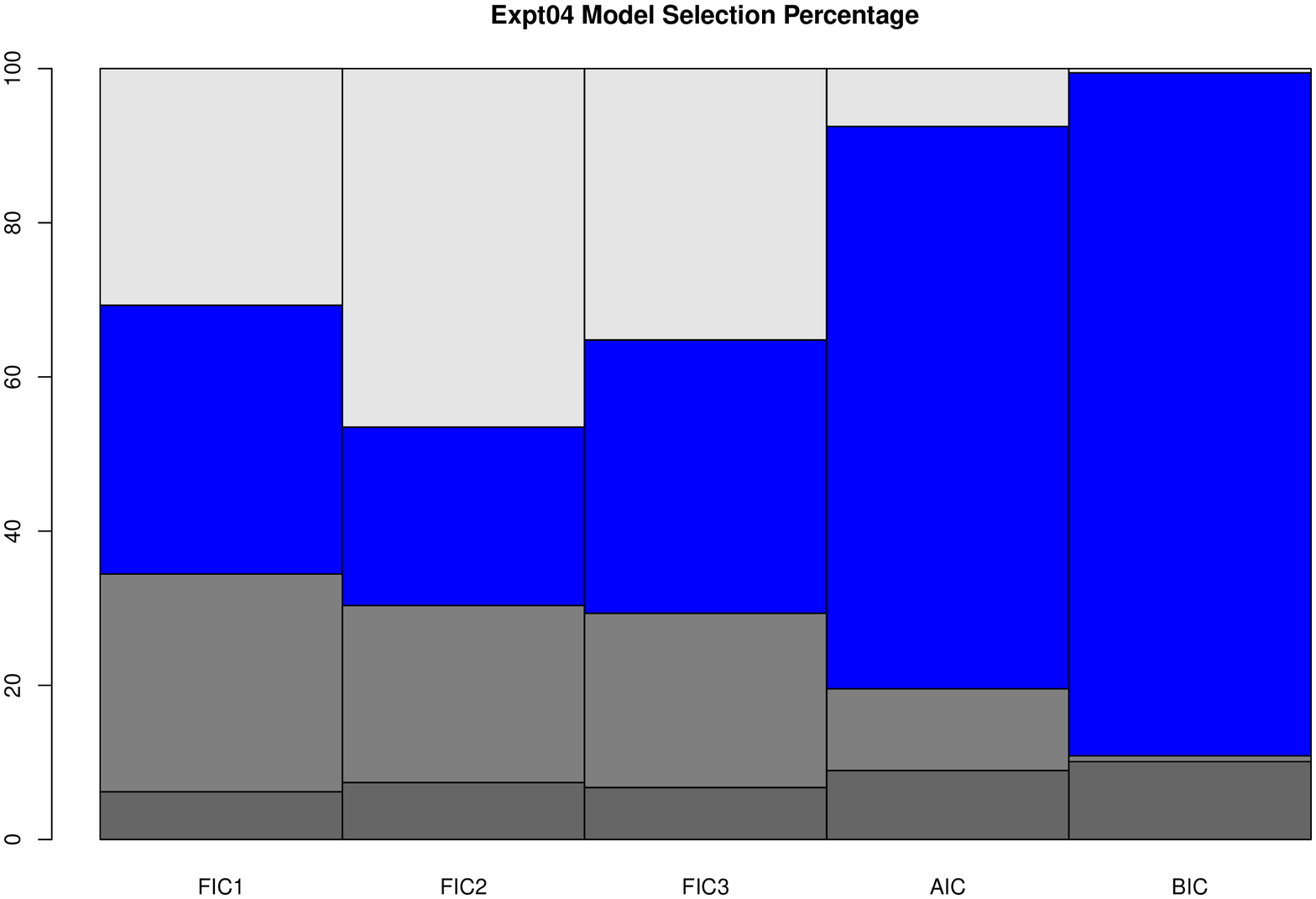} &
\includegraphics[width=\mysmallwidth,height=\mysmallheight,angle=-90]{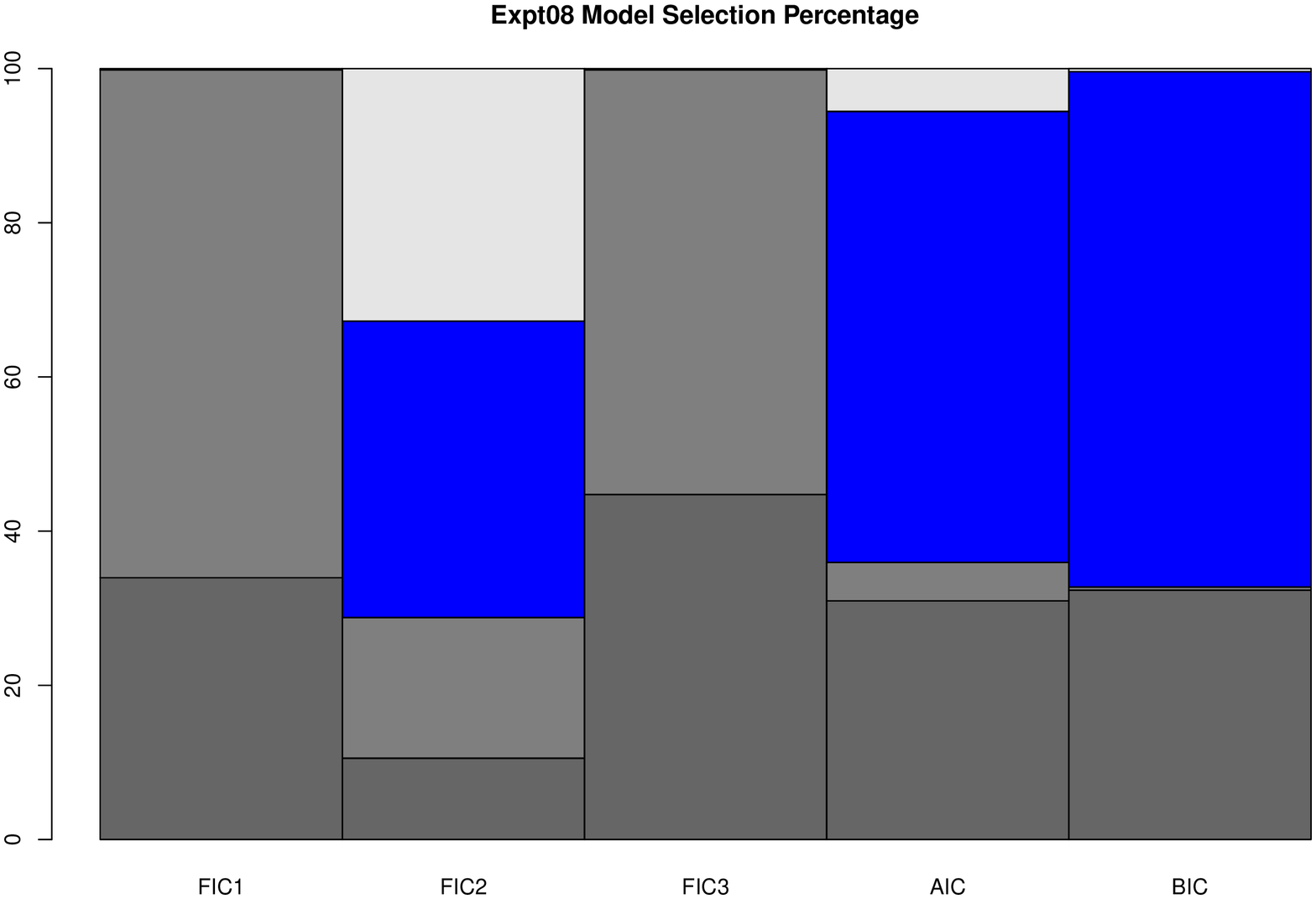} &
\includegraphics[width=\mysmallwidth,height=\mysmallheight,angle=-90]{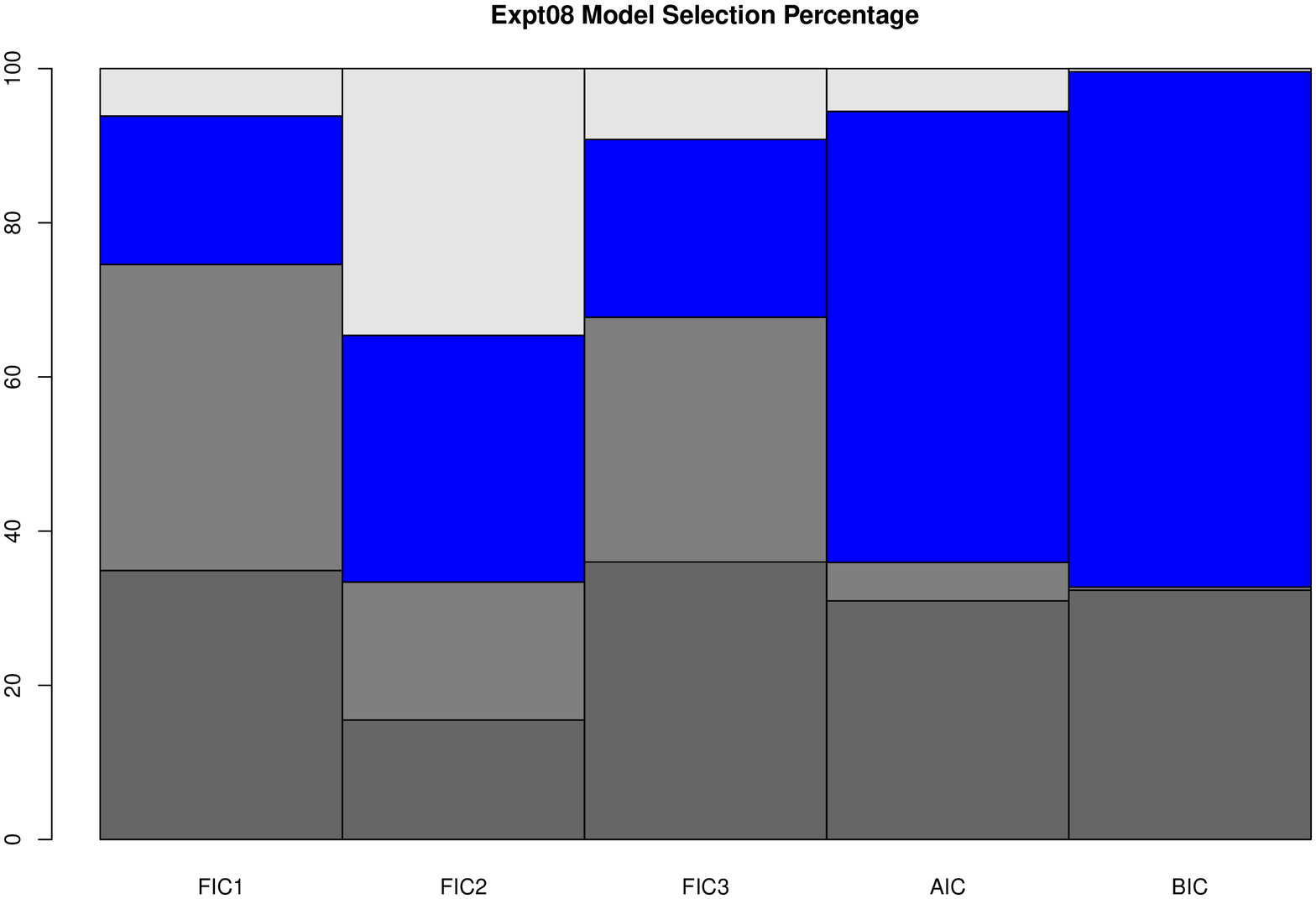} \\ \hline
\includegraphics[width=\mysmallwidth,height=\mysmallheight,angle=-90]{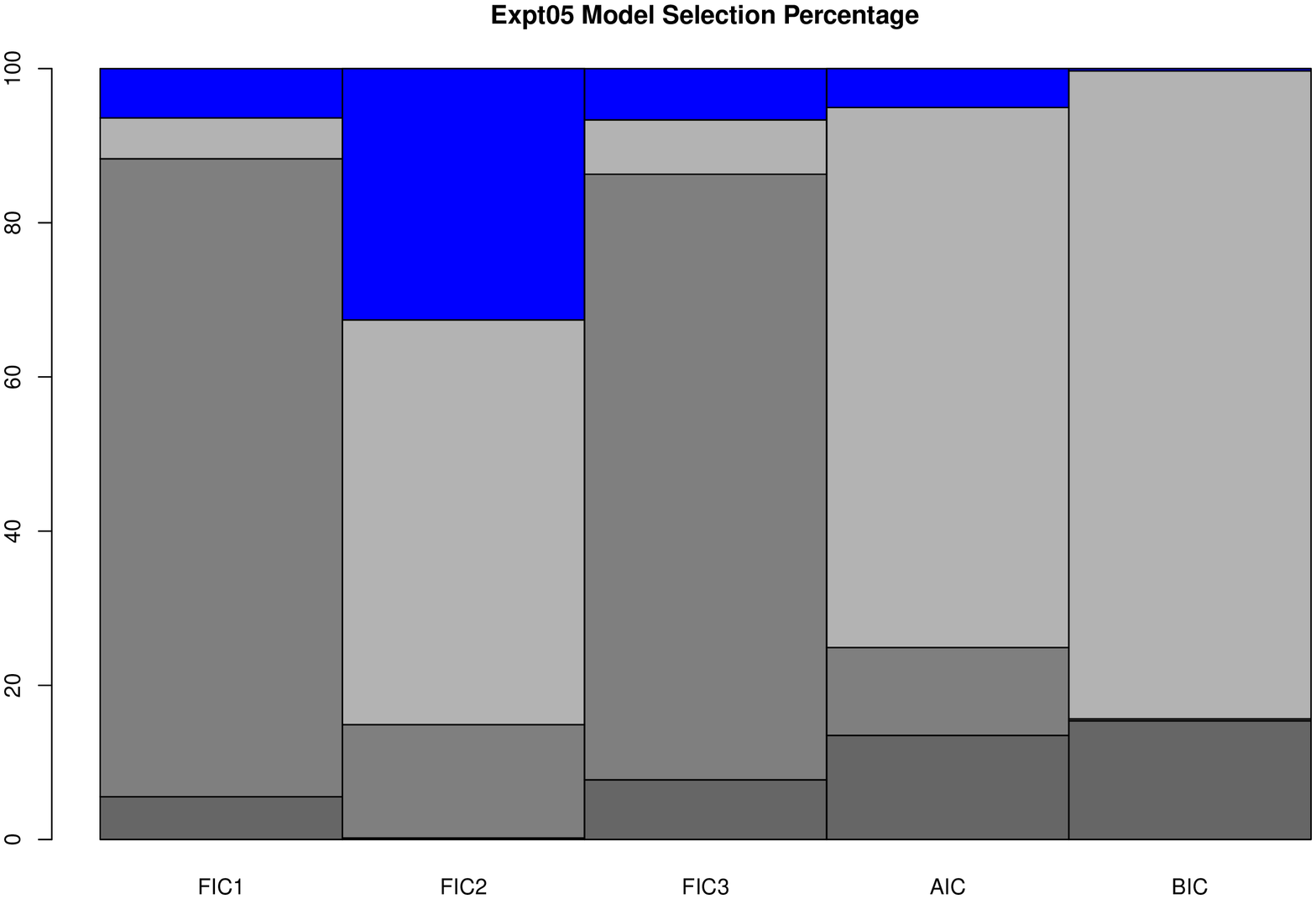} &
\includegraphics[width=\mysmallwidth,height=\mysmallheight,angle=-90]{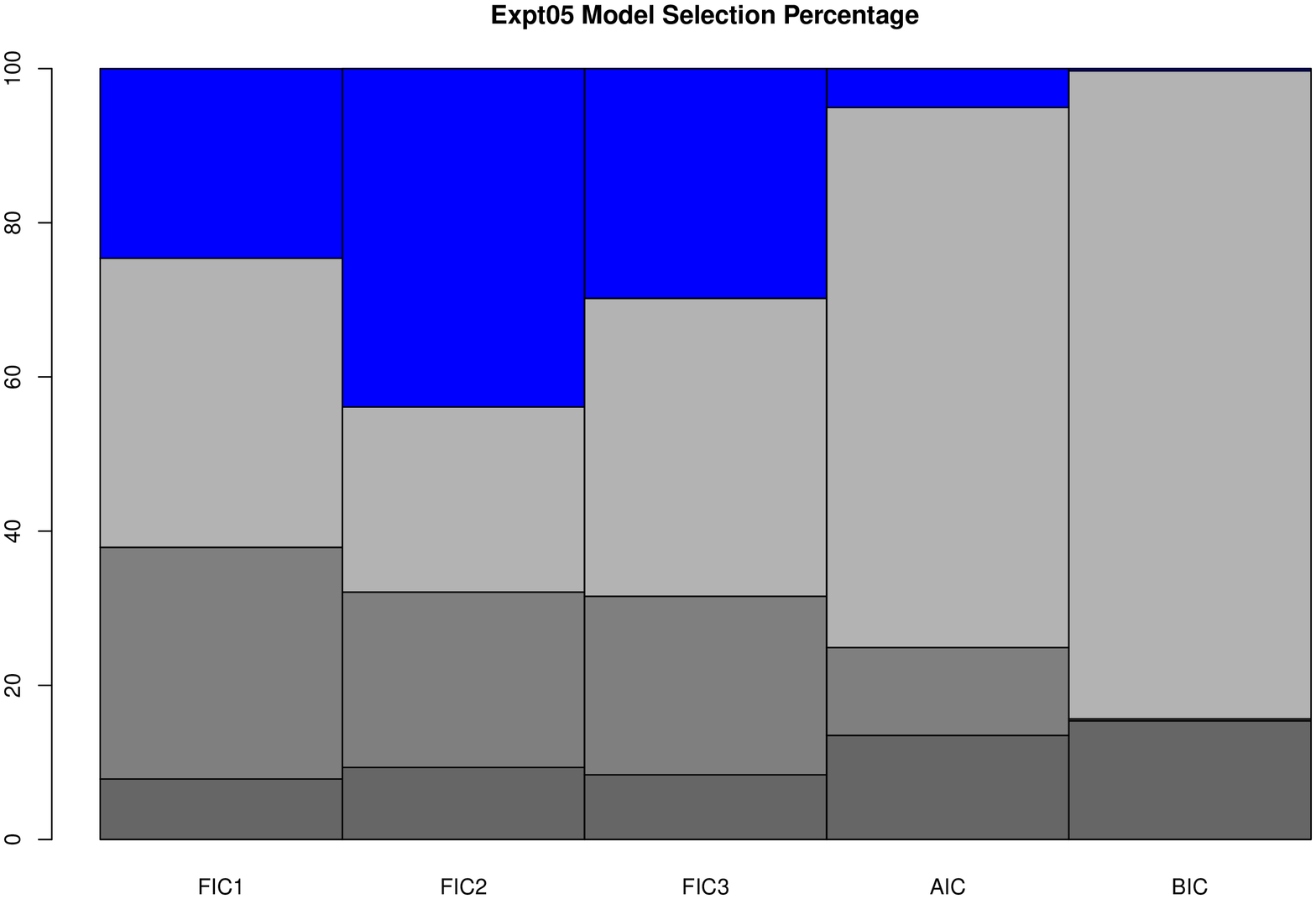} &
\includegraphics[width=\mysmallwidth,height=\mysmallheight,angle=-90]{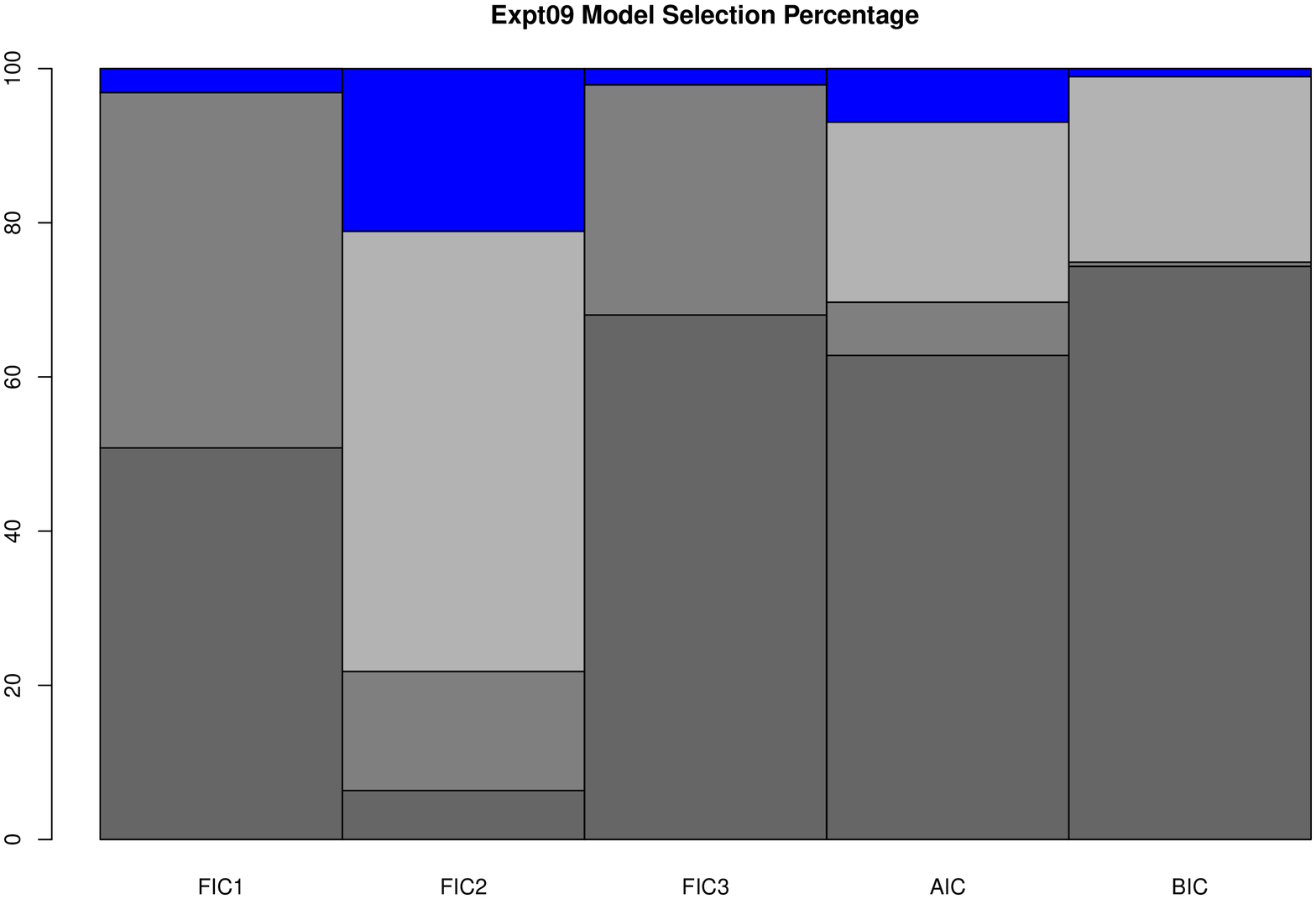} &
\includegraphics[width=\mysmallwidth,height=\mysmallheight,angle=-90]{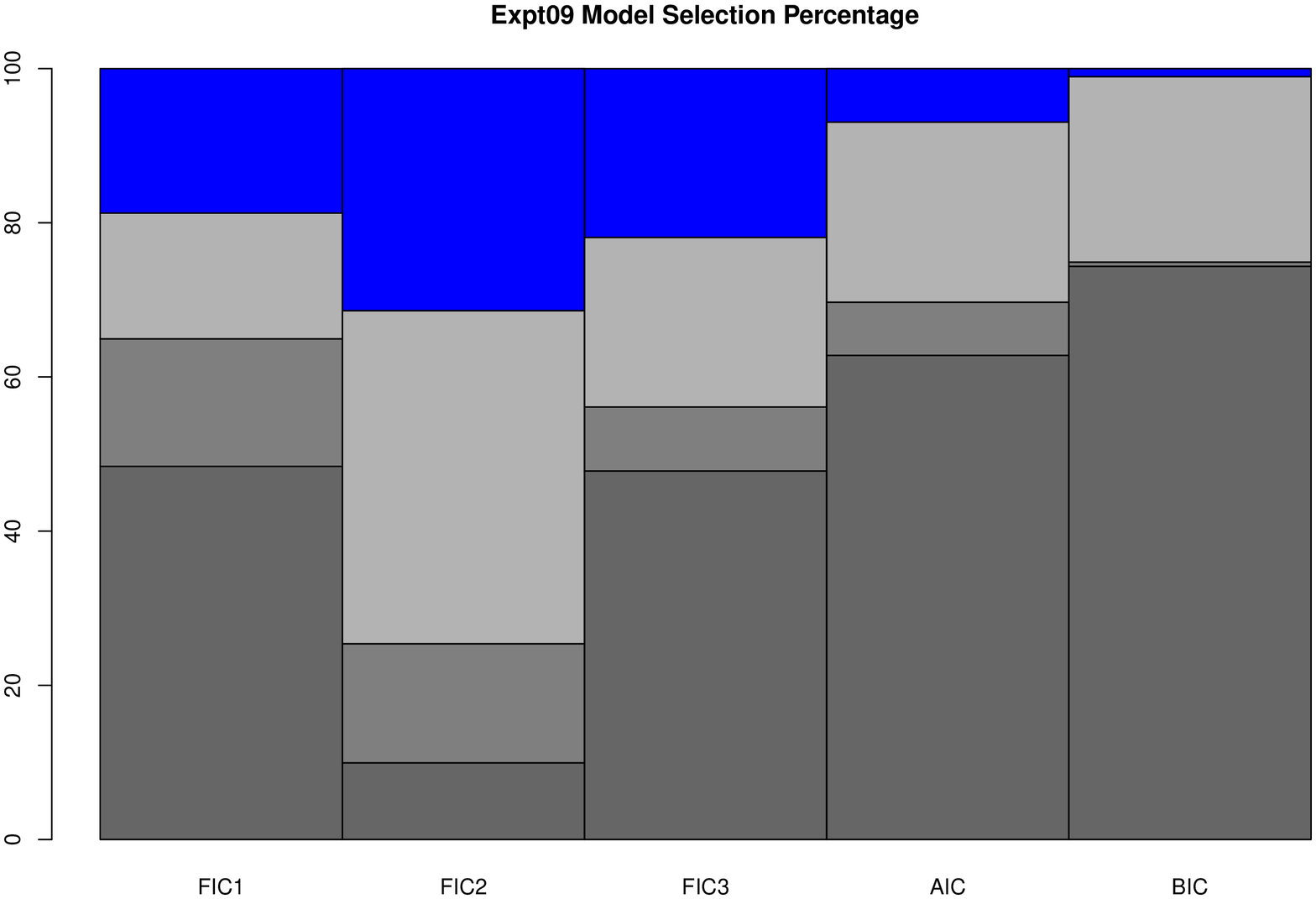}  \\ \hline
\end{tabular}
\end{center}
\end{figure}

Examining Figure \ref{fig: RMSE plots for experiment 2 to 9} and Figure \ref{fig: model selection for expt 2 to 9}, we find that the eight BMD estimators and five model selectors exhibit varied performance among the eight simulation experiments. None of these estimators and model selectors could be said to totally dominate the others. The BIC-based estimators appeared to possess the most stable performance over the three BMR values. In essence, which estimator performs best depended to a large extent on the shape of the true dose-response function. Generally, the FIC2 and BIC-based estimators tended to perform similarly, whereas the FIC1, FIC3, AIC-based, and the nonparametric estimators tended to have comparable RMSE plot patterns. It is also interesting to note that the BIC-model averaged estimators did not always dominate the BIC-two step estimator, and the same could be said for the AIC-based estimators, though for the latter the RMSEs tended to be closer. Surprisingly, the FIC2 estimator appeared to perform well when the true model class is multistage, and poorly when the true model class is logistic. This could be because the FIC2 model selector tended to choose the multistage model class (cf.~Figure \ref{fig: model selection for expt 2 to 9}). Also, we note that the AIC and BIC model selectors tended to choose lower-order models,
as expected \citep{NeaCav12},
though not necessarily lower-order models {\em{in}} the true model class. More importantly, we call attention to the fact that the BIC model selector hardly ever chose the correct model class when the true model class is multistage of order 2 (MS2), but at the same time its associated BMD estimators performed quite well in terms of RMSE!
This seems to indicate that, perhaps, in the
context of estimating relevant parametric functionals when there are several competing model classes, it is
not {\em so} crucial that the model selector involved in two-step procedures be able to choose the correct model,
but rather that its associated parameter estimator perform well in estimating the
functional of interest.

\section{Concluding Remarks}
\label{sec:Conclusions}

This paper has provided several strategies for estimating the benchmark dose (BMD)
in quantal-response studies when the dose-response function may
be thought to
belong to several
competing model classes. Two-stage type procedures, wherein a model is first chosen and
then an estimate is obtained within the chosen model, are described as arising from the
most common
information measures, the AIC and BIC, and also from a focused-inference
approach which relies centrally on the Kullback-Leibler divergence. Model-averaging
type procedures are also described, which are characterized by combining estimates
from the different model classes according to appropriate data-dependent weights.

The model selection procedures and BMD estimators are illustrated using a
carcinogenecity data set.
Through simulation studies, the performance of the different model selectors
and BMD estimators are compared. A nonparametric BMD estimator based on an
empirical estimator of the dose-response function obtained by applying the
pooled-adjacent-violators algorithm (PAVA) on the empirical probabilities at
each dose level \citep{PieXioBha12} was also included among the estimators
that were compared. An interesting phenomenon is that with two-step procedures,
in order for the BMD estimator to perform well in terms of RMSE, it
does not appear
imperative for the associated model selector to be able to choose with high
probability the true generating model. This was particularly evident with the
BIC-based procedures where the BIC model selector did not perform well with
respect to
choosing the true
generating model, but its associated
BMD estimators, both in two-step and model-averaged versions,
exhibited competitive RMSEs.
Of course it should be recognized that the limited simulation study
we have performed is
insufficient to make truly definitive conclusions. Clearly, further examinations and
comparisons of these different model selectors and BMD estimators are warranted
to obtain more definitive conclusions, especially when there are more than two model types.

Finally, we
emphasize further that extreme caution is called-for when assessing the properties of
estimates, where `double-dipping' of the data leads to inferential instabilities.
In particular, additional studies
will be needed to ascertain the impact of model selection on the
distributional properties.
For instance,
the standard errors of the resulting BMD estimators,
or small-sample coverage of any confidence regions based on these estimators,
will be important to determine;
see, for instance, the recent papers by \cite{WesPiePenAnWuEtAl12}
and \cite{PieAnWicWesPenWu13}. Such results will have important bearing in the
construction of
statistical inferences on
the BMD.

\section*{Acknowledgements}

We acknowledge the following research grants, which partially supported this research: EPA Grant RD-83241902, NSF Grants
DMS 0805809 and DMS 1106435, and NIH Grants 2 P20 RR17698, R01 CA154731, R21 ES016791, and  1 P30 GM103336-01A1.



\vspace*{-8pt}

\bibliography{DoseResponse}

\begin{thebibliography}{10}

\bibitem{Aka73}
H.~Akaike.
\newblock Information theory and an extension of the maximum likelihood
  principle.
\newblock In {\em Second International Symposium on Information Theory
  (Tsahkadsor, 1971)}, pages 267--281. Akad\'emiai Kiad\'o, Budapest, 1973.

\bibitem{BaiNobWhe05}
A.~John Bailer, R.~B. Noble, and M.~W. Wheeler.
\newblock Model uncertainty and risk estimation for experimental studies of
  quantal responses.
\newblock {\em Risk Analysis}, 25:291--299, 2005.

\bibitem{BucPieWes09}
Brooke~E. Buckley, Walter~W. Piegorsch, and R.~Webster West.
\newblock Confidence limits on one-stage model parameters in benchmark risk
  assessment.
\newblock {\em Environmental and Ecological Statistics}, 16:53--62, 2009.

\bibitem{BurAnd02}
Kenneth~P. Burnham and David~R. Anderson.
\newblock {\em Model selection and multimodel inference: a practical
  information-theoretic approach}.
\newblock Springer-Verlag, New York, second edition, 2002.

\bibitem{MR1707178}
Joseph~E. Cavanaugh.
\newblock A large-sample model selection criterion based on {K}ullback's
  symmetric divergence.
\newblock {\em Statist. Probab. Lett.}, 42(4):333--343, 1999.

\bibitem{ClaHjo03}
Gerda Claeskens and Nils~Lid Hjort.
\newblock The focused information criterion.
\newblock {\em J. Amer. Statist. Assoc.}, 98(464):900--945, 2003.
\newblock With discussions and a rejoinder by the authors.

\bibitem{ClaHjo08}
Gerda Claeskens and Nils~Lid Hjort.
\newblock {\em Model Selection and Model Averaging}.
\newblock Cambridge Series in Statistical and Probabilistic Mathematics.
  Cambridge University Press, Cambridge, 2008.

\bibitem{DeuGreHab10}
R.~C Deutsch, J.~M. Grego, B.~T. Habing, and Walter~W. Piegorsch.
\newblock Maximum likelihood estimation with binary-data regression models:
  small-sample and large-sample features.
\newblock {\em Advances and Applications in Statistics}, 14(2):101--116, 2010.

\bibitem{DukPen05}
Vanja~M. Duki{\'c} and Edsel~A. Pe{\~n}a.
\newblock Variance estimation in a model with {G}aussian submodels.
\newblock {\em J. Amer. Statist. Assoc.}, 100(469):296--309, 2005.

\bibitem{FaeAerGey07}
C.~Faes, M.~Aerts, H.~Geys, and G.~Molenberghs.
\newblock Model averaging using fractional polynomials to estimate a safe level
  of exposure.
\newblock {\em Risk Analysis}, 27:111--123, 2007.

\bibitem{HjoCla03}
Nils~Lid Hjort and Gerda Claeskens.
\newblock Frequentist model average estimators.
\newblock {\em J. Amer. Statist. Assoc.}, 98(464):879--899, 2003.

\bibitem{BMA99}
J.~Hoeting, D.~Madigan, A.~Raftery, and C.~Volinsky.
\newblock Bayesian model averaging.
\newblock {\em Statistical Science}, 14:382--401, 1999.

\bibitem{MR1016020}
Clifford~M. Hurvich and Chih-Ling Tsai.
\newblock Regression and time series model selection in small samples.
\newblock {\em Biometrika}, 76(2):297--307, 1989.

\bibitem{HwaYooKim09}
M.~Hwang, E.~Yoon, J.~Kim, D.~D. Jang, and T.~M. Yoo.
\newblock Toxicity value for 3-monochloropropane-1,2-diol using a benchmark
  dose methodology.
\newblock {\em Regulatory Toxicology and Pharmacology}, 53:102--106, 2009.

\bibitem{KusLasDrew75}
M.~Kuschner, S.~Laskin, R.T. Drew, V.~Cappiello, and N.~Nelson.
\newblock Inhalation carcinogenicity of alpha halo ethers. iii. lifetime and
  limited period inhalation studies with bis(chloromethyl)ether at 0.1 ppm.
\newblock {\em Arch. Environ. Health}, 30(2):73--77, 1975.

\bibitem{LehCas98}
E.~L. Lehmann and George Casella.
\newblock {\em Theory of point estimation}.
\newblock Springer Texts in Statistics. Springer-Verlag, New York, second
  edition, 1998.

\bibitem{MR1718844}
Michael~A. Messig and William~E. Strawderman.
\newblock The asymptotic behaviour of {B}ayes estimators for dichotomous
  quantal response models.
\newblock {\em Sankhy\=a Ser. A}, 60(3):418--425, 1998.

\bibitem{MorIbrChe06}
Knashawn~H. Morales, Joseph~G. Ibrahim, Chien-Jen Chen, and Louise~M. Ryan.
\newblock Bayesian model averaging with applications to benchmark dose
  estimation for arsenic in drinking water.
\newblock {\em J. Amer. Statist. Assoc.}, 101(473):9--17, 2006.

\bibitem{Mor92}
B.~J.~T. Morgan.
\newblock {\em Analysis of Quantal Response Data}.
\newblock Chapman \& Hall, New York, 1992.

\bibitem{NeaCav12}
A.~A. Neath and J.~E. Cavanaugh.
\newblock The {B}ayesian information criterion: background, derivation, and
  applications.
\newblock {\em Wiley Interdisciplinary Reviews: Computational Statistics},
  4(2):199--203, 2012.

\bibitem{NitPieWes07}
Daniela Nitcheva, Walter~W. Piegorsch, and R.~Webster West.
\newblock On use of the multistage dose-response model for assessing laboratory
  animal carcinogenicity.
\newblock {\em Regul. Toxicol. Pharmacol.}, 48:135--147, 2007.

\bibitem{PieAnWicWesPenWu13}
Walter~W. Piegorsch, Lingling An, Alissa~A. Wickens, R.~Webster~West, Edsel~A.
  Pe{\~n}a, and Wensong Wu.
\newblock Information-theoretic model-averaged benchmark dose analysis in
  environmental risk assessment.
\newblock {\em Environmetrics}, 24(3):143--157, 2013.

\bibitem{PieBai05}
Walter~W. Piegorsch and A.~John Bailer.
\newblock {\em Analyzing Environmental Data}.
\newblock John Wiley \& Sons, Chichester, 2005.

\bibitem{PieXioBha12}
Walter~W. Piegorsch, Hui Xiong, Rabi~N. Bhattacharya, and Lizhen Lin.
\newblock Nonparametric estimation of benchmark doses in environmental risk
  assessment.
\newblock {\em Environmetrics}, 23(8):717--728, 2012.

\bibitem{Port94}
C.~J. Portier.
\newblock Biostatistical issues in the design and analysis of animal
  carcinogenicity experiments.
\newblock {\em Environmental Health Perspectives}, 102, Suppl. 1:5--8, 1994.

\bibitem{R11}
{{R} {D}evelopment {C}ore {T}eam}.
\newblock {\em {R}: {A} Language and Environment for Statistical Computing}.
\newblock {R} Foundation for Statistical Computing, Vienna, Austria, 2011.

\bibitem{RobWriDyk88}
Tim Robertson, F.~T. Wright, and R.~L. Dykstra.
\newblock {\em Order restricted statistical inference}.
\newblock Wiley Series in Probability and Mathematical Statistics: Probability
  and Mathematical Statistics. John Wiley \& Sons Ltd., Chichester, 1988.

\bibitem{SanFalVic02}
S.~Sand, A.~FalkFilipsson, and K.~Victorin.
\newblock Evaluation of the benchmark dose method for dichotomous data: Model
  dependence and model selection.
\newblock {\em Regulatory Toxicology and Pharmacology}, 36:184--197, 2002.

\bibitem{SanVicFal08}
S.~Sand, K.~Victorin, and A.~FalkFilipsson.
\newblock The current state of knowledge on the use of the benchmark dose
  concept in risk assessment.
\newblock {\em Journal of Applied Toxicology}, 28:405--421, 2008.

\bibitem{Sch78}
G.~Schwartz.
\newblock Estimating the dimension of a model.
\newblock {\em Ann.\ Statist.}, 6:461--464, 1978.

\bibitem{ShaSma11}
K.~Shao and M.~J. Small.
\newblock Potential uncertainty reduction in model-averaged benchmark dose
  estimates informed by an additional dose study.
\newblock {\em Risk Analysis}, 31:1561--1575, 2011.

\bibitem{Tak76}
K.~Takeuchi.
\newblock Distribution of informational statistics and a criterion of model
  fitting (in {J}apanese).
\newblock {\em Suri-Kagaku (Mathematical Sciences)}, 153:12--18, 1976.

\bibitem{USEPA12}
{U.S. EPA}.
\newblock {\em Benchmark dose technical guidance document: Technical Report No.
  EPA/100/R-12/001}.
\newblock U.S. Environmental Protection Agency, Washington, DC, 2012.

\bibitem{WesPiePenAnWuEtAl12}
R.~Webster West, Walter~W. Piegorsch, Edsel~A. Pe{\~n}a, Lingling An, Wensong
  Wu, Alissa~A. Wickens, Hui Xiong, and Wenhai Chen.
\newblock The impact of model uncertainty on benchmark dose estimation.
\newblock {\em Environmetrics}, 23(8):706--716, 2012.

\bibitem{WheBai07}
M.~W. Wheeler and A.~John Bailer.
\newblock Properties of model-averaged {BMDL}s: A study of model averaging in
  dichotomous response risk estimation.
\newblock {\em Risk Analysis}, 27:659--670, 2007.

\bibitem{WheBai09}
M.~W. Wheeler and A.~John Bailer.
\newblock Comparing model averaging with other model selection strategies for
  benchmark dose estimation.
\newblock {\em Environmental and Ecological Statistics}, 16:37--51, 2009.

\bibitem{MR752447}
Arnold Zellner and Peter~E. Rossi.
\newblock Bayesian analysis of dichotomous quantal response models.
\newblock {\em J. Econometrics}, 25(3):365--393, 1984.

\end{thebibliography}
\bibliographystyle{plain}




\end{document}